\documentclass{amsart}

\usepackage[pdf]{pstricks}
\usepackage{pstricks-add}

\usepackage{amssymb}
\usepackage{amscd}
\usepackage{epsfig}
\usepackage{graphicx}
\usepackage{psfrag}
\usepackage{amsmath}
\usepackage{wrapfig}
\usepackage{picins}
\usepackage{comment}

\newtheorem{theorem}{Theorem}[section]
\newtheorem*{theo}{Theorem}

\newtheorem{lemma}[theorem]{Lemma}
\theoremstyle{definition}

\newtheorem{proposition}[theorem]{Proposition}
\newtheorem{corollary}[theorem]{Corollary}
\theoremstyle{remark}
\newtheorem{remark}[theorem]{Remark}

\numberwithin{equation}{section}

\begin{document}

\title[Convexity]{Convexity of length functions and Thurston's shear coordinates} 

\author{Guillaume Th\'eret}
\address{Lyc\'ee Ni\'epce\\ 
 71100 Chalon-sur-Sa\^one, France\\
\&\\
Institut de Math\'ematiques de Bourgogne\\
9 avenue Alain Savary,
21078 Dijon - France}
\curraddr{}
\email{guillaume.theret71@orange.fr}
\thanks{}


\date{\today}

\dedicatory{}

\begin{abstract}
We show that the length function of a measured geodesic lamination is convex in Thurston's shear coordinates over Teichm\"uller space and \emph{strictly} convex for generic laminations. 
We give some consequences of this result in the context of Thurston's asymmetric metric on Teichm\"uller space.
\end{abstract}

\maketitle

\section{Introduction}
\label{intro}

It is shown in \cite{BBFS09} that the length of a simple closed geodesic on a hyperbolic surface $\Sigma$ is a convex function over Teichm\"uller space, the latter being parameterized by a particular type of Thurston's shear coordinates,
namely, those arising from a geodesic lamination $\mu$ which completes a pants decomposition of the surface $\Sigma$. 
Furthermore, \emph{strict} convexity holds for simple closed geodesic curves intersecting all the leaves of $\mu$.

Our goal here is to generalize the above result in two directions.
Firstly, we shall establish that these convexity results hold true for all measured geodesic laminations, and not only for simple closed geodesics.
Secondly, we shall show that the results are also valid for all kinds of shear coordinates put on Teichm\"uller space, i.e., when the lamination $\mu$ is arbitrarily chosen among all complete geodesic laminations, and not only among those that complete pants decompositions.
Because we aim at obtaining a result about \emph{strict} convexity of length functions, we cannot simply use \cite{BBFS09} and a density argument. 

At the end of the paper, we shall give a few consequences in the context of Thurston's asymmetric metric on Teichm\"uller space and stretch lines.\\

Our paper is organized as follows.
In the first section following this introduction, we define the length of a measured lamination and show that the infimum of the lengths of the measured laminations that belong to the same measure class is realized by the geodesic representative
(see also Papadopoulos \cite{Pap91} for a somewhat weaker result). This result is notorious for simple closed curves.
This section gives us the occasion to introduce thick train-track approximations of a geodesic lamination, a notion introduced by Thurston (see also \cite{Bonahon96}) and which is used all along our paper.

In section \ref{section:shear}, we define Thurston's \emph{local} shear coordinates associated to a thick train-track approximation of a complete geodesic lamination.
Follows then a short -- most details are omitted there -- discussion on Thurston's \emph{global} shear coordinates on Teichm\"uller space associated to a complete geodesic lamination (and not merely to a train-track approximation). 
This section is rather lengthy, mainly because we strove in that section to recall and to reconcile notions scattered in literature. 

In section \ref{section:convexity}, we establish convexity of length function of measured geodesic laminations expressed in terms of shear coordinates.

Finally, we give in the last section a few applications of the obtained results to stretch lines.\\

We now give some more details about what is proven in here.

In the whole text, the underlying surface $\Sigma$ is closed, \emph{oriented}, compact and with negative Euler-Poincar\'e characteristic.
Such a surface always admits a hyperbolic structure.
Possibly our results might be extended to the case of finitely punctured surfaces, but we kept saying anything about this case, 
chiefly because it certainly would have tangled the presentation which stretches long enough.

A \textbf{geodesic lamination} of a hyperbolic surface $\Sigma$ is a foliation of a closed subset of $\Sigma$ by simple geodesics.
Even though one needs to fix some hyperbolic structure on the underlying topological surface $\Sigma$ in order to speak about \emph{geodesic} laminations, there exists a natural correspondence between geodesic laminations for different hyperbolic structures that enables one to define them without refering to a particular hyperbolic structure on $\Sigma$. 

A geodesic lamination of $\Sigma$ is \textbf{complete} when each of its complementary components is the interior of an ideal triangle.
One particular type of such complete geodesic laminations is obtained by considering a geodesic pants decomposition of the surface and by adding 
as many disjoint simple geodesics as possible, each necessarily spiralling around the components of the pants decomposition.
These are the type of complete geodesic laminations that were considered in \cite{BBFS09}. 
We emphasize that the latter have finitely many leaves whereas generic complete geodesic laminations are made up of uncountably many leaves.

Some geodesic laminations can be endowed with a {\bf transverse measure} which gives a ``mass" to the set of leaves crossing any given transverse arc. Geodesic laminations equipped with a transverse measure are called {\bf measured} geodesic laminations and they all together form a manifold $\mathcal{ML}(\Sigma)$. Once the surface $\Sigma$ is equipped with a hyperbolic metric, it is possible to satisfactorily define the length of any measured geodesic lamination with respect to the given hyperbolic metric (see Section \ref{section:length}).

Let us fix a complete geodesic lamination $\mu$ of $\Sigma$. ($\mu$ need not carry any transverse measure.)
Thurston introduced global coordinates on Teichm\"uller space $\mathcal{T}(\Sigma)$ called the \textbf{shear coordinates} associated to $\mu$.
These coordinates record the ``amount of shearing", so to speak, between pairs of ideal triangles of $\Sigma\setminus\mu$:
an ideal triangle has one \textbf{distinguished point} per side (called the middle point in \cite{BBFS09}). 
The shear between two ``facing" ideal triangles is then defined by measuring the signed shift between the distinguished points of the  facing sides of the ideal triangles. 
All these notions will be made clearer shortly, but it is worth saying right now that we shall use the universal covering to define them properly.
As usual, defining coordinates on Teichm\"uller space requires some choices that eventually turn out to be not so important (in our topologically finite situation).

Let $\lambda\in\mathcal{ML}(\Sigma)$ be a measured geodesic lamination and let 
$$
\ell_{\lambda}\ :\ \mathcal{T}(\Sigma)\to\mathbb{R}_{+}
$$
be its length function.
The theorem we shall prove here is the following.

\begin{theorem}
\label{theorem_principal}
Let $\mu$ be a complete geodesic lamination of the closed surface $\Sigma$.
The length function $\ell_{\lambda}$ of a measured geodesic lamination $\lambda$ is convex in terms of Thurston's shear coordinates on $\mathcal{T}(\Sigma)$ associated to $\mu$. Furthermore, if $\lambda$ intersects all leaves of $\mu$ transversely, then the length function is strictly convex.
\end{theorem}

As a consequence we describe the behaviour of the length of a measured geodesic lamination along a stretch line. 
In particular we obtain the following result whose words will be defined and explained later on.

\begin{corollary}
The length of the horocyclic lamination strictly decreases from infinity to zero along a stretch line.
\end{corollary}

We also recover Kerckhoff's theorem about strict convexity of lengths along earthquake paths in Teichm\"uller space.

\section{Length of measured laminations}
\label{section:length}

This section is devoted to the length of a measured lamination.
It is here established that, among the measured laminations that belong to the same measure class, the \emph{geodesic} one has the smallest length.

For the time being, a \textbf{lamination} refers to a foliation of a closed subset of $\Sigma$ by simple curves called the \textbf{leaves} of the lamination. When leaves are geodesic, we recover the notion of \emph{geodesic} lamination.
We shall later consider objects we shall also call laminations for which leaves need not be simple nor disjoint.

A \textbf{transverse measure} on a (geodesic) lamination $\lambda$ is a positive Radon measure defined on each compact arc transverse to $\lambda$ which is invariant under isotopies that leave the leaves of $\lambda$ globally fixed.

\subsection{Length of measured laminations}
\label{subsection:length}

Let us first define the length of a measured lamination $L$ of the surface $\Sigma$ equipped with a hyperbolic structure $h$.

Consider a finite collection of simple, compact arcs $A=\{\alpha_{1},\ldots,\alpha_{n}\}$ in $\Sigma$ with disjoint interiors, cutting the leaves of $L$ transversely into compact arcs. We require moreover that the endpoints of each arc of $A$ lie outside $L$. We call such a collection $A$ an {\bf arc system} for $L$. 

For each pair $i,j\in\{1,\ldots,n\}$, let $A_{ij}$ be the set of connected components of $L\setminus L\cap(\cup_{i=1}^{n}\alpha_{i})$ with one endpoint on $\alpha_{i}$ and the other on $\alpha_{j}$. 
Note that $A_{ij}=A_{ji}$ and that $A_{ij}$ is empty when all arcs contained in $L$ joining $\alpha_{i}$ and $\alpha_{j}$ meet another arc of the arc system. 
Each set $A_{ij}$ is equipped with the measure induced by the transverse measure of $L$.
The {\bf length} of $L$ with respect to $h$ (and to the arc system $A=\{\alpha_{1},\ldots,\alpha_{n}\}$) is defined by
$$
\ell_{L}(h) = \sum_{i=1}^{n}\sum_{j\geq i}\int_{A_{ij}}\ell_{x}(h)dL(x).
$$
Here, $\ell_{x}(h)$ denotes the length of one component $x$ of $A_{ij}$.
In somewhat less rigorous terms, the length of $L$ is obtained by cutting $L$ into arcs and by summing the lengths of all these arcs using the transverse measure of $L$.

\begin{lemma}
The definition of $\ell_{L}(h)$ does not depend upon the choice of the arc system $A=\{\alpha_{1},\ldots,\alpha_{n}\}$. 
\end{lemma}

\begin{proof}
We first single out three operations on arc systems which do not change the length of $L$ as defined above.
\begin{enumerate}
\item a slight isotopy of the arcs of $A$ which respects $L$ and does not change the respective positions of these arcs. (The length of $L$ does not change because of the invariance of the transverse measure of $L$).

\item adding an arc to an existing arc system so that the union is still an arc system.

\item spliting an arc into two subarcs. (This yields another arc system which does not modify the length of $L$ as soon as we take care of taking the common endpoint of the two subarcs away from any leaf of $L$).
\end{enumerate}

Now let $B$ be another arc system for $L$. 
Up to performing a slight isotopy (1), we can assume that wherever an arc of $B$ intersects an arc of $A$, it does it transversely. 
Now by cutting arcs that cross each other transversely (3), we obtain an arc system (2) we denote by $A\cup B$. 
Because these operations do not change the length of $L$, we obtain that the length of $L$ with respect to $A$ and the length of $L$ with respect to $B$ are both equal to the length of $L$ with respect to $A\cup B$. 
Hence the length of $L$ with respect to $A$ is equal to the length of $L$ with respect to $B$.
\end{proof}

\subsection{Thick train-track approximations of measured geodesic laminations}
\label{subsection:thicktraintrack}

Provided with the notion of length for a measured lamination, we now focus on measured \emph{geodesic} laminations and give a way for getting natural curve systems to compute their lengths. 
This uses the notion of thick train-track approximations to measured geodesic laminations.
The construction to follow is due to Thurston \cite{Thurston76} \cite{Thurston86}.
\parpic{
\setlength\fboxrule{0.1pt}
\psfrag{N}{\tiny $N$}
\psfrag{l}{\tiny $\lambda$}
\fbox{\includegraphics[width=3cm, height=2.5cm]{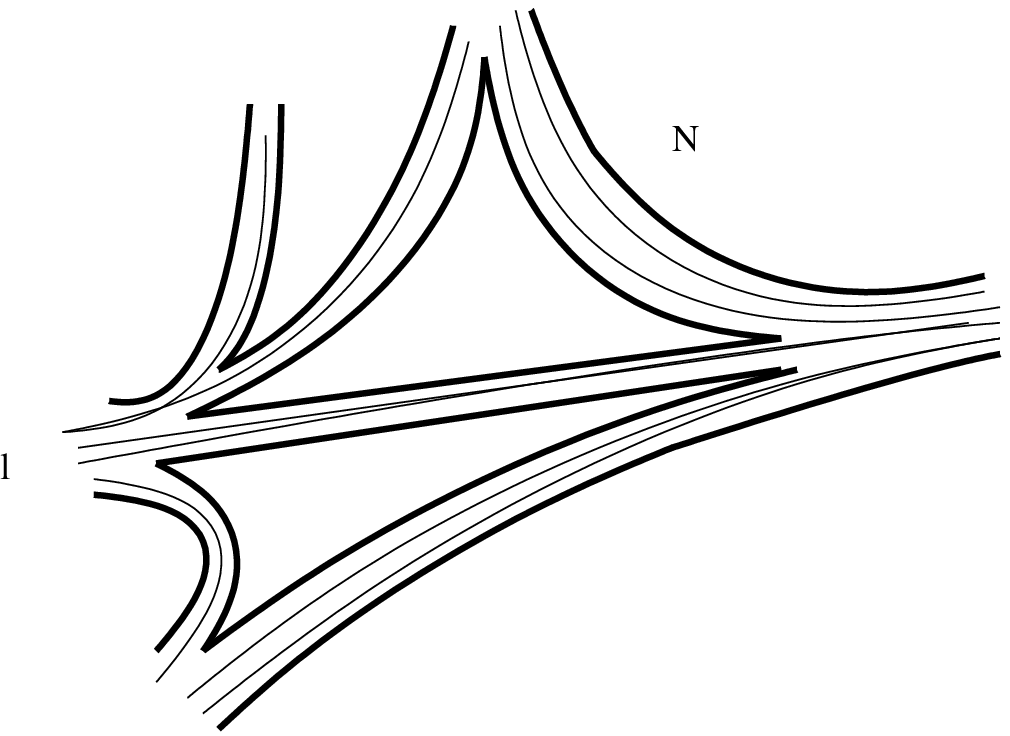}}
}
\noindent Assume that the surface $\Sigma$ is endowed with a hyperbolic structure $h$.
Let $\lambda$ be a geodesic lamination of $\Sigma$.
Choose a positive number $\epsilon$ small enough such that the $\epsilon$-regular neighbourhood of $\lambda$ has stable topological type,
which means that for any $\epsilon'\leq\epsilon$, the $\epsilon'$-regular neighbourhood is homotopically equivalent to the $\epsilon$-one.
Let $N$ denote this $\epsilon$-neighbourhood of $\lambda$. 

The neighbourhood $N$ is also obtained as follows.
\parpic{
\setlength\fboxrule{0.1pt}
\fbox{\includegraphics[width=2cm, height=2.5cm]{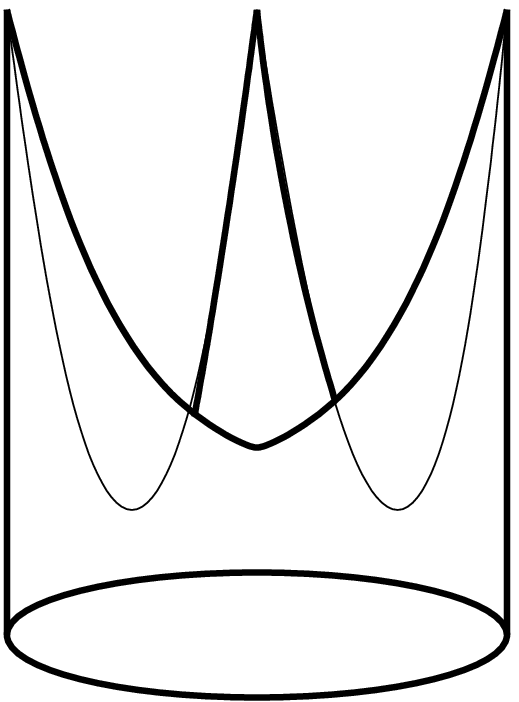}}
}
\noindent Consider the metric completion of a component of $\Sigma\setminus\lambda$:
it is a hyperbolic surface, $M$, with totally geodesic boundary.
Each component of $\partial M$ is either a simple closed geodesic or an infinite geodesic.
Moreover, any end of an infinite geodesic of $\partial M$ is asymptotic to an end of another one, 
and these asymptotic geodesics gather to form a \textbf{crown} boundary; 
one instance is when the component $M$ is an ideal polygon.
A \textbf{spike} is the region enclosed between two asymptotic half-geodesics. 
(See the picture showing the completion $M$ of a component with one three-spiked crown boundary and one closed boundary).
The number $\epsilon$ is then taken small enough so that the $\epsilon$-regular neighbourhoods in $M$ of the various types of boundaries of $\partial M$ are disjoint and their interiors all are topological annuli.
The neighbourhood $N$ in $\Sigma$ is then obtained by taking the union of $\lambda$ and of these various $\epsilon$-regular neighbourhoods for all components of $\Sigma\setminus\lambda$.

One advantage of the second description is that the neighbourhood $N$ can be equipped with a measured foliation obtained by foliating every annulus about boundary components:
for an annulus about a simple closed curve, foliate with geodesic segments perpendicular to both boundary components of the annulus. 
For an annulus about a crown boundary, first choose disjoint spikes of the annulus and then foliate each of them with pieces of horocycles perpendicular to the spike's sides.
Now between two foliated spikes remains a small non-foliated compact zone of the annular neighbourhood; foliate such a zone with leaves averaging smoothly between the leaves of the foliations of the spikes and which are \emph{perpendicular} to the \emph{geodesic} boundary of the annulus. 
The foliations of the various annuli yield a foliation of $N\setminus\lambda$.
Since the leaves of the foliation are by construction perpendicular to leaves of $\lambda$ and since the leaves of $\lambda$ form a Lipschitz field of directions,
the foliation on $N\setminus\lambda$ extends to a foliation over the whole neighbourhood $N$.
Let us denote by $F_{\lambda}(h)$ this foliation;
it is transverse to the boundary of $N$ and perpendicular to the leaves of $\lambda$. 
We further equip $F_{\lambda}(h)$ with the transverse measure defined by requiring that the transverse measure of an arc entirely contained in a leaf of $\lambda$ coincides with the length of that arc.
\parpic{
\setlength\fboxrule{0.1pt}
\psfrag{T}{\tiny $\Theta_{(\lambda,h)}$}
\fbox{\includegraphics[width=3cm, height=2.5cm]{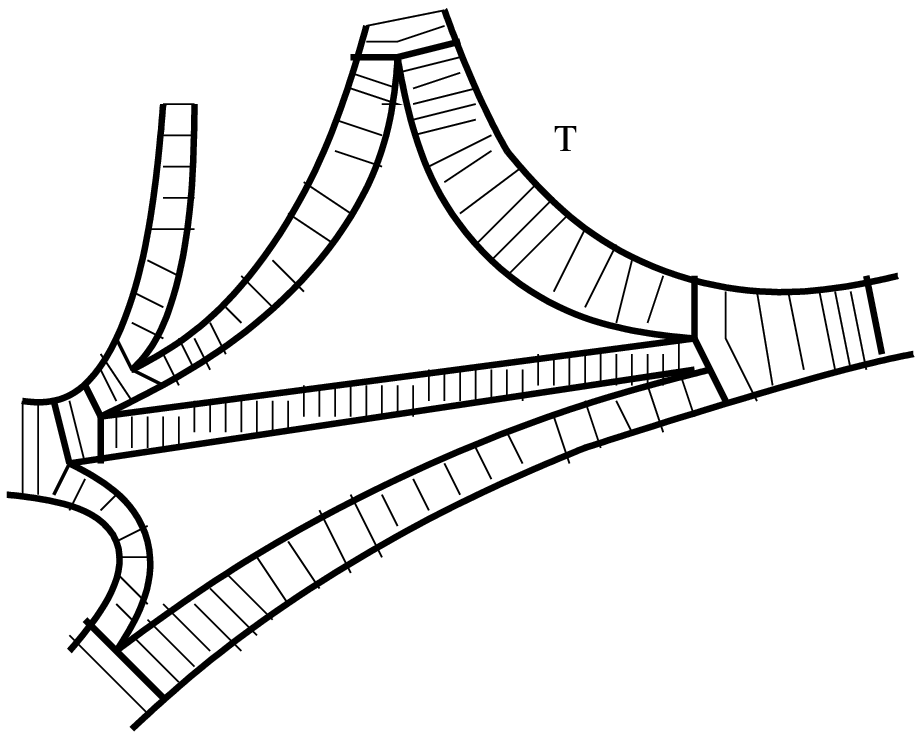}}
}\noindent The neighbourhood $N$ together with the foliation $F_{\lambda}(h)$ forms what is called a (thick) \textbf{train-track approximation} $\Theta_{(\lambda,h)}$ of $\lambda$.
The leaves of the foliation $F_{\lambda}(h)$ are called the \textbf{traverses} of the train-track $\Theta_{(\lambda,h)}$.
The boundary of the train-track $\Theta_{(\lambda,h)}$ is smooth except maybe at finitely many points where cusps arise.
The traverses passing through these singular points are said \textbf{singular}.

The singular traverses decompose $\Theta_{(\lambda,h)}$ into finitely many rectangles called the \textbf{branches} of $\Theta_{(\lambda,h)}$. 
An annular component of $\Theta_{(\lambda,h)}$ around an isolated closed leaf of $\lambda$ is also viewed as a single branch by dividing the annulus with an arbitrary (non-singular) traverse. 
Each branch $b$ is foliated by traverses and traversed by pieces of leaves of $\lambda$ whose lengths are \emph{all the same};
we call this common length the \textbf{width} of the branch $b$ and denote it by $w_{b}$.
The transverse measure of $\lambda$ gives a measure to any traverse of $\Theta_{(\lambda,h)}$.
Clearly the measure of all traverses in a given branch $b$ are equal and this number is called the \textbf{mass} of $b$ (induced by $\lambda$) and is denoted by $m_{b}$. 

\parpic{
\setlength\fboxrule{0.1pt}
\psfrag{T}{\tiny $\Theta_{\mu}$}
\psfrag{T'}{\tiny $\hat{\Theta}_{\mu}$}
\psfrag{m}{\tiny $\mu$}
\fbox{\includegraphics[width=2.5cm, height=5cm]{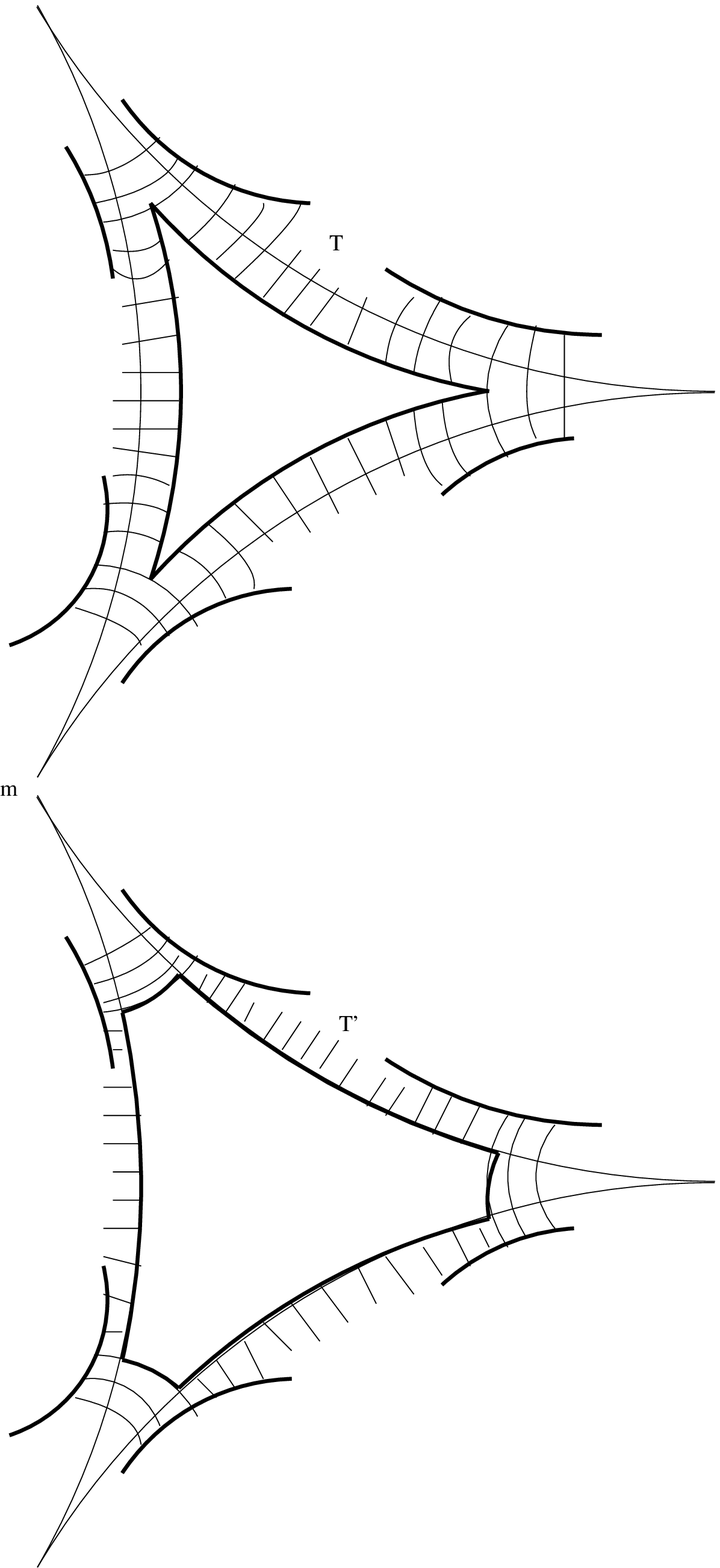}}
}\noindent Suppose here that the lamination $\mu$ is \emph{complete} and consider a train-track approximation $\Theta_{\mu}$ of $\mu$.
For later use we introduce a slight modification $\hat{\Theta}_{\mu}$ of $\Theta_{\mu}$.
Since $\mu$ is complete, each component of $\Sigma\setminus\Theta_{\mu}$ is a non-geodesic triangle contained in some ideal triangle of $\mu$.
We collapse the boundary of that triangle along the traverses of $\Theta_{\mu}$ onto the geodesic boundary of the ideal triangle containing it.
We thus get a train-track $\hat{\Theta}_{\mu}$ contained in $\Theta_{\mu}$.
Each component of $\Sigma\setminus\hat{\Theta}_{\mu}$ has the shape of a hexagon with three disjoint geodesic sides contained in $\mu$ and three other small horocyclic sides each contained in a singular traverse of $\Theta_{\mu}$.
Note that the widths and masses of $\hat{\Theta}_{\mu}$ are equal to those of $\Theta_{\mu}$.

\subsection{Length of measured geodesic laminations}

Let us now turn to measured \emph{geodesic} laminations.
Train-track approximations of a measured geodesic lamination $\lambda$ provide natural arc systems, namely singular traverses, for cutting the leaves of $\lambda$ and thus computing its length.

Consider a train-track approximation $\Theta_{(\lambda,h)}$ of $\lambda$ as above.
Let $b_{1},\cdots,b_{n}$ be the branches of $\Theta_{(\lambda,h)}$ and let $w_{1},\cdots,w_{n}$ and $m_{1},\cdots,m_{n}$ respectively denote the widths and the masses of the branches.
We use the arc system provided by the singular traverses of $\Theta_{(\lambda,h)}$. Any component $x$ cut out by this arc system is completely contained in a branch of $\Theta_{(\lambda,h)}$ and all components contained in the same branch $b_{i}$ have length $w_{i}$. Therefore, if $A_{ij}$ is the set of components of $L\cap b_{k}$ contained in the branch $b_{k}$ bordered by the traverses $\alpha_{i}$, $\alpha_{j}$, we get $\int_{A_{ij}}\ell_{x}(h)dL(x)=w_{k}\int_{A_{ij}}dL(x)=w_{k}\cdot m_{k}$.
The length of the measured geodesic lamination $\lambda$ is then given by
$$
\ell_{\lambda}(h) = \sum_{k=1}^{n} w_{k}\cdot m_{k} = W\cdot M,
$$
where $W$ and $M$ respectively denote the $n$-vectors of widths and of masses induced by $\lambda$.
This shows in particular that length is a linear function of masses and widths.

\begin{remark}
The length $\ell_{\lambda}(h)$ of a measured \emph{geodesic} lamination $\lambda$ with respect to the hyperbolic structure $h$ can also be elegantly defined in more abstract terms using the language of geodesic currents, see \cite{Bonahon88} and below. 
\end{remark}

\subsection{More information about laminations}
\label{subsection:moreinfo}

Let us start with a definition. Two \emph{measured} laminations are in the same {\bf measure class} if they are homotopic and if this homotopy transports one transverse measure on the other.
Now that we have seen a way of computing the length of a given measured \emph{geodesic} lamination $\lambda$, we want to extend it to a way of computing the length of any measured lamination $L$ in the same measure class as $\lambda$. 
Actually, as we shall see, the leaves of $L$ need not be simple nor disjoint (we may however require that the leaves of $L$ are piecewise smooth).
The ``lamination" $L$ is thus viewed as a collection of curves in $\Sigma$ (not necessarily simple nor disjoint) each of which being homotopic to a unique leaf of the given geodesic lamination $\lambda$, so that the leaves of $L$ and of $\lambda$ are in one-to-one correspondence. 
Such a ``lamination" is better understood by looking at its preimage in the universal covering and by considering the collection of  endpoints of its leaves, for leaves of $\lambda$ and of $L$ have the same endpoints. 
The material of this paragraph is borrowed from Bonahon's great paper \cite{Bonahon88} about geodesic currents and is aimed at giving a way of equipping any lamination $L$ homotopic to a given measured geodesic lamination $\lambda$ with a transverse measure so that $L$ and $\lambda$ fall in the same measure class.

Consider a geodesic lamination $\lambda$ of the surface $\Sigma$ equipped with some hyperbolic structure.
This structure yields an identification between the universal covering $\widetilde{\Sigma}$ of $\Sigma$ and the hyperbolic plane $\mathbb{H}^{2}$.
Let us fix such an identification.
In this paper, we shall mainly use the unit disc model of the hyperbolic plane.
Let $\widetilde{\lambda}$ be the preimage of $\lambda$ in $\widetilde{\Sigma}$.
Each leaf of $\widetilde{\lambda}$ has two (unordered) endpoints on  $S_{\infty}$, the boundary at infinity of $\mathbb{H}^{2}$. 
The set of endpoints of $\widetilde{\lambda}$ thus forms a closed subset $G(\lambda)$ of 
$(S_{\infty}\times S_{\infty}-\Delta)/\mathbb{Z}_{2}$, where $\Delta$ is the diagonal of 
$S_{\infty}\times S_{\infty}$ and $\mathbb{Z}_{2}$ acts on $S_{\infty}\times S_{\infty}$ by swapping both factors.
The closed subset $G(\lambda)$ is invariant under the action of the fundamental group $\pi_{1}(\Sigma)$ through deck transformations.
We thus get a map $\lambda\mapsto G(\lambda)$, well-defined up to positive isometry of $\mathbb{H}^2$.
Now for any lamination $L$ homotopic to $\lambda$ we have $G(L)=G(\lambda)$.

Assume now that the geodesic lamination $\lambda$ is equipped with a transverse measure. The latter induces a measure on $G(\lambda)$ invariant by the action of $\pi_{1}(\Sigma)$ as follows.
First, equip the preimage $\widetilde{\lambda}$ of $\lambda$ in $\widetilde{\Sigma}$ with the $\pi_{1}(\Sigma)$-invariant transverse measure induced by the transverse measure of $\lambda$.
Consider a geodesic $g$ of $\widetilde{\lambda}$ and a compact arc $k$ transverse to $g$.
The set of leaves of $\widetilde{\lambda}$ that intersect $k$ yields a compact subset of $G(\lambda)$ whose measure is defined to be the $\widetilde{\lambda}$-transverse measure of $k$. 
This defines a measure for neighbourhoods of $g$ in $G(\lambda)$ and thus a $\pi_{1}(\Sigma)$-invariant measure on $G(\lambda)$.
A $\pi_{1}(\Sigma)$-invariant measure with compact support on $(S_{\infty}\times S_{\infty}-\Delta)/\mathbb{Z}_{2}$ is called a {\bf geodesic current}.
A measured geodesic lamination $\lambda$ of $\Sigma$ thus defines a geodesic current $G(\lambda)$ and any lamination in the same measure class as $\lambda$ yields the same geodesic current. More generally, any collection of curves in $\Sigma$ which is homotopic to $\lambda$ and which can be equipped with a transverse measure preserved by this homotopy defines the same geodesic current.

We shall conversely encounter ``laminations" $L$ which are collections of curves in one-to-one correspondence with the leaves of a given measured geodesic lamination $\lambda$ such that each curve is homotopic to the corresponding leaf of $\lambda$. 
We need to equip $L$ with a transverse measure so that $L$ and $\lambda$ fall into the same measure class.
To do this, first recall that $\lambda$ defines a geodesic current on $G(\lambda)$.
Because $L$ and $\lambda$ are homotopic, we have $G(L)=G(\lambda)$.
The measure on $G(\lambda)$ easily induces a transverse measure on $L$: 
define the measure of an arc $k$ transverse to $L$ as the measure of the endpoints of the leaves of $L$ crossing $k$.

\subsection{Length of a measured lamination in a given measure class}

Let $\lambda$ be a measured \emph{geodesic} lamination of $\Sigma$. 
In order to express the length of any measured lamination $L$ in the same measure class as $\lambda$, we take a train-track approximation of $\lambda$ so as to obtain an arc system for $L$.

So consider a train-track approximation $\Theta=\Theta_{(\lambda,h)}$ of $\lambda$ and a leaf $l$ of $L$.
This leaf $l$ is homotopic to the leaf $\lambda_{l}$ of $\lambda$.
We want to cut the leaf $l$ into pieces, each of whose is associated to a branch of $\Theta_{(\lambda,h)}$, then to sum up the lengths of all pieces associated to a given branch.
However, when $l$ has a lot of meanders and loops lying outside $\Theta_{(\lambda,h)}$, it is at first sight not obvious how to know which pieces of $l$ should be associated to a particular branch of $\Theta_{(\lambda,h)}$.
Using the universal covering (identified with the hyperbolic plane) and the geodesic lamination $\lambda$, there is a  way of sorting this out.
\parpic{
\setlength\fboxrule{0.1pt}
\psfrag{s}{\tiny $S_{i}$}
\psfrag{b}{\tiny $\widetilde{b}_{i}$}
\psfrag{l}{\tiny $\widetilde{l}$}
\fbox{\includegraphics[width=2.5cm, height=2.5cm]{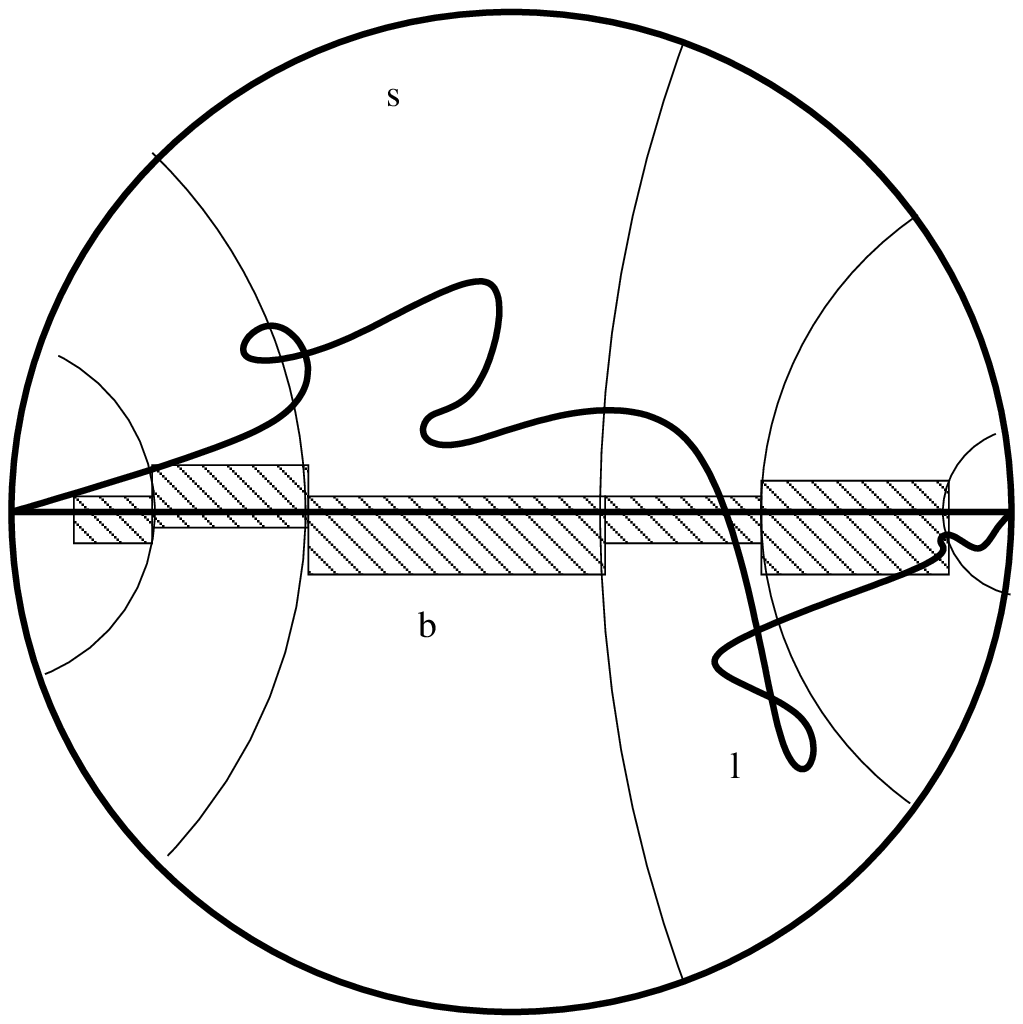}}
}
\noindent Consider a lift $\widetilde{l}$ of $l$ and the lift $\widetilde{\lambda_{l}}$ of $\lambda_{l}$ that has the same endpoints as $\widetilde{l}$. 
The geodesic $\widetilde{\lambda_{l}}$ goes through a countable collection of branches (shaded on the picture) $\widetilde{b}_{i}$ of the preimage of $\Theta_{(\lambda,h)}$ in the universal covering. 
These branches cut the geodesic $\widetilde{\lambda_{l}}$ into segments $]x_{i};y_{i}[=\widetilde{b}_{i}\cap\widetilde{\lambda_{l}}$.
The hyperbolic plane is covered by the strips $S_{i}$ bordered by the geodesics passing through the points $x_{i}, y_{i}$ perpendicularly to $\widetilde{\lambda_{l}}$. 
These strips have disjoint interiors and they cut, in a well-defined manner, the leaf $\widetilde{\lambda_{l}}$ into countably many pieces each of which is attached to a branch of $\Theta_{(\lambda,h)}$ crossed by $\widetilde{\lambda_{l}}$.

Now let us go back down to our hyperbolic surface and pick a branch $b_{i}$ of $\Theta_{(\lambda,h)}$ ($1\leq i\leq n$). 
Let $x$ be a connected component of $b_{i}\cap\lambda$;
it is contained in the leaf $\lambda_{x}$ of $\lambda$.
Lift the leaf $\lambda_{x}$ and the branch $b_{i}$ so that the geodesic lift of $\lambda_{x}$ crosses the lift of $b_{i}$.
There is a unique leaf $l_{x}$ of $L$ that has the same endpoints as $\lambda_{x}$.
We get from what has just been explained a countable collection $x_{j}$, $j\in J(x,b_{i})$, of pieces of the leaf $l_{x}$ attached to the branch $b_{i}$ and to the component $x$. 
By choosing an orientation on $l_{x}$ we get a total ordering on the set $J(x,b_{i})$ and thus a bijection with a subset of $\mathbb{N}$. 
We equip $J(x,b_{i})$ with the counting measure $\delta$.

The length of $L$ can thus be expressed by summing up all these pieces for each branch $b_{i}$, $1\leq i\leq n$, of $\Theta_{(\lambda,h)}$:
$$
\ell_{L}(h)=\sum_{i=1}^{n}\int_{\mathcal{C}(\lambda\cap b_{i})}\left(\int_{J(x,b_{i})}\ell_{x_{j}}(h)\,\delta(j)\right)d\lambda(x),
$$
where $\mathcal{C}(\lambda\cap b_{i})$ denotes the set of connected components $x$ of $\lambda\cap b_{i}$.
This formula naturally extends the definition of length of measured geodesic laminations. 
It also covers the classical notion of length of (non necessarily simple) closed curves, or of any finite collection of such. 
It is easily checked that this definition makes sense and that the leaves of $L$ are cut in a well-defined manner.

Let us close up this section by giving our most general result about length of measured laminations.

\begin{theorem}
\label{theorem:minimal_length}
Let $\lambda$ be a measured geodesic lamination of the surface $\Sigma$
and let $L$ be a measured lamination in the measure class of $\lambda$.
Then the length of $L$ with respect to any fixed hyperbolic structure on $\Sigma$ is greater or equal to the length of $\lambda$ and equality holds if and only if $L$ is equal to $\lambda$.
\end{theorem}

\begin{proof}
For any $i\in\{1,\cdots,n\}$ and any component $x\in\mathcal{C}(\lambda\cap b_{i})$, the piece $x_{j}$, $j\in J(x,b_i)$, is a curve contained in the strip $S_{i}$ whose endpoints belong to the boundary of $S_{i}$.
The strip $S_{i}$ is bordered by two geodesics connected perpendicularly by the geodesic segment $\lambda_{x}\cap b_{i}$.
The length of this geodesic segment is $w_{i}$, the width of $b_{i}$.
For all $j\in J(x,b_i)$, the length of the curve $x_{j}$ hence satisfies
\begin{align*}
\ell_{x_{j}}(h)&\geq w_{i} &&\text{if the endpoints of $x_{j}$ are on different boundary components of $S_{j}$}\\
\ell_{x_{j}}(h)&\geq 0 &&\text{if the endpoints of $x_{j}$ are on the same boundary component of $S_{j}$.}
\end{align*}

Let $ J^{0}(x,b_i)$ be the subset of $J(x,b_i)$ indexing those curves $x$ whose endpoints lie on different sides of the strip.
We have $\vert J^{0}(x,b_i)\vert\geq1$, since there is at least one component $x_{j}$ whose endpoints lie on different boundary components. 
Therefore,
$$
\int_{J(x,b_{i})}\ell_{x_{j}}(h)\,\delta(j)\geq\int_{J^{0}(x,b_{i})}\ell_{x_{j}}(h)\,\delta(j)\geq w_{i}\int_{J^{0}(x,b_{i})}\delta(j)=w_{i}\vert J^{0}(x,b_{i})\vert\geq w_{i}.
$$
It follows that
$$
\ell_{L}(h)\geq\sum_{i=1}^{n}\int_{\mathcal{C}(\lambda\cap b_{i})}w_{i}d\lambda(x)=\sum_{i=1}^{n}w_{i}\left(\int_{\mathcal{C}(\lambda\cap b_{i})}d\lambda(x)\right).
$$
But $\int_{\mathcal{C}(\lambda\cap b_{i})}d\lambda(x)=m_{i}$, the mass of the branch $b_{i}$.
We therefore get
$$
\ell_{L}(h)\geq\sum_{i=1}^{n}w_{i}\cdot m_{i}=\ell_{\lambda}(h).
$$
By inspection of the above inequalities, it is easily checked that this infimum is uniquely realized by $\lambda$. 
The proof is over.
\end{proof}

\section{Shears between pairs of ideal triangles}
\label{section:shear}

\subsection{Horocyclic foliation}

Before giving the description of shear coordinates, we need to introduce the horocyclic foliation associated to a complete geodesic lamination $\mu$ and to a given hyperbolic structure $h$ on $\Sigma$.
As it will clearly appear, the construction of the horocyclic foliation is very similar to that of a thick train-track approximation of a geodesic lamination $\mu$ taking however into account completeness of $\mu$.

So fix a \emph{complete} geodesic lamination $\mu$ of $\Sigma$.
Consider a hyperbolic structure $h$ on $\Sigma$ and straighten $\mu$ to a genuine geodesic lamination for this structure.
Each component of $\Sigma\setminus\mu$ is isometric to the interior of an ideal triangle of the hyperbolic plane.
\parpic{
\setlength\fboxrule{0.1pt}
\fbox{\includegraphics[width=2cm, height=2.5cm]{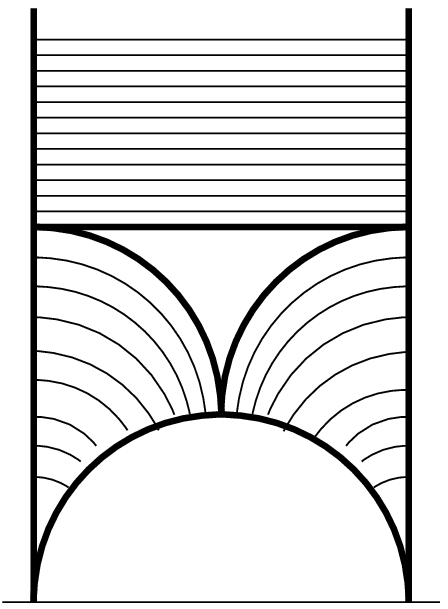}}
}
\noindent Now start by foliating symetrically the three spikes of every ideal triangle with arcs of horocycles perpendicular to the edges of the ideal triangles. 
Eventually, it necessarily remains a non-foliated region which is a triangle bounded by three isometric arcs of horocycles meeting tangentially on the edges of the ideal triangle.
The three vertices of the non-foliated region are called the \textbf{distinguished points} of the ideal triangle.
The leaves of the geodesic lamination $\mu$ form a Lipschitz field of directions, 
so the partial foliations defined in the interior of all ideal triangles extend over the lamination $\mu$.
This yields a partial foliation of $\Sigma$, called the \textbf{horocyclic foliation} and denoted by $F_{\mu}(h)$, which we equip with the transverse measure obtained by requiring that the measure of an arc contained in a leaf of $\mu$ equals its length.  
We denote the equivalence class of $F_{\mu}(h)$ in $\mathcal{MF}(\Sigma)$ by the same symbols.
We thus get a map

\begin{eqnarray*}
F_{\mu}\ :\ \mathcal{T}(\Sigma)&\to&\mathcal{MF}(\Sigma)\\
h&\mapsto&F_{\mu}(h).
\end{eqnarray*}

Thurston showed that this map is a homeomorphism onto its image, the latter being the space $\mathcal{MF}(\mu)$ of equivalence classes of measured foliations that are transverse to $\mu$ (see \cite{Thurston86}). 
Since this result is the corner stone of Thurston's theory of shear coordinates, we state it here as a theorem.
It says that hyperbolic structures on $\Sigma$ are completely understood through measured foliations transverse to $\mu$ (actually, knowing the foliations in sufficiently fine neighbourhoods of $\mu$ is sufficient) and that any foliation in $\mathcal{MF}(\mu)$ is, in a unique way, the horocyclic foliation of some hyperbolic structure.

\begin{theo}[Thurston \cite{Thurston86}]
Let $\Sigma$ be a complete hyperbolic surface of finite area and let $\mu$ be a complete geodesic lamination of $\Sigma$. 
Let $\mathcal{MF}(\mu)$ be the space of measured foliations of $\Sigma$ transverse to $\mu$ and standard near the cusps and let $\mathcal{T}(\Sigma)$ be the Teichm\"uller space of $\Sigma$, that is, the space of isotopy classes of complete hyperbolic structures of finite area on $\Sigma$. 
The map 
\begin{eqnarray*}
F_{\mu}\ :\ \mathcal{T}(\Sigma)&\to&\mathcal{MF}(\mu)\\
h&\mapsto&F_{\mu}(h)
\end{eqnarray*}
that associates to a point $h$ of Teichm\"uller space its horocyclic foliation $F_{\mu}(h)$ is a homeomorphism. 
\end{theo}

\begin{remark}
\label{remark:technical}
This theorem is stated in its original, general form when $\Sigma$ might have punctures. The hyperbolic structures considered in this case must be so that all punctures are given the shape of a cusp. This is the reason for the ``standardness" condition around the cusps in the statement. Since we stick to closed surfaces, we can ignore this condition here. The readers who would like to have more information about the general situation are referred to Thurston's paper \cite{Thurston86}.
 
In order to prove that the above map $F_{\mu}$ is surjective, Thurston constructs the inverse map, 
namely, he constructs starting from a class $F$ of measured foliations transverse to $\mu$ a hyperbolic structure 
$h$ such that $F_{\mu}^{-1}(F)=h$.
For technical reasons, this is more easily done by assuming that the surface $\Sigma$ is equipped with some base hyperbolic structure $h_{0}$ and by constructing from this base structure the new hyperbolic structure $h$ that fulfills the requirements (see \cite{Thurston86});
we shall also encounter in the sequel the need of fixing a base structure $h_{0}$.   
\end{remark}

We shall now explain in some detail why horocyclic foliations encode hyperbolic structures.

\subsection{Shear between ideal triangles of $\mu$ for a given hyperbolic structure}

Assume that the surface $\Sigma$ is equipped with a hyperbolic structure $h$ 
and a complete geodesic lamination $\mu$.
Each ideal triangle of $\Sigma\setminus\mu$ has a \textbf{centre}, 
namely, the centre of symmetry of the ideal triangle.
By symmetry, the three orthogonal projections of the centre on the ideal triangle's sides give the distinguished points. 

The shear between two ideal triangles measures how these two ideal triangles sit one with respect to the other for the hyperbolic structure $h$ on $\Sigma$. 
It is obtained by measuring the signed ``shift" between the centres of these ideal triangles. 
However, the shear between two ideal triangles of $\Sigma\setminus\mu$ does not make sense as it, but can be satisfactorily defined when one considers pairs of ideal triangles in the universal covering, that is, by considering the preimage $\widetilde{\mu}$ of $\mu$ in the universal covering $\widetilde{\Sigma}$ of $\Sigma$ and pairs of ideal triangles of $\widetilde{\Sigma}\setminus\widetilde{\mu}$.

The \emph{oriented} surface $\Sigma$ being equipped with a hyperbolic structure $h$, identify, using the structure $h$, 
the metric universal covering $\widetilde{\Sigma}$ with the hyperbolic plane, so that orientations match.
The preimage $\widetilde{\mu}$ of $\mu$ in $\widetilde{\Sigma}$ is then a complete geodesic lamination of $\widetilde{\Sigma}$ which is globally invariant under the usual action of the fundamental group of $\Sigma$.
Equip $\widetilde{\Sigma}$ with the preimage $\widetilde{F}_{\mu}(h)$ of the horocyclic foliation $F_{\mu}(h)$ associated
to the complete geodesic lamination $\mu$ and to the hyperbolic structure $h$.
We collapse the non-foliated regions of $\widetilde{F}_{\mu}(h)$ onto tripods in order to obtain a genuine measured foliation;
we denote it by $\widetilde{F}_{\mu}(h)$ again, even though it is well-defined up to isotopies supported in neighbourhoods of non-foliated regions.

Consider a pair $T_{1}$, $T_{2}$ of ideal triangles of $\widetilde{\Sigma}\setminus\widetilde{\mu}$.
Connect the centres of these ideal triangles by a curve $\delta_{1,2}$ which is quasi-transverse with respect to $\widetilde{F}_{\mu}(h)$, in the sense of \cite{FLP}.
This means that the curve $\delta_{1,2}$ is a concatenation of finitely many compact paths each connecting the centres of ideal triangles, 
each being transverse to the foliation $\widetilde{F}_{\mu}(h)$ and such that two consecutive paths do not lie, 
in the vicinity of the centre where they meet, in the same wedged region enclosed by the singular leaves of $\widetilde{F}_{\mu}(h)$ emanating from the centre.

The main remark to be made is that, since the curve $\delta_{1,2}$ is quasi-transverse, 
once it crosses an ideal triangle of $\widetilde{\Sigma}\setminus\widetilde{\mu}$, 
it does not come through this ideal triangle back again.

Replace every path composing $\delta_{1,2}$ and joining a consecutive pair of centres of ideal triangles by a curve obtained as follows.
The path $\delta$ of $\delta_{1,2}$ to be replaced connects the centre $c$ of an ideal triangle 
to the centre $c'$ of another ideal triangle and this path is transverse to the leaves of $\widetilde{F}_{\mu}(h)$.
By pushing the path $\delta$ along the leaves of $\widetilde{F}_{\mu}(h)$ onto a leaf of $\widetilde{\mu}$ 
which is crossed by $\delta$, 
we obtain a curve $\delta^{*}$ made out of a singular leaf of $\widetilde{F}_{\mu}(h)$ emanating from $c$, 
followed by a geodesic segment contained in a leaf of $\widetilde{\mu}$, 
and finally followed by a singular leaf of $\widetilde{F}_{\mu}(h)$ ending on $c'$.
We say that the curve $\delta^{*}$ is \textbf{horogeodesic}.
The intersection number of $\delta$ with respect to $\widetilde{F}_{\mu}(h)$ then equals the length of the geodesic 
segment in $\delta^{*}$.
By replacing every path of $\delta_{1,2}$ by a horogeodesic curve, 
we obtain a curve $\delta_{1,2}^{*}$ homotopic to $\delta_{1,2}$ relative to the centres it connects 
whose intersection number with respect to $\widetilde{F}_{\mu}(h)$ is equal to $i(\delta_{1,2},\widetilde{F}_{\mu}(h))$ 
and is obtained by adding the lengths of all geodesic segments that compose it. 
We say that the curve $\delta_{1,2}^{*}$ is horogeodesic as well.

Let us now define the shear $\sigma_{1,2}$ between the ideal triangles $T_{1}$ and $T_{2}$.
Connect the centres of $T_{1}$ and $T_{2}$ by a quasi-transverse curve $\delta_{1,2}$.
Replace the quasi-transverse curve by a horogeodesic one, $\delta_{1,2}^{*}$, as above.
The \textbf{shear} $\sigma_{1,2}$ between the ideal triangles $T_{1}$ and $T_{2}$ is the signed transverse measure 
with respect to $\widetilde{F}_{\mu}(h)$ of $\delta_{1,2}^{*}$.
More precisely, the real number $\sigma_{1,2}$ is the sum of the \emph{signed} lengths of the geodesic components of the 
horogeodesic curve $\delta_{1,2}^{*}$, where the sign convention is such that 
once the curve $\delta_{1,2}^{*}$ is oriented from $T_{1}$ to $T_{2}$,
a geodesic segment going to the left has a positive sign while it has a negative sign when it goes to the right.

It is easily checked that the sign assigned to every geodesic segment does actually not depend upon the orientation of the curve $\delta_{1,2}^{*}$ but only on the orientation of $\Sigma$.
We also check that the shear $\sigma_{1,2}$ is well-defined, that is, it does not depend on the chosen quasi-transverse curve.
Since we replace a given quasi-transverse curve by a horogeodesic one in the above definition,
it suffices to check that the shear $\sigma_{1,2}$ does not depend upon the chosen horogeodesic curve.
First note that the sequence of centres connected by a horogeodesic curve $\delta_{1,2}^{*}$ joining the centres of   
$T_{1}$ and $T_{2}$ is uniquely defined.
This is due to the fact that the ideal triangles that the curve $\delta_{1,2}^{*}$ crosses are nested, 
which means that any ideal triangle crossed by $\delta_{1,2}^{*}$ is contained in one of the three 
complementary components of the ideal triangle the curve $\delta_{1,2}^{*}$ has just crossed before.
Therefore, if a horogeodesic curve has to connect the centres of $T_{1}$ and $T_{2}$, 
it must cross a uniquely defined sequence of nested ideal triangles. 
The indeterminacy of the horogeodesic curve thus lies in the pushing process represented in Figure \ref{figshear3}.
It is clear that such a move does not change the signed transverse measure, hence preserves the shear $\sigma_{1,2}$.
This establishes that the shear $\sigma_{1,2}$ is well-defined.
Note that $\sigma_{1,2}=\sigma_{2,1}$, which explains why we talk about the shear \emph{between} $T_{1}$ and $T_{2}$ and not only \emph{from} $T_{1}$ \emph{to} $T_{2}$.

\begin{figure}
\centering
\includegraphics[width=.5\linewidth]{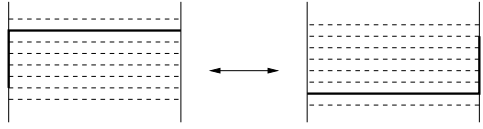}
\caption{Move to be repeated to pass from one horogeodesic curve to another: the horizontal dotted leaves are leaves of the horocyclic foliation whereas the two vertical curves are pieces of leaves of $\mu$. The bold segments represent pieces of some horogeodesic curve.}
\label{figshear3}
\end{figure}

We shall say that two ideal triangles $T_{1}$ and $T_{3}$ of the hyperbolic plane are \textbf{separated} by a third one $T_{2}$ if any geodesic segment going from one side of $T_{1}$ to another side of $T_{3}$ intersects $T_{2}$.

\begin{lemma}
\label{lemma:separation}
Let $T_{1}$ and $T_{3}$ be two ideal triangles separated by a third one $T_{2}$.
Then we have
$$
\sigma_{1,3}=\sigma_{1,2}+\sigma_{2,3}.
$$  
\end{lemma}

\begin{proof}
One can connect the centre of the ideal triangle $T_{1}$ with the centre of the ideal triangle $T_{2}$ by
a horogeodesic curve $\delta_{1,2}$ and do the same with a horogeodesic curve $\delta_{2,3}$ 
connecting the centre of the ideal triangle $T_{2}$ with the centre of the ideal triangle $T_{3}$.
The concatenation of these two curves $\delta_{1,2}$ and $\delta_{2,3}$ we call $\delta_{1,3}$.
The separation assumption of the three ideal triangles ensures that 
the curve $\delta_{1,3}$ is quasi-transverse to $\widetilde{F}_{\mu}(h)$. 
It is furthermore horogeodesic.
The signed transverse measures of the curves $\delta_{1,2}$, $\delta_{2,3}$ and $\delta_{1,3}$ are the sum of the signed transverse measures of their geodesic segments.
This proves that $\sigma_{1,3}=\sigma_{1,2}+\sigma_{2,3}$, hence the lemma.
\end{proof}

A quick way of defining shears is, as it is done in \cite{Bonahon96}, by considering the strip enclosed by the facing sides of the two ideal triangles we want to define the shear of.
By {\bf strip} we mean the closed region of the hyperbolic plane enclosed by two parallel geodesics (the geodesics might coincide though).
When the two geodesics have no common endpoints, there is a unique geodesic segment joining them perpendicularly that we call the {\bf core} of the strip.
When the two geodesics have exactly one common endpoint, the strip is also called a {\bf wedge} and the common endpoint is referred to as the endpoint of the wedge.
\parpic{
\setlength\fboxrule{0.1pt}
\psfrag{s}{\tiny $\sigma_{1,2}$}
\psfrag{S}{\tiny $S_{1,2}$}
\psfrag{t1}{\tiny $T_{1}$}
\psfrag{t2}{\tiny $T_{2}$}
\fbox{\includegraphics[width=3.5cm, height=3.5cm]{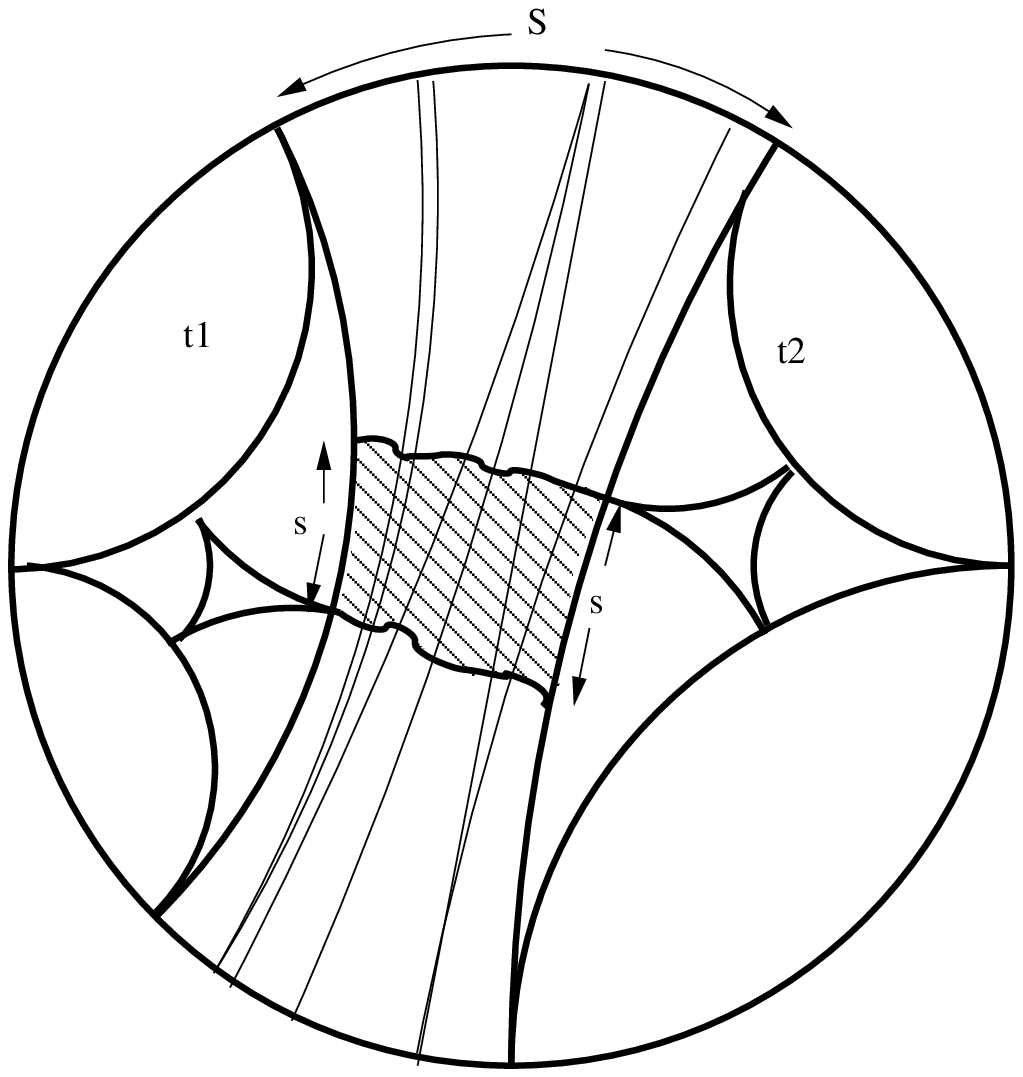}}
}
\noindent So consider the strip $S_{1,2}$ bounded by the facing sides $t_{1}$ and $t_{2}$ of two ideal triangles $T_{1}$, $T_{2}$.
Amongst the leaves of $\widetilde{\mu}$ that are contained in $S_{1,2}$, get rid of those that do not intersect the core. 
The remaining leaves cut the strip into countably many wedges and we say that we thus have a {\bf wedge cut} $\check\mu$ of the strip.
Foliate each wedge of the wedge cut with its horocyclic foliation, that is, the measured foliation whose leaves are contained in the horocycles centered at the endpoint of the wedges.
We get a measured foliation $F_{1,2}$ of the strip $S$ and, using this foliation, we project the distinguished point of the side $t_{1}\subset T_{1}$ of the strip on the other side $t_{2}\subset T_{2}$ of the strip. 
The shear $\sigma_{1,2}$ between $T_{1}$, $T_{2}$ is simply the signed distance between that projection and the distinguished point of $t_{2}$.
The reason why the two definitions we gave for shears are equivalent is based upon a joint use of Lemma \ref{lemma:separation} and Figure \ref{figshear3}.\\

\subsection{Shear between ideal triangles of $\mu$ with respect to a measured foliation transverse to $\mu$}

Thurston's map $h\mapsto F_{\mu}(h)$ characterizes hyperbolic structures through measure classes of foliations on a fixed (hyperbolic) surface $\Sigma$. It is therefore possible to describe shears between pairs of ideal triangles of $\Sigma\setminus\mu$ using a measured foliation transverse to the complete geodesic lamination $\mu$.
Let us thus end this section by defining shears between pairs of ideal triangles of $\Sigma\setminus\mu$ with respect
to a class of measured foliations transverse to $\mu$, that is, with respect to an element $F$ of $\mathcal{MF}(\mu)$.

Consider a complete geodesic lamination $\mu$ and an element $F$ of $\mathcal{MF}(\mu)$ (to fix the objects, it is better to think of the surface $\Sigma$ as equipped with a base hyperbolic structure $h_{0}$; see Remark \ref{remark:technical}).
\parpic{
\setlength\fboxrule{0.1pt}
\fbox{\includegraphics[width=3.5cm, height=3.5cm]{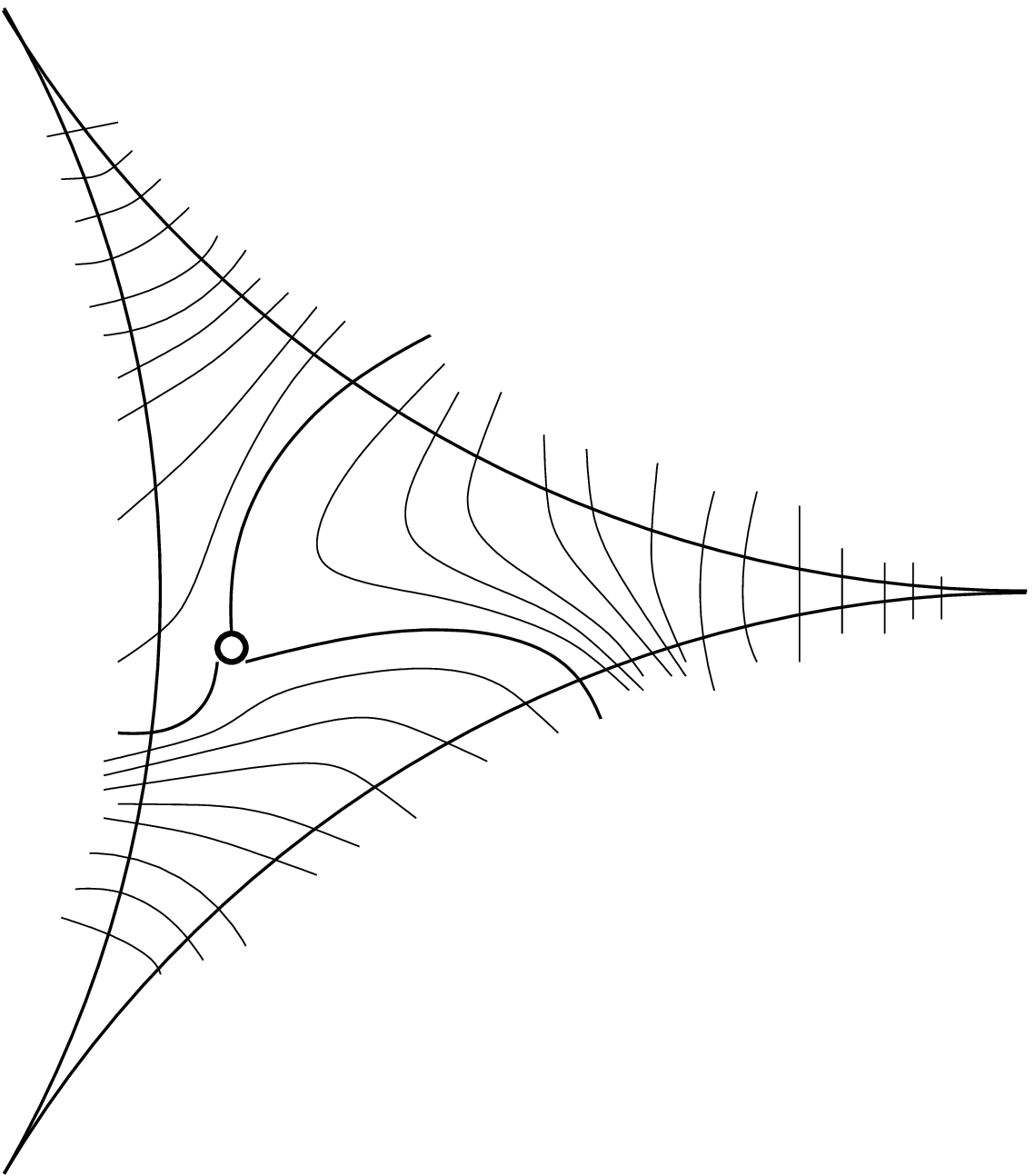}}
}Fix a representative of $F$ whose singular points all lie away from $\mu$. We still denote by $F$ this representative. Since $F$ is transverse to $\mu$ and $\mu$ is complete, the foliation $F$ has as many singular points as ideal triangles of $\Sigma\setminus\mu$. Furthermore, each one of these singular points are $3$-pronged, that is, there are three singular leaves emanating from each singularity and these three singular branches leave the ideal triangle containing the singularity through distinct sides. 

Using the base hyperbolic structure $h_{0}$, we lift the situation in the $h_{0}$-universal covering of $\Sigma$. 
Everything that has been previously said about shears with respect to $h$ and $F_{\mu}(h)$ carries over in the present situation: 
we replace the horocyclic foliation $F_{\mu}(h)$ by $F$, we consider ``horogeodesic curves" with respect to $\mu$ and $F$, i.e. paths which are concatenations of pieces contained in some leaves of $\widetilde{F}$ and of $\widetilde{\mu}$ alternately, and we use the transverse measure of $\widetilde{F}$ to define shear between a pair of ideal triangles in $\widetilde{\Sigma}$ as the signed transverse measure of a horogeodesic curve connecting the singular points of $\widetilde{F}$ inside the ideal triangles we want to determine the shear of.

The connection between both definitions of shears is made by Thurston's map, namely, denoting the hyperbolic structure $F_{\mu}^{-1}(F)$ by $h$, the horocyclic foliation $F_{\mu}(h)$ coincides with the equivalence class $F$.
A representative of $F$ can thus be chosen so that its preimage in the metric universal covering over $h$ coincides 
with the preimage of the horocyclic foliation $\widetilde{F}_{\mu}(h)$ 
(once the non-foliated regions have been collapsed onto tripods).
Evaluating then the shears between pairs of ideal triangles by using the metric universal covering over $h$ and the horocyclic foliation $\widetilde{F}_{\mu}(h)$ yields the shears evaluated through the measured foliation $\widetilde{F}$.\\

Given a complete geodesic lamination $\mu$ on the surface $\Sigma$ and given either a hyperbolic structure $h$ on $\Sigma$ (or, which is equivalent, a measure class $F$ of a foliation transverse to $\mu$), Thurston proved that the collection of shears between all the pairs of ideal triangles of $\widetilde{\Sigma}\setminus\widetilde{\mu}$ with respect to $h$ (or to $F$) characterizes the hyperbolic structure univoquely (see \cite{Thurston86} Proposition 4.1, where he uses slightly different objects he calls ``sharpness functions" which are intimately connected to shears). We shall also establish this fact, but we shall follow a route first leading to the notion of shear coordinates.

\section{Thurston's shear coordinates}
\label{section:shearcoordinates}

\subsection{Introduction}
\label{subsection:intro}

We describe in detail Thurston's shear coordinates on Teichm\"uller space which were introduced in \cite{Thurston86} (see also \cite{Pap91}).
Another description, in the same vein, of these coordinates can be found in \cite{Bonahon96}.

Since there are finitely many ideal triangles with disjoint interiors in a closed hyperbolic surface (of finite genus), it is intuitively clear that knowledge of shears between finitely many pairs of ideal triangles suffices for recovering all the other shears univoquely and eventually hyperbolic structure.  
One therefore needs a way of selecting finitely many shears on the surface, in large number enough so as to recover the other shears and the hyperbolic structure all together.
Building upon these remarks will first lead us to local shear coordinates and eventually to global shear coordinates over Teichm\"uller space.
But before entering this matter, we first quickly deal with the case where complete geodesic laminations have countably many leaves (finitely many, actually) for three reasons: 
firstly, this is the simplest situation, which directly leads to \emph{global} coordinates on Teichm\"uller space. 
Secondly, it provides good examples to see how shears work.  
Thirdly, this is the case that is mainly developed in literature (see \cite{PapThe08}, \cite{Fock07}, \cite{BBFS09}).

\subsection{Special case where $\mu$ has finitely many leaves}
\label{subsection:finite_lamination}

It turns out that a geodesic lamination which has countably many leaves necessarily has finitely many leaves.
So assume that $\mu$ has finitely many leaves.
It is then made up of finitely many simple closed geodesics along which isolated bi-infinite geodesics spiral.
In this case, except in a neighbourhood of the closed leaves, the ideal triangles are glued edge-to-edges, that is, one can talk about adjacent ideal triangles sharing a common side.
On such a common side lie two distinguished points, one for each adjacent ideal triangle.
The shear between them can thus be seized directly on the surface $\Sigma$ itself (without lifting to the universal cover) as the signed distance between the two distinguished points on the common side.

When the geodesic lamination $\mu$ has only isolated leaves (this requires the surface $\Sigma$ to have at least one puncture) these shears suffice for defining \emph{global} coordinates on Teichm\"uller space and for recovering all the other shears defined through the universal covering (see \cite{Fock07} for instance). 
These finitely many shears are actually slightly too numerous for recording all \emph{complete} hyperbolic structures of \emph{finite area} on the surface -- for a surface of genus $g$ with $p$ punctures, there are $6g-6+3p$ shears while Teichm\"uller space's dimension is $6g-6+2p$. However, one has to remember that a cusped shape is imposed to all punctures, yielding as many independent linear relations amongst the shears.

When the surface is closed, there is a closed leaf in $\mu$ necessarily.
There is no canonical way to choose a shear that would encode how shifted the ideal triangles are from one side of the closed leaf with respect to those from the other side.
One has to lift the situation to the universal covering once again and \emph{choose} a pair of ideal triangles separated by a lift of the closed leaf. 
This is what is done in \cite{BBFS09}.
Another way would have been to choose a homotopy class relative endpoints of a curve in the surface $\Sigma$ crossing the closed leaf and joining the centres of two ideal triangles. 
Anyway, the choice of one shear per closed curve component of $\mu$ together with canonical shears coming from adjacent ideal triangles yield \emph{global} coordinates on Teichm\"uller space and enable one to recover all other shears.

\subsection{General theory (with arbitrary $\mu$): Thurston's local shear coordinates}
\label{subsection:local_shear_general}

Let us give right away an overview of this general theory before entering into details.
We are aiming at parameterizing Teichm\"uller space by shear coodinates lying in a vector space of the same dimension as the dimension of Teichm\"uller space, namely $3\vert\chi(\Sigma)\vert$, where $\chi(\Sigma)$ is the Euler-Poincar\'e characteristic of the surface $\Sigma$.
To do this, me must find a way of selecting finitely many shears among those defined for each pair of ideal triangles of $\widetilde{\Sigma}\setminus\widetilde{\mu}$, where $\widetilde{\Sigma}$ is the universal covering of the surface $\Sigma$ and $\widetilde{\mu}$ is the preimage in $\widetilde{\Sigma}$ of a fixed complete geodesic lamination $\mu$.
We shall do this by using a thick train-track approximation $\Theta$ of $\mu$:
$\Theta$ is a finite simplification of the geodesic lamination $\mu$ and it enables to recover a situation where it is possible to talk about adjacent ideal triangles in the surface.
We shall thus get the finitely many shears we are looking for and we shall call them the distinguished shears.
They can be seen as a collection of real numbers, one for each branch of $\Theta$, satisfying some conditions, called switch conditions.
However, distinguished shears provide \emph{local} coordinates only, defined over an open subset $\mathcal{T}(\Theta)$ of Teichm\"uller space:
$$
\sigma_{\Theta}\, :\,  \mathcal{T}(\Theta)\to\mathbb{R}^{3\vert\chi(\Sigma)\vert}.
$$
It turns out that if $\Theta'$ is a finer train-track approximation of $\mu$, then $\mathcal{T}(\Theta)\subset\mathcal{T}(\Theta')$ and, as $\Theta$ converges to $\mu$, the subsets $\mathcal{T}(\Theta)$ converge to $\mathcal{T}(\Sigma)$.
In order to get \emph{global} shear coordinates, we transfer local shear coordinates $\mathcal{T}(\Theta)$ as $\Theta$ gets finer and finer to a \emph{fixed} train-track approximation $\Theta_{0}$ of $\mu$ in a coherent way.
Global shear coordinates are then a collection of real numbers, one for each branch of $\Theta_{0}$.

\subsubsection*{\bf \S 1. Foliations transverse to $\mu$ and foliations transverse to $\Theta$.} 
Let us fix a complete geodesic lamination $\mu$ of $\Sigma$.
We endow $\Sigma$ with a base hyperbolic structure $h_{0}$.
The complete lamination $\mu$ is straightened to a geodesic lamination with respect to the base structure $h_{0}$.
We fix a train-track approximation $\Theta=\Theta(\mu,h_{0})$ of $\mu$ as explained in Section \ref{subsection:thicktraintrack}.
For the sake of simplicity, we arrange $\Theta$ such that it is \textbf{generic}, that is, such that each singular traverse has exactly three branches abutting on. 
This possibly requires pushing inside the cusps of the train-track approximation slightly. 
Even though the boundary curves of the train-track approximation are then slightly deformed, this assumption of genericity has the only unimportant consequence that the train-track approximation is not exactly an $\epsilon$-neighbourhood of $\mu$ anymore. 

Let $F$ be an element of $\mathcal{MF}(\mu)$. 
This measure class $F$ has representatives which are transverse to $\mu$. 
Moreover, a representative, also denoted by $F$, can be chosen so that each of its leaves is transverse to the boundary $\partial\Theta$ of $\Theta$, meaning that each of its leaves does not bound with $\partial\Theta$ a disc inside $\Theta$, and so that its singularities lie outside $\Theta$, one in each component of $\Sigma\setminus\Theta$ (see, for a visual argument, the preceding picture). 
The restriction $F_{\vert\Theta}$ to $\Theta$ of this representative $F$ is called a \textbf{measured foliation} of $\Theta$. 
Two measured foliations of $\Theta$ are equivalent if they are isotopic through an isotopy which respects the transverse measures and keeps the ends of the leaves on $\partial\Theta$. 
The class $F$ thus induces an equivalence class of foliations of $\Theta$. 
We call such a class a measure class of foliations of $\Theta$.

Conversely, every measured foliation of $\Theta$ can be extended to a measured foliation of $\Sigma$ transverse to $\mu$ which is \emph{unique} up to isotopy supported outside $\Theta$. 
We can thus safely say that restriction induces a one-to-one correspondence between $\mathcal{MF}(\mu)$ and the measure classes of foliations of $\Theta$ and therefore we can use classes of foliations of a fixed train-track approximation $\Theta$ of $\mu$ to characterize hyperbolic structures on $\Sigma$ rather than using foliations transverse to $\mu$ defined over the whole surface $\Sigma$.

Let $F$ be a measured foliation of $\Theta$ (or, of $\Sigma$ which is transverse to $\mu$). 
It might be transverse or not transverse to $\Theta$ in the following sense. 
A measure class of foliations in $\Theta$ (or in $\mathcal{MF}(\mu)$) is {\bf transverse} to $\Theta$ if there exists a representative of $F$ which (or whose restriction to $\Theta$) coincides with the standard foliation of $\Theta$ by traverses; the transverse measures of both foliations do not necessarily match, though.
Another way of saying is that $F$ is transverse to $\Theta$ if the transverse measure of $F$ induces on every branches of $\Theta$ a \emph{positive} weight (see Figure \ref{figconvex9}). 
\begin{figure}[h!]
\centering
\psfrag{T}{\tiny $\Theta$}
\psfrag{T'}{\tiny $\Theta'$}
\psfrag{F}{\tiny $F$}
\includegraphics[width=.8\linewidth]{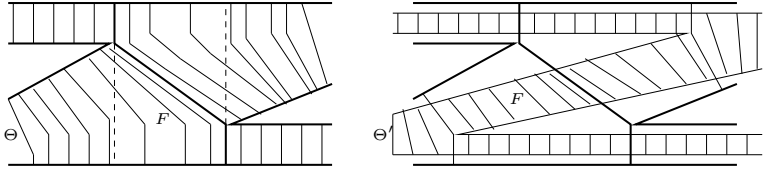}
\caption{(A schematic picture.) The foliation $F$ of $\Theta$ is \emph{not} transverse to $\Theta$, i.e., $F\notin\mathcal{MF}(\Theta)$ since it induces no weight on the central branch bordered by dotted lines. However, $F\in\mathcal{MF}(\Theta')$, where $\Theta'$ is a finer train-track approximation of $\mu$.}
\label{figconvex9}
\end{figure}

\subsubsection*{\bf \S 2. Localization to train-tracks.}
Let $\mathcal{MF}(\Theta)$ be the space of measure classes of foliations of $\Sigma$ that are transverse to $\Theta$.
We set $\mathcal{T}(\Theta):=F_{\mu}^{-1}(\mathcal{MF}(\Theta))$; this is an open subset of $\mathcal{T}(\Sigma)$.
We therefore have the following diagram:

\begin{center}
\begin{pspicture}(-1,-1)(4,3.4)
\begin{psmatrix}[nodesep=3pt]
$\mathcal{T}(\Theta)$ & $\mathcal{MF}(\Theta)$  \\
$\mathcal{T}(\Sigma)$ & $\mathcal{MF}(\mu)$ 
\ncline{->}{1,1}{1,2}^{$F_{\mu\vert\mathcal{T}(\Theta)}$}
\ncline{->}{1,2}{2,2}
\ncline{->}{1,1}{2,1}
\ncline{->}{2,1}{2,2}_{$F_{\mu}$}
\end{psmatrix}
\end{pspicture}
\end{center}

Consider a finer thick train-track approximation $\Theta'$ of $\mu$.
The train-track approximation $\Theta'$ is contained in $\Theta$ and the foliation of $\Theta$ by traverses, 
restricted to $\Theta'$, yields the foliation of $\Theta'$ by traverses.
Up to isotopy, the train-track approximation $\Theta'$ is obtained from $\Theta$ by unzipping along finitely many arcs,
each starting from a cusp of $\Theta$ and running transversely to the traverses of $\Theta$.
This unzipping process is combinatorially equivalent to saying that the train-track approximation $\Theta'$ is obtained from $\Theta$,
up to isotopy, by performing finitely many moves called \textbf{left} and \textbf{right splittings} along branches, as explained in \cite{Penner92} and as pictured on the top of Figure \ref{figshear5}. 

\begin{figure}
\centering
\psfrag{R}{$R$}
\psfrag{L}{$L$}
\psfrag{a}{$a$}
\psfrag{b}{$b$}
\psfrag{c}{$c$}
\psfrag{d}{$d$}
\psfrag{e}{$e$}
\psfrag{e'}{$e'$}
\includegraphics[width=.6\linewidth]{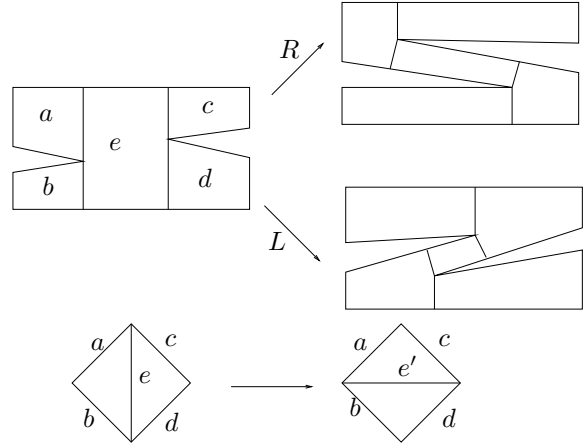}
\caption{The picture shows a right and a left splitting of a thick train-track $\Theta$. Translated into the graph $\Delta_{\Theta}$, these splittings become a flip of a diagonal.}
\label{figshear5}
\end{figure}

The open subset $\mathcal{T}(\Theta')$ contains the subset $\mathcal{T}(\Theta)$:
this is due to the fact that the train-track approximation $\Theta'$ is contained in $\Theta$ and that
a measured foliation which is transverse to $\Theta$ can therefore be arranged to be transverse to the finer train-track approximation $\Theta'$ (see Figure \ref{figconvex9} again). 
Thus $\mathcal{MF}(\Theta)$ is contained in $\mathcal{MF}(\Theta')$ and it follows,
by applying the map $F_{\mu}^{-1}$, that $\mathcal{T}(\Theta)$ is contained in $\mathcal{T}(\Theta')$, as claimed above.

If one lets the thick train-track approximations $\Theta$ converge towards $\mu$ in the sense of Hausdorff topology on $\Sigma$, 
one gets a sequence of wider and wider open subsets of $\mathcal{T}(\Sigma)$.
At the limit, the open subsets converge to the whole Teichm\"uller space.

\subsubsection*{\bf \S 3. Distinguished shears and signed transverse measures on $\Theta$.}
Let $F$ be a measured foliation transverse to $\Theta$.
The singular points of the foliation $F$ are trivalent and each of them is contained in one component of $\Sigma\setminus\Theta$.
Conversely, every component of $\Sigma\setminus\Theta$ contains exactly one such singular point.

Consider a branch $b$ of $\Theta$.
This branch is adjacent to two components of $\Sigma\setminus\Theta$, both having the shape of a triangle.
Note that these components might coincide in $\Sigma$.
Each of these triangle-shaped regions contains a (trivalent) singular point of the foliation $F$. 
Consider lifts of the ideal triangles of $\Sigma\setminus\mu$ which contain the two singular points such that these lifts are separated by a lift of the branch $b$ only.
The shear, $\sigma_{b}$, associated to the branch $b$ is then the shear between these lifts with respect to $F$, in the sense explained before. 
This shear is also the signed transverse measure with respect to $F$ of a horogeodesic curve joining the two singular points.

Let $B$ be the set of branches of $\Theta$.
We get from the above description a collection $\{\sigma_{b}\}_{b\in B}$ of \textbf{distinguished shears} which will serve as coordinates on the open patch $\mathcal{T}(\Theta)$ of Teichm\"uller space $\mathcal{T}(\Sigma)$. 

The shears $\{\sigma_{b}\}_{b\in B}$ satisfy the \textbf{switch condition} at each  singular traverse of $\Theta$, namely,
if one arranges $\Theta$ to be generic (i.e., three branches abut on every singular traverse), then, denoting by $b$ and $c$ the two incoming branches and by $a$ the outgoing one, we have $\sigma_{a}=\sigma_{b}+\sigma_{c}$ (see Figure \ref{figshear2}). 
\begin{figure}
\psfrag{se}{$\sigma_{a}$}
\psfrag{sf}{$\sigma_{b}$}
\psfrag{sg}{$\sigma_{c}$}
\psfrag{e}{$a$}
\psfrag{f}{$b$}
\psfrag{g}{$c$}
\centering
\includegraphics[width=.5\linewidth]{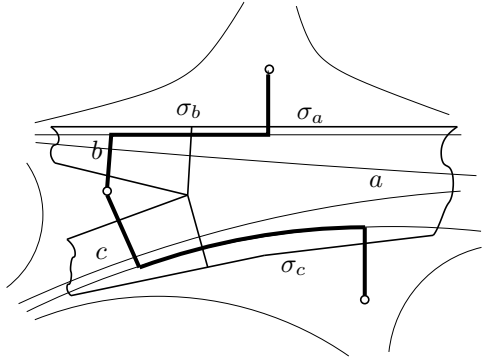}
\caption{The switch condition $\sigma_{a}=\sigma_{b}+\sigma_{c}$. 
Three centres of ideal triangles are represented by hollow dots. The shear between the uppermost and the lowermost ideal triangles is the number $\sigma_{a}$. It represents the shear associated to the outgoing branch $a$ (here, $\sigma_{a}>0$). The shear between the uppermost and the ideal triangle inbetween is the number $\sigma_{b}<0$ and the shear between the lowermost and the ideal triangle inbetween is the number $\sigma_{c}>0$. They represent the shears of the two incoming branches $b$ et $c$ respectively. We have represented horogeodesic curves joining various centres. The shears are the signed lengths of the geodesic parts of these horogeodesic curves. The two outermost ideal triangles are separated the other one. By Lemma \ref{lemma:separation}, we get $\sigma_{a}=\sigma_{b}+\sigma_{c}$.}
\label{figshear2}
\end{figure}
The shears $\{\sigma_{b}\}_{b\in B}$ can thus be viewed as a \emph{signed} transverse measure on $\Theta$, hence as a \emph{signed transverse measure} on $\mu$ itself (see \cite{Bonahon96}).

\subsubsection*{\bf \S 4. Distinguished shears and twisted cocycles.}

Let us now introduce the dual graph $\Delta_{\Theta}$ of the train-track approximation $\Theta$ of $\mu$ (or, maybe better to say, of any thin train-track $\theta$ obtained by collapsing $\Theta$ along its traverses).
There is exactly one pair of ideal triangles of $\Sigma\setminus\mu$ for each branch of the train-track $\Theta$.
The vertices of the graph $\Delta_{\Theta}$ can be taken to be the centres of the ideal triangles of $\mu$. 
The graph $\Delta_{\Theta}$ triangulates the surface $\Sigma$.
(Recall that the train-track approximation $\Theta$ has been arranged to be generic.)

Local shear coordinates can also be described as real weights on the branches of $\Delta_{\Theta}$, taken up to switch conditions.
The graph $\Delta_{\Theta}$ cannot be oriented in a coherent way such that the switch condition translates to as a cohomological condition (see Figure \ref{convex13}): but if we take the (singular two-fold) orientation covering of $\Sigma$ for which the preimage of $\Delta_{\Theta}$ is oriented accordingly to switch conditions, then shear coordinates can be seen as a cocycle with real coefficients up there.
Down to $\Sigma$, shear coordinates are then only to be seen as real cocycles twisted by local orientation, as it is also explained in \cite{Bonahon96}.

\begin{figure}
\psfrag{a}{$a$}
\psfrag{b}{$b$}
\psfrag{c}{$c$}
\psfrag{A}{$\sigma_{a}$}
\psfrag{B}{$\sigma_{b}$}
\psfrag{C}{$\sigma_{c}$}
\centering
\includegraphics[width=.3\linewidth]{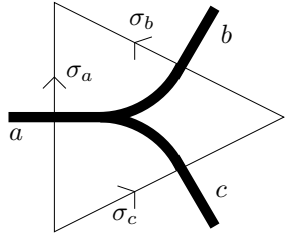}
\caption{Local picture of a train-track $\Theta$ about a trivalent switch (in thick lines) and of the associated triangle contained in the dual triangulation $\Delta_{\Theta}$. If we choose to orient the incoming branch $a$ in the direction of the switch and both outgoing branches $b$ and $c$ in the direction opposite to the switch, the edges of $\Delta_{\Theta}$ get a natural orientation (so that the directions of the branch of $\Theta$ and of the edge of $\Delta_{\Theta}$ associated to it form a positive basis for some fixed orientation on the surface). The switch condition $\sigma_{a}=\sigma_{b}+\sigma_{c}$ rephrases to as a local cohomological condition.}
\label{convex13}
\end{figure}

The number of vertices of $\Delta_{\Theta}$ is $2\vert\chi(\Sigma)\vert$, the number of edges is $\vert B\vert=9\vert\chi(\Sigma)\vert$ and the number of faces is $6\vert\chi(\Sigma)\vert$.
Each face of the triangulation corresponds bijectively with the singular traverses of $\Theta$. 
Hence there are as many switch conditions as faces of $\Delta_{\Theta}$.
A bit of linear algebra juggling easily establishes that these $6\vert\chi(\Sigma)\vert$ are independent. 
The dimension of the space of local shear coordinates is thus $3\vert\chi(\Sigma)\vert$, the dimension of the Teichm\"uller space of $\Sigma$.

\subsubsection*{\bf \S 5. Local shear coordinates}

Let a complete geodesic lamination $\mu$ of the surface $\Sigma$ (equiped with some base hyperbolic structure $h_{0}$) and a thick train-track approximation $\Theta$ of $\mu$ be given.
We have just seen that shears taken up to switch conditions generate a vector space $\textrm{T}(\Theta)$ of dimension $3\vert\chi(\Sigma)\vert$, where $\chi(\Sigma)$ denotes the Euler-Poincar\'e characteristic of $\Sigma$.
An open subset of this vector space $\textrm{T}(\Theta)$ serves as a domain for \emph{local} coordinates on $\mathcal{T}(\Sigma)$, parameterizing $\mathcal{T}(\Theta)$:
$$
\sigma_{\Theta}\,:\,\mathcal{T}(\Theta)\longrightarrow\textrm{T}(\Theta).
$$
These coordinates are called Thurston's local \textbf{shear coordinates} with respect to the train-track approximation $\Theta$ of the complete geodesic lamination $\mu$. (Thurston also called them cataclysm coordinates.)

The image of $\sigma_{\Theta}$ in $\textrm{T}(\Theta)$ is a cone. 
The open set $\mathcal{T}(\Theta)$ is the set of hyperbolic structures $h$ on $\Sigma$ whose horocyclic foliations (transfered to $\Sigma, h_{0}$) are transverse to $\Theta$. 
The coefficients of the vector $\sigma_{\Theta}(h)$ are the distinguished shears modulo switch relations between adjacent ideal triangles with respect to $\Theta$. 

As we shall see just below, the shear between any pair of ideal triangles for the structure $h$ is an integral linear combination of the distinguished shears $\sigma_{\Theta}(h)$. 
Furthermore, the collection of shears between pairs of ideal triangles characterizes in turn the hyperbolic structure $h$ univoquely (see Section \ref{section:developing}).

\subsubsection*{\bf \S 6. Shears are linear combinations of distinguished shears.}

We proceed on here with showing that distinguished shears $\{\sigma_{b}\}_{b\in B}$ given by a choice of some thick train-track approximation $\Theta$ of $\mu$ uniquely determine the shear between
any pair of ideal triangles of $\widetilde{\Sigma}\setminus\widetilde{\mu}$ 
(Recall that the surface is endowed with the base structure $h_{0}$).
More precisely we show, if $\sigma$ is the shear between a pair of ideal triangles $T$ and $T'$ of $\widetilde{\Sigma}\setminus\widetilde{\mu}$, that $\sigma$ can be expressed as a linear combination of distinguished shears $\{\sigma_{b}\}_{b\in B}$.

So let $\sigma$ be the shear between some pair of ideal triangles $T$ and $T'$ of $\widetilde{\Sigma}\setminus\widetilde{\mu}$.
Consider the preimage $\widetilde{\Delta}_{\Theta}$ of the graph $\Delta_{\Theta}$ in the universal covering $\widetilde{\Sigma}$.
Connect the centres of the ideal triangles $T$ and $T'$ with a 
horogeodesic curve $\gamma$.
If necessary, deform each component of $\gamma\setminus\Theta$ so that it passes through the centre of the triangle of $\widetilde{\Sigma}\setminus\Theta$ that contains it:
this results in a new horogeodesic curve we still denote by $\gamma$ passing through the centres $t=t_{1},t_{2},\ldots,t_{k}=t'$ of ideal triangles in its way from $t$ to $t'$.

For $i=1,\ldots,k-1$, let $\sigma_{i,i+1}$ be the shears between the ideal triangles with centres $t_{i},t_{i+1}$.
Since $\gamma$ is horogeodesic, we have by Lemma \ref{lemma:separation}
$$
\sigma=\sum_{i=1}^{k-1}\sigma_{i,i+1}.
$$ 
Two cases then occur: 
\begin{itemize}
\item either $t_{i},t_{i+1}$ are the centres of two adjacent ideal triangles with respect to $\Theta$, which means that the two components of $\widetilde{\Sigma}\setminus\Theta$ that contain these centres are adjacent to one branch $b_{i}$ of $\Theta$;
\item or $t_{i},t_{i+1}$ are the centres of two non-adjacent ideal triangles with respect to $\Theta$.
\end{itemize}
In the first case, the shear $\sigma_{i,i+1}$ equals by definition the distinguished shear $\sigma_{b_{i}}$ associated to the branch $b_{i}$ of $\Theta$. 
It therefore remains to focus on the second case. 
We can assume without loss of generality that our horogeodesic curve $\gamma$ joins the centres $t,t'$ of two non-adjacent ideal triangles without passing through other centres.

The curve $\gamma$ passes through a sequence of branches $b_{1},\ldots,b_{n}$ (in this order) of $\widetilde{\Theta}$ in its way from $t$ to $t'$.
Let $\delta_{1},\ldots,\delta_{n}$ be the edges of $\widetilde{\Delta}_{\Theta}$ associated to $b_{1},\ldots,b_{n}$ respectively.
The union $\cup_{1\leq j\leq n}\delta_{j}$ forms a connected curve in $\widetilde{\Delta}_{\Theta}$;
let $\delta$ be a continuous path 
$$
\delta\,:\,[0;1]\to\cup_{1\leq j\leq n}\delta_{j}
$$
whose image is this union of edges $\cup_{j}\delta_{j}$ such that $\delta(0)=t$, $\delta(1)=t'$, $\delta$ is injective in the interior of each edge $\delta_{j}$ and crosses at most twice the same edge.
Let $(\sigma_{k})_{1\leq k\leq m}$ ($m\geq n$) be the sequence of shears associated to $\delta$, ordered according to the orientation of the path.
Here $m\geq n$ since $\delta$ can pass successively twice through the same branch of $\widetilde{\Theta}$ (see the picture below).

\begin{lemma}
With the above notation, we have
$$
\sigma=\sum_{k=1}^{m}(-1)^{k+1}\sigma_{k}.
$$
\end{lemma}

\begin{proof}
Consider the strip $S$ in the universal covering bounded by the ideal triangles $T$ and $T'$ with centres $t,t'$ respectively.
Keep the leaves of $\widetilde{\mu}$ that separate $T$ and $T'$ only and foliate the strip $S$ with its horocyclic foliation $F$.
The horogeodesic curve $\gamma$ chosen to join $t$ to $t'$ can be deformed within the strip $S$ into a horogeodesic curve still denoted by $\gamma$ so that $\gamma$ is the concatenation of three curves: firstly a curve starting from $t$ and contained in a leaf of $F$; secondly, a geodesic segment we denote by $[uv]$, entirely contained in one leaf of $\widetilde{\mu}$; thirdly, a curve contained in a leaf of $F$ and ending at $t'$.
Note that $\gamma$ crosses all edges $\delta_{1},\ldots,\delta_{n}$ transversely.

The shear $\sigma$ between $t$ and $t'$ is the measure of $\gamma$ with respect to the signed transverse measure on $\widetilde{\Theta}$ (or equivalently, the signed length of the geodesic part $[uv]$ of $\gamma$).
Let $x_{0},\ldots,x_{m+1}$ be the intersection points between $\delta$ and $\gamma$, taken in order (hence $x_{0}=t$ and $x_{m+1}=t'$. We might have $x_{k}=x_{k+1}$ whenever the corresponding edge of $\Delta_{\Theta}$ is traversed twice).

On the one hand, for all $k\in\{1,\ldots,m-1\}$, the signed transverse measure $\tau_{k}$ of the subcurve $x_{k}x_{k+1}$ in the image of $\delta$ is equal to the signed length of the geodesic segment $[x_{k}x_{k+1}]$ which is contained in $\gamma$.
For $k=0$ and $k=m$, the signed transverse measures $\tau_{0}$ and $\tau_{m}$ of the subcurve $x_{0}x_{1}$ and $x_{m}x_{m+1}$ respectively are equal to the signed length of the geodesic segments $[ux_{1}]$ and $[x_{m}v]$ which are contained in $\gamma$ as well.
Since the signed transverse measure of $\gamma$ is the signed length of $[uv]=[ux_{1}]\cup[x_{1}x_{2}]\cup\ldots\cup[x_{m-1}x_{m}]\cup[x_{m}v]$, we have
$$
\sigma=\sum_{k=0}^{m}\tau_{k}.
$$

On the other hand, the image of $\delta$ is equal to $\cup_{k=0}^{m}x_{k}x_{k+1}$, so the sum $\sum_{k=0}^{m}\tau_{k}$ is the signed transverse measure of the image curve of $\delta$ computed with the weights on $\widetilde{\Theta}$.
Now since the transverse measure of each edge $\delta_{j}$ ($1\leq j\leq n$) equals $\sigma_{k}$ (for some $k\in\{0,\ldots,m\}$),  we get
$$
\sigma=\sum_{k=0}^{m}\tau_{k}=\sum_{k=0}^{m}\epsilon_{k}\sigma_{k},
$$
where $\epsilon_{k}\in\{-1;+1\}$.

Notice that the path $\delta$ crosses $\gamma$ transversely in every branch of $\widetilde{\Theta}$ and so alternates from one side of $\gamma$ to the other.
Therefore, almost all wedges of $\widetilde{\mu}$ crossed by $\delta_{j}$ are crossed again by $\delta_{j+1}$.
Lemma \ref{lemma:separation} implies that the total shear along the path $\delta_{j}\cup\delta_{j+1}$ is not the sum but the difference of $\sigma_{j}$ and $\sigma_{j+1}$.
Hence
$$
\sigma=\sum_{k=1}^{m}(-1)^{k+1}\sigma_{k}.
$$
This concludes the proof of the lemma.
\end{proof}

The lemma above settles the claim that shears between pairs of ideal triangles are integral linear combinations of distinguished shears.
Moreover, the above formula is well-defined because it does not depend upon the choice of the horogeodesic curve $\gamma$:
to see this, recall that $\gamma$, being horogeodesic, crosses a uniquely defined sequence of ideal triangles without turning back.
Therefore, the sequence of vertices of the graph $\Delta_{\Theta}$ associated to $\gamma$ in the proof above is univoquely determined.

\subsection{Thurston's global shear coordinates}
\label{subsection:global_shear}

We briefly explain in this section how Thurston gets global shear coordinates on Teichm\"uller space.

Let $\mu$ be a complete geodesic lamination of the surface $\Sigma$.
For a given thick train-track approximation $\Theta$ of $\mu$, we have shear coordinates $\sigma_{\Theta}\,:\,\mathcal{T}(\Theta)\longrightarrow\textrm{T}(\Theta)$ defined on the open subset $\mathcal{T}(\Theta)$ of $\mathcal{T}(\Sigma)$.

Let $\Theta'$ be a finer train-track approximation of $\mu$.
We have also local shear coordinates $\sigma_{\Theta'}\,:\,\mathcal{T}(\Theta')\longrightarrow\textrm{T}(\Theta')$ defined on the open subset $\mathcal{T}(\Theta')$ of $\mathcal{T}(\Sigma)$.
Since $\Theta'$ is finer than $\Theta$, we have $\mathcal{T}(\Theta')\subset\mathcal{T}(\Theta')$.

There is a natural bijective linear map from $\textrm{T}(\Theta')$ to $\textrm{T}(\Theta)$ induced by shear coordinates defined on $\mathcal{T}(\Theta')$ and on $\mathcal{T}(\Theta)$.
One way of obtaining this bijective linear map is first by choosing a finite sequence of train-track approximations $\Theta=\Theta_{0},\Theta_{1},\cdots,\Theta_{m}=\Theta'$
such that $\Theta_{j+1}$ is obtained from $\Theta_{j}$ by one (left or right) splitting. 
We notice that the splitting along a branch $e$ of $\Theta_{j}$ amounts to flipping the edge $\delta_{e}$ of the graph $\Delta_{\Theta_{j}}$ that crosses $e$, see Figure \ref{figshear5}.
The linear map sending shears defined on $\Theta_{j}$ to those defined on $\Theta_{j+1}$ can then be easily written down: it is the identity except for the central branch $e'$ where we have, using notation of Figure \ref{figshear5}, 
$\sigma_{e'}=\pm(\sigma_{b}-\sigma_{a})$, the minus sign being for right splitting. 
By passing to quotient linear subspaces, that is, by taking shears up to switch relations, 
we get an injective linear map from $\textrm{T}(\Theta_{j})$ to $\textrm{T}(\Theta_{j+1})$.
The dimensions of the linear spaces being equal, 
the above injective map is bijective and we consider it as defined over the whole vector space $\textrm{T}(\Theta_{j})$.

We denote by $p_{j,j+1}\ :\ \textrm{T}(\Theta_{j+1})\longrightarrow\textrm{T}(\Theta_{j})$ the \emph{inverse} of the above bijective linear map and we call it the \textbf{transport map} from $\Theta_{j+1}$ to $\Theta_{j}$.
The transport map is used to transport genuine shear coordinates defined on $\Theta_{j+1}$ to well-defined numbers on branches of $\Theta_{j}$ satisfying switch relations.
The transport map $p_{\Theta,\Theta'}\ :\ \textrm{T}(\Theta')\longrightarrow\textrm{T}(\Theta)$ from $\Theta'=\Theta_{m}$ to $\Theta=\Theta_{0}$ is defined as the composition of the transport maps,
namely, $p_{\Theta,\Theta'}=p_{0,1}\circ\cdots\circ p_{m-1,m}$.

We thus have the following diagram:
\begin{center}
\begin{pspicture}(-1,-1)(4,3.3)
\begin{psmatrix}[nodesep=3pt]
$\mathcal{T}(\Theta')$ & $\textrm{T}(\Theta')$  \\
$\mathcal{T}(\Theta)$ & $\textrm{T}(\Theta')$ 
\ncline{->}{1,1}{1,2}^{$\sigma_{\Theta'}$}
\ncline{->}{1,2}{2,2} \Aput{$p_{\Theta,\Theta'}$}
\ncline{<-}{1,1}{2,1}
\ncline{->}{2,1}{2,2}_{$\sigma_{\Theta}$}
\end{psmatrix}
\end{pspicture}
\end{center}

We have to check that the transport map $p_{\Theta,\Theta'}$ does not depend upon the chosen sequence of train-track approximations $\Theta=\Theta_{0}\supset\Theta_{1}\supset\cdots\supset\Theta_{m}=\Theta'$.
Consider two transport maps $p\ :\ \textrm{T}(\Theta')\longrightarrow\textrm{T}(\Theta)$ and $q\ :\ \textrm{T}(\Theta')\longrightarrow\textrm{T}(\Theta)$ from $\Theta'$ to $\Theta$.
It suffices to show that $q\circ p^{-1}$ is the identity over $\textrm{T}(\Theta)$.
Consider a point $h\in\mathcal{T}(\Theta)$ and the associated shear coordinates as a point $M_{h}$ in $\textrm{T}(\Theta)$.
The image of $M_{h}$ by $p^{-1}$ gives genuine shear coordinates on $\Theta'$.
The image of $p^{-1}(M_{h})$ by $q$ is the collection of shears, 
among those between the ideal triangles with respect to $h$, 
that are associated to the branches of $\Theta$, 
so we obviously get the shear coordinates we started with back.
The map $q\circ p^{-1}$ is therefore the identity over the subset of $\textrm{T}(\Theta)$ parameterizing $\mathcal{T}(\Theta)$.
Since this subset is open, the linear map $q\circ p^{-1}$ is the identity over the whole linear space $\textrm{T}(\Theta)$.
We thus have shown that the local shear coordinates can be univoquely transported, 
via transport maps, to coordinates defined on a fixed train-track $\Theta$.

If one lets the thick train-track approximations $\Theta$ converge towards $\mu$ in the sense of Hausdorff topology on $\Sigma$, 
one gets a sequence of shear coordinates,
each defined on wider and wider open subsets of $\mathcal{T}(\Sigma)$, together with transport maps relating shear coordinates on wider subsets with shear coordinates defined on smaller subsets.
At the limit, the open subsets converge to the whole Teichm\"uller space.
Transport maps give a way of transporting any shear coordinates to a fixed train-track approximation $\Theta$.
This enables us to define global coordinates on $\mathcal{T}(\Sigma)$ called (global) shear coordinates as well.
They give a parameterization of $\mathcal{T}(\Sigma)$ by a vector space of dimension $3\vert\chi(\Sigma)\vert$,
since switch relations are satisfied 
(we assume here the surjectivity of the limit coordinate map).
We denote this global parameterization by $\widetilde{\sigma}_{\Theta}$ from $\mathcal{T}(\Sigma)$ to $\mathbb{R}^{3\vert\chi(\Sigma)\vert}$.
With above notation, the sequence of local coordinates $\sigma_{\Theta_{n}}\,:\,\mathcal{T}(\Theta_{n})\to\textrm{T}(\Theta_{n})$ exhausts $\mathcal{T}(\Sigma)$, that is, for each point $h\in$ $\mathcal{T}(\Sigma)$, there exist $N$ such that, for all $n\geq N$, $h\in\mathcal{T}_{\Theta_{n}}$ and
$\widetilde{\sigma}_{\Theta}(h)=p_{n}\circ\sigma_{\Theta_{n}}(h)$.

\section{Shears determine hyperbolic structures}
\label{section:developing}

Let $\mu$ be a complete geodesic lamination of the surface $\Sigma$ and let $h_{0}$ be a base hyperbolic structure on $\Sigma$.
Fix a thick train-track approximation $\Theta$ of $\mu$ and let $\Delta_{\Theta}$ be the graph associated to $\Theta$ as defined above.
We saw in Section \ref{subsection:local_shear_general} that shears between pairs of ideal triangles can be expressed as linear combinations of distinguished shears viewed as weights on $\Delta_{\Theta}$.
In this section, we quickly recall how the collection of shears between pairs of ideal triangles characterize uniquely hyperbolic structures $h$ on $\Sigma$.

The idea is to build the developing map $D$ from the universal cover $\widetilde{\Sigma}$, identified with the set of homotopy classes relative endpoints of paths all starting from the same basepoint, and the hyperbolic plane. 
The holonomy map $H$ from $\pi_{1}(\Sigma)$ to the group of isometries of the hyperbolic plane has also to be built.
Since $\Delta_{\Theta}$ cuts the surface $\Sigma$ into triangles, we have an isomorphism between $\pi_{1}(\Sigma)$ and $\pi_{1}(\Delta_{\Theta})$.
The developing and the holonomy maps hence have only to be built along each edge of the graph $\Delta_{\Theta}$, that is, for a path passing through a branch of $\Theta$.
If we look at the situation in the universal cover identified, through the base hyperbolic structure $h_{0}$, with the hyperbolic plane, this amounts to considering a strip $S^0$ cut into ideal triangles by a geodesic lamination $\mu$ (with empty interior).

So consider a base strip $S^0$ of the hyperbolic plane and a geodesic lamination $\mu$ of $S^0$ with empty interior such that $S^0\setminus\mu$ consists in interiors of ideal triangles.
It is clear that shears between pairs of ideal triangles that do not both cross the core $\alpha_{core}$ of $S^0$ contribute nothing in the hyperbolic structure of $S^0$.
We can therefore forget about leaves of $\mu$ that do not cross $\alpha_{core}$ and keep up with the wedge cut $\check{\mu}$ induced by $\mu$, made up of the leaves of $\mu$ that do cross the core $\alpha_{core}$.
Let $\mathcal{W}$ denote the set of wedges of $S^0\setminus\check{\mu}$.
The set $\mathcal{W}$ is countable.

We fix an orientation of the core $\alpha_{core}$ of $S^0$.
By using this orientation, we index $\mathcal{W}$ with a bijection $I\to\mathcal{W}$, $i\mapsto w_i$, where the set of indexes $I$ is countable and totally ordered.
We also orient the leaves of the wedge cut $\check{\mu}$ positively toward the left of $\alpha_{core}$.

Note that a wedge has two distinguished points, one per side.
The shear between two wedges of $\check{\mu}$ is by definition the shear between the corresponding ideal triangles of $\mu$.
The set of unordered pairs of distinct wedges is denoted by $\underline{\mathcal{W}}^2=(\mathcal{W}\times\mathcal{W}-\Delta)/\mathbb{Z}_{2}$, where $\Delta$ is the diagonal of $\mathcal{W}\times\mathcal{W}$ and $\mathbb{Z}_{2}$ acts by permutation.
A collection of shears for pairs of wedges of $\check{\mu}$ is represented by a point in $\mathbb{R}^{\underline{\mathcal{W}}^2}$, that is, a function defined on unordered pairs of distinct wedges with values in $\mathbb{R}$.
The latter linear space has possibly infinite dimension.

Let $\mathcal{T}_{\check{\mu}}(S^0)$ be the Teichm\"uller space of $S^0$ relative to $\check{\mu}$, that is, 
$$
\mathcal{T}_{\check{\mu}}(S^0):=\left\{(S',\check{\mu}',f')\,\vert\,f'\,:S^0\to S'\ \textrm{orientation-preserving homeomorphism},\ f'(\check{\mu})\simeq\check{\mu}'\right\}/\sim
$$
where $(S',\check{\mu}',f')\sim(S'',\check{\mu}'',f'')$ if there exists an isometry $g$ of $\mathbb{H}^2$ with $f''\circ f'^{-1}\simeq g$.
The need of considering relative Teichm\"uller space stems from the fact that there are infinitely many different ways of gluing a collection of wedges and obtaining isometric strips. 
In $\mathcal{T}_{\check{\mu}}(S^0)$, these isometric strips obtained by gluing are distinct and $\check{\mu}$ acts as a marking in the classical theory.

There is thus a map $\sigma_{\check{\mu}}\,:\,\mathcal{T}_{\check{\mu}}(S^0)\to\mathbb{R}^{\underline{\mathcal{W}}^2}$.
It is intuitively clear that this map is not surjective for, in the vicinity of the origin of $\mathbb{R}^{\underline{\mathcal{W}}^2}$ for instance, there are shears that would give strips with core curves of infinite length.
Thurston proved that $\sigma_{\check{\mu}}$ is a homeomorphism onto its image which is an open subset of $\mathbb{R}^{\underline{\mathcal{W}}^2}$.
He gave the inverse map $\sigma_{\check{\mu}}^{-1}\,:\,\mathbb{T}_{\check{\mu}}(S^0)\to\mathcal{T}_{\check{\mu}}(S^0)$, where $\mathbb{T}_{\check{\mu}}(S^0)=\sigma_{\check{\mu}}(\mathcal{T}_{\check{\mu}}(S^0))\subset\mathbb{R}^{\underline{\mathcal{W}}^2}$, by constructing, given a collection of shears $\sigma_{h}$ in $\mathbb{T}_{\check{\mu}}(S^0)$, the developing and holonomy maps by finite approximations.
We very quickly recall how this is done.

Equip the strip $S^0$ with its horocyclic foliation $F_{\check{\mu}}(S^0)$ and fix a leaf $f_{0}$ of that foliation.
The leaf $f_{0}$ joins both sides of $S^0$ and its endpoints are well-defined, whatever the hyperbolic structure on $S^0$ is, by their signed distances from the distinguished points lying on both sides of $S^0$. 
Give $f_{0}$ the same orientation as the core $\alpha_{core}$.
The developing map is constructed along $f_{0}$.

A {\bf finite approximation of a wedge cut} $\check{\mu}$ of $S^0$ is a sequence of \emph{finite} subsets $(I_{n})$ of indices of $I$ such that, for all $n$,
$I_{n}\subset I_{n+1}$ and $\bigcup_{n} I_{n}=I$. 

Choose a finite approximation $(I_{n})$ of the wedge cut $\check{\mu}$. 
For each $n$ let $\mathcal{W}_{n}\subset\mathcal{W}$ be the set of wedges indexed by $I_{n}\subset I$ and $\underline{\mathcal{W}}^{2}_{n}=(\mathcal{W}_{n}\times\mathcal{W}_{n}-\Delta)/\mathbb{Z}_{2}$ be the set of unordered pairs of different wedges of $\mathcal{W}_{n}$.
Given a collection of infinitely many shears $\sigma\in\mathbb{T}_{\check{\mu}}(S^0)$, we get for each $n$ a finite collection of shears $\sigma_{n}$ between pairs of wedges in $\underline{\mathcal{W}}^{2}_{n}$:
if $i,j\in I_n\subset I$, then the shear between the wedges $w_i$ and $w_j$ is $\sigma_{n}(w_i,w_j):=\sigma(w_i,w_j)$.

For each $n$ a strip $S_{n}$ is constructed by gluing the wedges indexed by $I_{n}$ edge-to-edge using the shears $\sigma_{n}$:
if $i$ and $j$ are two adjacent indexes in $I_{n}$, glue the wedges $w_i$ and $w_j$ so that they are sheared by $\sigma_{n}(w_i,w_j)$.
Adjacent wedges of $\mathcal{W}_n$ are all glued together using this process using the total ordering on $I_n$.
This yields a strip $S_{n}$ together with a finite wedge cut $\check{\mu}_{n}$.
We say that the pair $(S_{n},\check{\mu}_{n})$ is a {\bf geometric realization} of the finite approximation $I_{n}$ (it is unique up to isometry).
The strip $S_n$ is equipped with its horocyclic foliation $F_{\check{\mu}_n}(S_n)$.

The core of the construction lies in showing that there is a sequence of properly chosen geometric realizations which converges in the following sense: 
the strips $S_n$ and the wedge cuts $\check{\mu}_n$, seen as compact subsets of the unit disc in the complex plane, converge in the Hausdorff topology to a strip $S$ equipped with the wedge cut $\check{\mu}$. 
Furthermore, the horocyclic foliations $F_{\check{\mu}_n}(S_n)$ converge to the horocyclic foliation of $S$: more precisely, each leaf of $F_{\check{\mu}_n}(S_n)$ parameterized by arc-length (respectively its length) converges uniformly to the corresponding leaf of $F_{\check{\mu}}(S)$ (respectively the length of the corresponding leaf).  
This is done in \cite{Thurston86} Proposition 4.1 p.14 (see also \cite{Bonahon96} and \cite{PapThe07}).

The construction shows in particular that the shear map $\sigma_{\check{\mu}}$ is a bijection between $\mathbb{T}_{\check{\mu}}(S^0)$ and $\mathcal{T}_{\check{\mu}}(S^0)$.
Some work remains to show that this map is actually a homeomorphism.

\begin{theorem}
Let $S^0$ be a strip and $\check{\mu}$ a wedge cut of $S^0$.
The shear map $\sigma_{\check{\mu}}$ is a homeomorphism from $\mathcal{T}_{\check{\mu}}(S^0)$ to $\mathbb{T}_{\check{\mu}}(S^0)$.
\end{theorem}

We know from Section \ref{subsection:local_shear_general} that in the surface $\Sigma$, the finite collection of distinguished shears obtained from a train-track approximation $\Theta$ defines all other shears between pairs of ideal triangles, each of the latter being an integral linear combination of the formers.
We thus get an \emph{injective} linear map $q_{\Theta,e}\,:\,\textrm{T}(\Theta)\to\mathbb{T}_{\check{\mu}}(S^0)$ for each strip $S^0$ coming from a branch $e$ of $\Theta$.
Distinguished shears in $\textrm{T}(\Theta)$ hence determine the developing map and the holonomy map in a neighbourhood of each branch of $\Theta$, hence over the surface $\Sigma$.
This shows that $\sigma_{\Theta}\,:\,\mathcal{T}(\Theta)\to\textrm{T}(\Theta)$ are local coordinates for the Teichm\"uller space of $\Sigma$ indeed. 

In the finite case, the collection of all shears between pairs of ideal wedges is also over-determined.
Because of Lemma \ref{lemma:separation}, only shears between pairs of adjacent wedges are actually needed.
We can thus define a linear projection $q_{n}\,:\,\mathbb{R}^{\underline{\mathcal{W}}^{2}_{n}}\to\mathbb{R}^{\vert\mathcal{W}_{n}\vert-1}$ by keeping only shears between adjacent wedges of $\mathcal{W}_{n}$, since they determine the other shears.
The restriction of the shear map $\sigma_{\check{\mu}_{n}}$ to shears between adjacent wedges defines a homeomorphism between $\mathcal{T}_{\check{\mu}_{n}}(S_{n})$ and $\mathbb{R}^{\vert\mathcal{W}_{n}\vert-1}$ and we talk in this case about shear coordinates.

\section{Convexity of length functions}
\label{section:convexity}

\subsection{Convexity of length in a wedge}
\label{subsection:wedgeconvexity}

We throughout rely on notation introduced in Section \ref{section:shearcoordinates}.
Let us consider a wedge $w$ in the upper half-plane model of $\mathbb{H}^{2}$ so that the common endpoint of the boundary geodesics is at infinity.
We equip this wedge with its two distinguished points and orient the edges of $w$ positively toward infinity.
The positive parameterization of these two oriented geodesics by arc length provides a map $\mathbb{R}\times\mathbb{R}\to\partial w$ which is unique if we impose the image of $(0,0)$ to be the two distinguished points.
Let $d\,:\,\mathbb{R}\times\mathbb{R}\to\mathbb{R}_{+}$ be the distance function between the points of $\partial w$ parameterized by $\mathbb{R}\times\mathbb{R}$.
A straightforward consequence of the strict convexity of the hyperbolic distance function (\cite{Thurston97}, Theorem 2.5.8 p.90) is 

\begin{proposition}
\label{proposition:d_convex}
The map $d$ is strictly convex on $\mathbb{R}\times\mathbb{R}$.
\end{proposition}

We now want to generalize this result to strips and we are looking for establishing strict convexity of the length of a geodesic segment with endpoints on opposite sides of a strip in terms of their positions with respect to some fixed distinguished points.
Of course, convexity of the length of such a segment easily follows from the convexity of hyperbolic distance function.
But we actually want more: we shall also allow the strip to be deformed through shearing along a fixed wedge cut,
so our convexity result shall depend upon the positions of the endpoints of the segment but also on the shear coordinates describing the strip.
We shall first study the case where the wedge cut is \emph{finite} and we shall take up the general case later by approximating a given infinite wedge cut by a sequence of finite wedge cuts. 

\subsection{Convexity of length function over the Teichm\"uller space of a strip equipped with a finite wedge cut}
\label{subsection:finiteconvexity}

Let us consider a strip $S^0$ in $\mathbb{H}^{2}$ bounded by two geodesics $g^{-}$ and $g^{+}$ and equipped with a \emph{finite} wedge cut $\check{\mu}$ ($g^{\pm}\subset\check{\mu}$).
Fix one distinguished point $A^{+}$ on $g^{+}$ and one distinguished point $A^{-}$ on $g^{-}$.

Recall that $\mathcal{W}$ denotes the set of wedges cut in the strip $S^0$ by $\check{\mu}$.
We hence work for the time being under the assumption that $\vert\mathcal{W}\vert$ is finite.
Note that we have $\vert\check{\mu}\vert=\vert\mathcal{W}\vert+1$.
We index the leaves of $\check{\mu}$ from $g^{-}$ to $g^{+}$.

Recall that Teichm\"uller space $\mathcal{T}_{\check{\mu}}(S^0)$ is here finite-dimensional and that shear coordinates define a homeomorphism $\sigma_{\check{\mu}}$ from $\mathcal{T}_{\check{\mu}}(S^0)$ to $\mathbb{R}^{\vert\check{\mu}\vert-2}$ (see Section \ref{section:developing}).
Set $\mathbf{x}^{0}=\sigma_{\check{\mu}}(S^0)$ and let $\mathbf{x}=(x_{2},\cdots,x_{\vert\check{\mu}\vert-1})$ be a point of $\mathbb{R}^{\vert\check{\mu}\vert-2}$ representing a hyperbolic structure $S$ of the strip $S^0$ in $\mathcal{T}_{\check{\mu}}(S^0)$.
Recall that each $x_{j}$, $2\leq j\leq\vert\check{\mu}\vert-1$, can be viewed as a number attached to the \emph{inner}, $j$-th leaf of $\check{\mu}$ encoding the shear between the two adjacent wedges glued along that leaf.
(In what follows, the letter $j$ will be used to index leaves of laminations while the letter $i$ will be preferably used for indexing wedges.)

Consider a geodesic segment $\alpha$ joining the sides of $S^0$ and denote by $y^{+}$ and $y^{-}$ the signed distances of its endpoints with respect to the distinguished points $A^{+}$ and $A^{-}$ respectively. 
We want to show that the length of $\alpha$ is a strictly convex function of $y^{\pm}$ and $\mathbf{x}$.
The idea of the proof we borrow from \cite{BBFS09} is to study the length of piecewise geodesic curves $\beta$ joining the points parameterized by $y^{+}$ and $y^{-}$.
More precisely, a \textbf{$\check{\mu}$-piecewise geodesic curve} $\beta$ shall by definition refer to any curve such that the intersection of $\beta$ with the interior of every wedge is a geodesic segment -- we call a \textbf{wedge segment} -- 
and its intersection with every \emph{inner} leaf of $\check{\mu}$ is also a geodesic segment -- we call a \textbf{leaf segment}.
We let $\mathcal{C}_{\check{\mu}}(S^0)$ be the set of all these $\check{\mu}$-piecewise geodesic curves.
A curve $\beta$ of $\mathcal{C}_{\check{\mu}}(S^0)$ is encoded by the positions of the endpoints of its \emph{wedge segments} with respect to the distinguished points of the wedges, that is, by $2\times\vert\mathcal{W}\vert$ real numbers we denote by $\textbf{y}=(y_{1}^{\pm},\ldots,y_{\vert\mathcal{W}\vert}^{\pm})$, where $y_{i}^{\pm}$ encode the endpoints of the $i$-th wedge segment.
This gives a bijective map $\tau_{S^0}\,:\,\mathcal{C}_{\check{\mu}}(S^0)\longrightarrow\left(\mathbb{R}^{\vert\mathcal{W}\vert}\right)^2$ which associates to 
the $\check{\mu}$-piecewise geodesic curve $\beta$ the positions of the endpoints of its wedge segments on the leaves of $\check{\mu}$.

For each pair $y^{\pm}\in\mathbb{R}^2$ there is a unique geodesic segment $\alpha$ joining the points on $\partial S^0$ parameterized by $y^{\pm}$.
For such a fixed geodesic segment $\alpha$ we denote by $\mathcal{C}^{\alpha}_{\check{\mu}}(S^0)$ the subset of all $\beta\in\mathcal{C}_{\check{\mu}}(S^0)$ having the same endpoints as $\alpha$.
We have $\ell_{\alpha}(S^0)=\inf_{\beta\in\mathcal{C}^{\alpha}_{\check{\mu}}(S^0)}\ell_{\beta}(S^0)$.
Since the infimum of strictly convex functions is itself strictly convex, it suffices to establish the strict convexity of every $\beta$.
Now the length of $\beta$ is the sum of the lengths of its geodesic pieces and we know from the previous section that each of them is a convex
function of the position of its endpoints (and is strictly convex for wedge segments).
Therefore the strict convexity of $\ell_{\beta}(S^0)$ follows in terms of the position numbers $\mathbf{y}$ for \emph{fixed} $S^0$.
Of course, this result also follows directly from the convexity of the distance function on $\mathbb{H}^{2}\times\mathbb{H}^{2}$.
However we want to take into account that $S^0$ can be deformed into another strip.
This implies to define a way of keeping track of the curve $\beta$ as the strip $S^0$ gets deformed into another strip $S$. 

So now we also deform the hyperbolic structure on $S^0$ by shearing it along $\check{\mu}$.
In order to pass from the strip $S^0=\sigma_{\check{\mu}}^{-1}(\mathbf{x}^{0})$ to the strip $S=\sigma_{\check{\mu}}^{-1}(\mathbf{x})$,
a shear of amplitude $x_{j}-x^{0}_{j}$ is performed along the \emph{inner} $j$-th leaf of $\check{\mu}$ ($2\leq j\leq\vert\check{\mu}\vert-1$).
Geometrically this shear is realized by fixing this $j$-th leaf which is the common geodesic of the adjacent wedges $w_{i}$ and $w_{i+1}$ and by isometrically translating
the upper half-plane containing $w_{i+1}$ by the amount $\frac{1}{2}(x_{j}-x^{0}_{j})$ and by isometrically translating
the lower half-plane containing $w_{i}$ by the same amount $\frac{1}{2}(x_{j}-x^{0}_{j})$ (positive translation amounts are towards the left).
The total shear that enables to pass from $S^0$ to $S$ is obtained by composing these geometric operations about each inner leaf of $\check{\mu}$;
different orders of composition yield different but isometric strips, so the order of composition turns out to be unimportant.

Let us have a closer look to what happens to a piecewise geodesic curve $\beta\in\mathcal{C}_{\check{\mu}}(S^0)$ in $S^0$ as the strip $S^0$ gets sheared into the strip $S$. 
As two adjacent wedges $w_{i}$ and $w_{i+1}$ get sheared, the endpoints of the wedge segments on the separating leaf get disconnected $x_{j}-x^{0}_{j}$ apart.
Their positions with respect to the distinguished points of $w_{i}$ and $w_{i+1}$ are still $y_{i}^{+}$ and $y_{i+1}^{-}$.
However, the leaf segment of $\beta$ contained in the common geodesic of $w_{i}$ and $w_{i+1}$ gets disconnected from the wedge segment in that geodesic and its endpoint positions are changed into $y_{i}^{+}+\frac{1}{2}(x_{j}-x^{0}_{j})$ and $y_{i+1}^{-}+\frac{1}{2}(x_{j}-x^{0}_{j})$.
We choose to reconnect the wedge segments to the leaf segment contained in the leaf of $\check{\mu}$ separating $w_{i}$ and $w_{i+1}$.
We perform this reconnecting process for every inner leaf of $\check{\mu}$: this gives the definition of the piecewise geodesic curve $\beta$ in the sheared strip $S$.
Note that every leaf segment of $\beta$ has the same length for the strips $S^0$ and $S$.
Note also that the numbers $y^{-}, y_{1}^{-}, y^{+}, y_{\vert\mathcal{W}\vert}^{+}$ remain unchanged whatever ${\bf x}$ is.

Another way of describing this reconnecting process is by saying that the curve $\beta=\tau_{S^0}^{-1}(\mathbf{y})$ in the strip $S=\sigma_{\check{\mu}}^{-1}(\mathbf{x})$ is defined by the 
position numbers $\mathbf{y}+\frac{1}{2}(\overline{\mathbf{x}}-\overline{\mathbf{x}}^{0})$, where we extend the $\vert\check{\mu}-2\vert$-vectors ${\bf x}$ and ${\bf x}_0$ to $(2\times\vert\mathcal{W}\vert)$-vectors denoted by $\overline{{\bf x}}$ and $\overline{{\bf x}}_0$, by adding one zero for the first component and one zero for the last component and by repeating twice every other component (thus if ${\bf x}=(x_1,x_2,\cdots,x_{\vert\check{\mu}\vert-1})$ then $\overline{{\bf x}}=(0,x_1,x_1,x_2,x_2\cdots,x_{\vert\check{\mu}\vert-1},x_{\vert\check{\mu}\vert-1},0)$).
We define a bijective map
$$
\tau\,:\,\mathcal{T}_{\check{\mu}}(S^0)\times\mathcal{C}_{\check{\mu}}(S^0)\longrightarrow\mathbb{R}^{\vert\check{\mu}\vert-2}\times\left(\mathbb{R}^{\vert\mathcal{W}\vert}\right)^2,
$$
such that $\tau(S^0,\cdot)=\tau_{S^0}(\cdot)$ and $\tau(S,\beta)=\mathbf{y}+\frac{1}{2}(\overline{\mathbf{x}}-\overline{\mathbf{x}}^{0})$ where $S=\sigma_{\check{\mu}}^{-1}(\mathbf{x})$ and $\beta=\tau_{S^0}^{-1}(\mathbf{y})$.
This reconnecting process we have just defined has a slight drawback: in general, starting from any geodesic segment $\alpha$ in $S^0$ and applying this process gives endpoint positions that \emph{are not} the 
endpoint positions of the genuine geodesic segment $\alpha$ in $S$; in other words, $\tau_{S}(\alpha)\neq\tau(S,\alpha)$. 
Hence we might have $\tau_{S}(\cdot)\neq\tau(S,\cdot)$ whenever $S\neq S^0$. 
This feature will nevertheless cause no trouble.

Now that we have defined a way of recognizing a curve $\beta$ in all possible strips of $\mathcal{T}_{\check{\mu}}(S^0)$, we can focus on the length of those curves.
We define length function 
$$
\underline{\ell}\,:\,\mathbb{R}^{\vert\check{\mu}\vert-2}\times\left(\mathbb{R}^{\vert\mathcal{W}\vert}\right)^2\longrightarrow\mathbb{R}_{+},
$$
by
$$
\underline{\ell}(\mathbf{x},\mathbf{y})=\ell_{\beta}(S),
$$
where $(S,\beta)=\tau^{-1}(\mathbf{x},\mathbf{y})$.
It is now easy to show the strict convexity of this length function.

\begin{proposition}
The length function $\underline{\ell}$ is strictly convex over $\mathbb{R}^{\vert\check{\mu}\vert-2}\times\left(\mathbb{R}^{\vert\mathcal{W}\vert}\right)^2$.
\end{proposition}

\begin{proof}
Consider the $i$-th wedge $w_i$ of $\mathcal{W}$ bordered by the $j$-th and $(j+1)$-th leaves of $\check{\mu}$ ($1\leq i\leq\vert\mathcal{W}\vert$).
The map 
$$
(x_{j},y_{i}^{-},x_{j+1},y_{i}^{+})\mapsto\left(y_{i}^{-}+\frac{1}{2}(x_{j}-x_{j}^{0}),y_{i}^{+}+\frac{1}{2}(x_{j+1}-x_{j+1}^{0})\right)
$$
that associates to $(x_{j},y_{i}^{-},x_{j+1},y_{i}^{+})$ the signed distance from distinguished points of the points on $\partial w$ encoded by these four numbers is linear (if $j=1$, respectively $j+1=\vert\check{\mu}\vert$, we set $x_j=0=x_j^0$, respectively $x_{j+1}=x_{j+1}^0=0$).

Compounding with the distance function, Proposition \ref{proposition:d_convex} implies that the length of a wedge segment in $w_i$ is a \emph{strictly} convex function $\underline{\ell}_{{w}_{i}}$ of $(x_{j},y_{i}^{-},x_{j+1},y_{i}^{+})$ and is therefore a convex function of $(\mathbf{x},\mathbf{y})$.

Consider the $j$-th leaf $l_j$ of $\check{\mu}$ separating by the $i$-th and $(i+1)$-th wedges of  $\mathcal{W}$ ($2\leq j\leq\vert\check{\mu}\vert-1$).
The function $\underline{\ell}_{{l}_{j}}$ that associates to $(y_{i}^{+},y_{i+1}^{-})$ the length of the leaf segment contained in the $j$-th leaf of $\check{\mu}$ and joining the points encoded by those two numbers is constant, hence is a convex function of $(\mathbf{x},\mathbf{y})$.
Now the length function $\underline{\ell}$ is the sum of these maps, namely, 
$$
\underline{\ell}=\sum_{i=1}^{\vert\mathcal{W}\vert}\underline{\ell}_{{w}_{i}}+\sum_{j=2}^{\vert\check{\mu}\vert-1}\underline{\ell}_{{l}_{j}}.
$$
It is therefore a convex function which is actually strictly convex since for any two points $(\mathbf{x},\mathbf{y})$ and $(\mathbf{x}',\mathbf{y}')$ in 
$\mathbb{R}^{\vert\check{\mu}\vert-2}\times\left(\mathbb{R}^{\vert\mathcal{W}\vert}\right)^2$ there exists at least one function $\underline{\ell}_{{w}_{i}}$ which is not constant,
hence which is strictly convex, along the affine path joining $(\mathbf{x},\mathbf{y})$ and $(\mathbf{x}',\mathbf{y}')$ in $\mathbb{R}^{\vert\check{\mu}\vert-2}\times\left(\mathbb{R}^{\vert\mathcal{W}\vert}\right)^2$.
This concludes the proof.
\end{proof}

We now turn to the length of the geodesic curve $\alpha$.
It is worth mentioning here that, for each strip $S$, all $\check{\mu}$-piecewise geodesic curves $\beta$ obtained by our reconnecting process are \emph{connected}.
Therefore, for each $S$ and each geodesic segment $\alpha$ connecting the sides of $\partial S$, 
$$
\ell_{\alpha}(S)=\inf_{\beta\in\mathcal{C}^{\alpha}_{\check{\mu}}(S^0)}\ell_{\beta}(S).
$$
The geodesic curve $\alpha$, as well as all $\beta$ in $\mathcal{C}^{\alpha}_{\check{\mu}}(S^0)$, connects points located on $\partial S$ with respect to the distinguished points $A^{\pm}$ by both numbers $y^{-}$ and $y^{+}$.
We denote by $R_\alpha$ the linear subspace of $\left(\mathbb{R}^{\vert\mathcal{W}\vert}\right)^2$ obtained by fixing the first and last coordinates $y_1^-$ and $y_{\vert\mathcal{W}\vert}^+$ of ${\bf y}$ so that they encode point positions matching those of $\alpha$.
In other words, if $y_{A^\pm}$ denote the positions of $A^{\pm}$ with respect to the first and to the last distinguished points of the wedges, then $y_1^-=y_{A^{-}}+y^{-}$ and $y_{\vert\mathcal{W}\vert}^+=y_{A^{+}}+y^{+}$.
Hence 
$$
\ell_{\alpha}(S)=\inf_{\mathbf{y}\in R_{\alpha}}\underline{\ell}(\mathbf{x},\mathbf{y}).
$$
We set, for all $\mathbf{x}\in\mathbb{R}^{\vert\check{\mu}\vert-2}$ and $y^{\pm}\in\mathbb{R}^2$,
\begin{eqnarray*}
\underline{\ell}_{\alpha}(\mathbf{x},y^{\pm})&:=&\ell_{\alpha}(S)\ \textrm{\scriptsize{(where $\alpha$ is the geodesic curve connecting the points given by $y^{\pm}$})}\\
&=&\inf_{\beta\in\mathcal{C}^{\alpha}_{\check{\mu}}(S^0)}\ell_{\beta}(S)\\
&=&\inf_{\mathbf{y}\in R_{\alpha}}\underline{\ell}(\mathbf{x},\mathbf{y}).
\end{eqnarray*}
Note that the infimum is uniquely attained (by $\alpha$).

\begin{corollary}
\label{corollary:finiteconvexalpha}
The length function $\underline{\ell}_{\alpha}$ of $\alpha$ is strictly convex with respect to shear coordinates and signed distances of its endpoints.
\end{corollary}

\begin{proof}
Let $(\mathbf{x},y^{\pm})$ and $(\mathbf{x}',y'^{\pm})$ be two different points of $\mathbb{R}^{\vert\check{\mu}\vert-2}\times\mathbb{R}^{2}$.
Let $\mathbf{y}$ and $\mathbf{y}'$ be two points of $R_\alpha$ such that 
$$
\underline{\ell}_{\alpha}(\mathbf{x},y^{\pm})=\underline{\ell}(\mathbf{x},{\bf y})\ \textrm{and}\ \underline{\ell}_{\alpha}(\mathbf{x}',y'^{\pm})=\underline{\ell}(\mathbf{x}',{\bf y}').
$$

Let $t$ be a real number in $[0;1]$.
We have
\begin{eqnarray*}
\underline{\ell}_{\alpha}((1-t)(\mathbf{x},y^{\pm})+t(\mathbf{x}',y'^{\pm}))
&=&\underline{\ell}_{\alpha}((1-t)\mathbf{x}+t\mathbf{x'},(1-t)y^{\pm}+ty'^{\pm})\\
&\leq&\underline{\ell}((1-t)\mathbf{x}+t\mathbf{x}',(1-t)\mathbf{y}+t\mathbf{y}')\quad\textrm{(as $\underline{\ell}_{\alpha}$ is the infimum)}\\
&\leq&\underline{\ell}((1-t)(\mathbf{x},\mathbf{y})+t(\mathbf{x}',\mathbf{y}'))\\
&<& (1-t)\underline{\ell}(\mathbf{x},\mathbf{y})+t\underline{\ell}(\mathbf{x}',\mathbf{y}')\quad\textrm{(strict convexity of $\underline{\ell}$)}\\
&<&(1-t)\underline{\ell}_{\alpha}(\mathbf{x},y^{\pm})+t\underline{\ell}_{\alpha}(\mathbf{x}',y'^{\pm})\quad\textrm{(definition of $\mathbf{y},\mathbf{y}'$)}.
\end{eqnarray*}
Hence $\underline{\ell}_{\alpha}$ is strictly convex.
\end{proof}

Similarly we get the following nice result.

\begin{corollary}
\label{corollary:finiteconvexcore}
The length function $\underline{\ell}_{\textrm{core}}$ of the core curve of a strip is strictly convex with respect to shear coordinates.
\end{corollary}

The proof is the same as the previous one, once we have set
$$
\underline{\ell}_{\textrm{core}}(\mathbf{x}):=
\inf_{\mathbf{y}\in\left(\mathbb{R}^{\vert\mathcal{W}\vert}\right)^2}\underline{\ell}(\mathbf{x},\mathbf{y}).
$$

\begin{proof}
Let  $\mathbf{x}$ and $\mathbf{x}'$ be two different points of $\mathbb{R}^{\vert\check{\mu}\vert-2}$.
Let $\mathbf{y}$ and $\mathbf{y}'$ be two points of $\left(\mathbb{R}^{\vert\mathcal{W}\vert}\right)^2$ such that 
$$
\underline{\ell}_{core}(\mathbf{x})=\underline{\ell}(\mathbf{x},{\bf y})\ \textrm{and}\ \underline{\ell}_{core}(\mathbf{x}')=\underline{\ell}(\mathbf{x}',{\bf y}').
$$
Let $t$ be a real number in $[0;1]$.
We have
\begin{eqnarray*}
\underline{\ell}_{\textrm{core}}((1-t)\mathbf{x}+t\mathbf{x}')&\leq&\underline{\ell}((1-t)\mathbf{x}+t\mathbf{x}',(1-t)\mathbf{y}+t\mathbf{y}')\quad\textrm{(as $\underline{\ell}_{\textrm{core}}$ is the infimum)}\\
&\leq&\underline{\ell}((1-t)(\mathbf{x},\mathbf{y})+t(\mathbf{x}',\mathbf{y}'))\\
&<& (1-t)\underline{\ell}(\mathbf{x},\mathbf{y})+t\underline{\ell}(\mathbf{x}',\mathbf{y}')\quad\textrm{(strict convexity of $\underline{\ell}$)}\\
&<&(1-t)\underline{\ell}_{\textrm{core}}(\mathbf{x})+t\underline{\ell}_{\textrm{core}}(\mathbf{x}')\quad\textrm{(definition of $\mathbf{y},\mathbf{y}'$)}.
\end{eqnarray*}
\end{proof}

\subsection{Convexity of length function over the Teichm\"uller space of a strip equipped with an infinite wedge cut}
\label{subsection:infiniteconvexity}

\subsubsection*{{\bf \S 1.}}
We now deal with the case where the wedge cut $\check{\mu}$ of $S^0$ has infinitely many leaves and therefore the cardinality of $\mathcal{W}$ is infinite.
The leaves that bound wedges of $\mathcal{W}$ are called the {\bf frontier leaves} of $\check{\mu}$ and their set is denoted by $\check{\mu}_{f}$; it is countable and it is dense in $\check{\mu}$.
Each wedge yields exactly two frontier leaves while one frontier leaf bounds either one or two wedges.

For reasons connected with our dealing later on with surfaces, we shall assume that $\check{\mu}$ has zero area (or, more precisely, that the intersection between the core of the strip and $\check{\mu}$ is a Cantor set).

Let $A^{\pm}$ be as before the distinguished points on the boundary of a strip $S^0$ equipped with a wedge cut $\check{\mu}$.
As for the finite case, we consider curves joining $\partial S^0$ which are piecewise geodesic with respect to $\check{\mu}$.
To be precise, a curve $\beta$ is \textbf{$\check{\mu}$-piecewise geodesic} if every wedge component (exactly one per wedge) of $\beta\setminus\beta\cap\check{\mu}$ and every leaf component (exactly one per \emph{inner frontier} leaf, possibly reduced to a point) of $\beta\cap\check{\mu}_{f}$ is geodesic.
We futhermore require the length of $\beta$ to be finite: $\ell_{\beta}(S^0)<\infty$.
%Such a curve can be rather complicated with uncountably many singular points, one for each leaf of $\check{\mu}$ (as a paradigm, think of a Cantor ``staircase", whose length is nevertheless finite).
Geodesic segments $\alpha$ connecting sides of $\partial S^0$ belong to this family of curves.
We denote by $\mathcal{C}_{\check{\mu}}(S^0)$ the set of all $\check{\mu}$-piecewise geodesic curves and by $\mathcal{C}^{\alpha}_{\check{\mu}}(S^0)$ the subset of all $\check{\mu}$-piecewise geodesic curves with the same endpoints as a fixed $\alpha$.

In the finite case, $\mathcal{C}_{\check{\mu}}(S^0)$ was parameterized by finitely many real numbers encoding the position of the wedge segment endpoints.   
In the infinite case, there are countably many such endpoints, all lying on frontier leaves, two per wedge.
We consider the map $\tau_{S^0}\,:\,\mathcal{C}_{\check{\mu}}(S^0)\to\mathbb{R}^{\check{\mu}_{f}}\times\mathbb{R}^{\check{\mu}_{f}}=\left(\mathbb{R}^{\check{\mu}_{f}}\right)^2$ that associates to any $\check{\mu}$-piecewise geodesic curve $\beta$ the position of its wedge segment endpoints.
We show now that this map parameterizes the set $\mathcal{C}_{\check{\mu}}(S^0)$ conveniently.

\begin{lemma}
\label{lemma:tau_injective}
The map $\tau_{S^0}$ is injective.
\end{lemma}

\begin{proof}
Let $\beta$ and $\beta'$ be two $\check{\mu}$-piecewise geodesic curves such that $\tau_{S^0}(\beta)=\tau_{S^0}(\beta')$.
This readily implies that $\beta$ and $\beta'$ coincide outside $\check{\mu}$.

Consider a leaf segment $s$ of $\beta$.
Suppose that the leaf containing $s$ is isolated, that is, squeezed between two wedges.
Then its endpoints are endpoints of wedge segments of $\beta$, hence of $\beta'$.
Therefore, $s$ is also a leaf segment of $\beta'$ and $\beta$ and $\beta'$ coincide also over the isolated leaves of $\check{\mu}$.

Let now $s$ be a leaf segment of $\beta$ contained in a non-isolated leaf of $\check{\mu}$.
At least one endpoint $p$ of $s$ is not an endpoint of some wedge segment.
The frontier leaves being dense in $\check{\mu}$ and by continuity of $\beta$, there exists a sequence of points $(p_{n})$, each one contained in a leaf segment $s_{n}$ of $\beta$ which converges to $p$.
Since the length of $\beta$ is finite and bounded below by the sum of the lengths of the leaf segments, the lengths $\ell_{s_{n}}(S^0)$ converge to zero.
We can therefore assume that each point $p_{n}$ is an endpoint of $s_{n}$ which we can choose furthermore to be an endpoint of some wedge segment.
Consequently each $p_{n}$ belongs to $\beta'$.
Passing to the limit, we conclude that $p\in\beta'$.
This shows that $s\subset\beta'$.
By symmetry, this proves that $s$ is a leaf segment of $\beta'$.
We thus have shown that $\beta$ and $\beta'$ coincide over $\check{\mu}$ as well.
Hence $\beta=\beta'$, which settles the injectivity of the map $\tau_{S^0}$.
\end{proof}

\subsubsection*{{\bf \S 2.}}

We equip $\left(\mathbb{R}^{\mathcal{W}}\right)^2$ with the topology arising from the distance $\Vert\mathbf{y}'-\mathbf{y}\Vert=\sum_{i\in I}\vert y'^{\pm}_{i}-y_{i}^{\pm}\vert$.

For each geodesic segment $\alpha$ joining the opposite sides of $\partial S^0$, set $\mathbf{y}^0=\tau_{S^0}(\alpha)$ and consider the open set $B({\bf y}^0)=\left\{\mathbf{y}\in\left(\mathbb{R}^{\mathcal{W}}\right)^2\,:\,\Vert \mathbf{y}-\mathbf{y}^0\Vert<\infty \right\}$.
Since $B({\bf y}^0)$ is the increasing union of balls centered in ${\bf y}^0$ and since these balls are each convex subsets of $\left(\mathbb{R}^{\mathcal{W}}\right)^2$ (see \cite{Papad13} Proposition 5.3.14), we conclude that the set $B({\bf y}^0)$ is convex as well.
We define $\textrm{C}_{\check{\mu}}^0(S^0)$ to be the convex hull of $\bigcup_{{\bf y}^0}B({\bf y}^0)$.

We want to show that any point ${\bf y}$ in $\textrm{C}_{\check{\mu}}^0(S^0)$ defines a $\check{\mu}$-piecewise geodesic curve of finite length joining opposite sides of $S^0$.
We make use of finite approximations.
So let $I_n\subset I$ be a finite approximation of $\check{\mu}$ as described in Section \ref{section:developing} and let $(S^{0}_{n},\check{\mu}_{n})$ be a sequence of geometric realizations of $(S^{0},\check{\mu})$.
Every $\check{\mu}$-piecewise geodesic curve $\beta$ in $S^0$ induces a $\check{\mu}_n$-piecewise geodesic curve $\beta_n$ in $S^0_n$ by keeping the wedge segments of $\beta$ indexed by $I_n\subset I$ and by connecting them with leaf segments contained in $\check{\mu}_n$.
We say that $\beta_{n}$ is the finite approximation of $\beta$ induced by $(S^{0}_{n},\check{\mu}_{n})$.
Recall that $\mathcal{W}_n$, respectively $\mathcal{W}$, denotes the set of wedges of $\check{\mu}_n$, respectively of $\check{\mu}$.

\begin{lemma}
\label{lemma:convergence_pwgeodesic}
For each $\mathbf{y}\in\textrm{C}_{\check{\mu}}^0(S^0)$ there exists a unique $\check{\mu}$-piecewise geodesic $\beta\in\mathcal{C}_{\check{\mu}}(S^0)$ such that $\tau_{S^0}(\beta)=\mathbf{y}$.
Furthermore, if $\beta_n$ denotes the finite approximation of $\beta$ induced by some geometric realization $(S^{0}_{n},\check{\mu}_{n})$ converging to $(S^{0},\check{\mu})$, then $\ell_{\beta_n}$ converges to $\ell_{\beta}$.
More precisely, if $\beta^l, \beta^w$ respectively denote the leaf part and the wedge part of $\beta$ and likewise if $\beta_{n}^l, \beta_{n}^w$ denote the leaf part and the wedge part of $\beta_n$, then
$$
\lim_{n\to\infty}\ell_{\beta_n^l}=\ell_{\beta^l}\ \textrm{and}\ \lim_{n\to\infty}\ell_{\beta_n^w}=\ell_{\beta^w}.
$$
\end{lemma}

\begin{proof}
Pick a finite approximation $(I_n)$ of $I$ and choose the finite geometric realizations $(S^0_{n},\check{\mu}_{n})$ of $\check{\mu}$ so that they converge to $(S^0,\check{\mu})$.
For each $n$, let $\beta_{n}$ be the finite approximation of $\beta$ induced by $(S^0_{n},\check{\mu}_{n})$;
it is a $\check{\mu}_n$-piecewise geodesic curve.

Consider the length $\ell_{\beta_{n}}$ of each $\check{\mu}_n$-piecewise geodesic curve $\beta_{n}$.
We first show that the sequence $(\ell_{\beta_{n}})_n$ is bounded.

As a first case, we assume that there exists a geodesic segment $\alpha$ with $\mathbf{y}^0=\tau_{S^0}(\alpha)$ such that ${\bf y}$ belongs to the ball $B({\bf y}^0)$ centered at ${\bf y}^0$.
The curve $\alpha$ connects the points located by $y^{-}$ and $y^{+}$ on $\partial S^0$.
Set $\mathbf{y}^0=(y^{0\pm}_{i})_{i\in I}$.
We denote by $\alpha_{n}$ the $\check{\mu}_n$-piecewise geodesic curve which is the finite approximation of $\alpha$; 
for each $n$ the curve $\alpha_n$ is parameterized by the numbers ${\bf y}_{n}^{0}\in\left(\mathbb{R}^{\check{\mu}_{n}}\right)^2$.

We first consider the sum $\ell_{\beta^{w}_{n}}$ of the lengths of wedge segments of $\beta_n$.
Let $W\subset\mathbb{H}^2$ be the wedge in $(S^0_{n},\check{\mu}_{n})$ indexed by $i\in I_n\subset I$.
Let $w_n$ be the wedge segment of $\beta_{n}$ in $W$ and let $w^0_n$ be the wedge segment of $\alpha_n$ in $W$.
Triangle inequality gives
$$
\ell_w\leq\ell_{w^0}+\vert y_{i}^{+}-y_{i}^{0,+}\vert+\vert y_{i}^{-}-y_{i}^{0,-}\vert.
$$
Hence, for each $n$, the sum $\ell_{\beta^{w}_{n}}$ of the lengths of wedge segments of $\beta_{n}$ satisfies
\begin{eqnarray*}
\ell_{\beta^{w}_{n}}=\sum_{w\in \mathcal{W}_n}\ell_{w}&\leq&\sum_{W\in\mathcal{W}_n}\ell_{w^0}+\sum_{i\in I^0_n}\vert y_{i}^{+}-y_{i}^{0,+}\vert+\vert y_{i}^{-}-y_{i}^{0,-}\vert\\
&\leq&\sum_{i\in I_n}\vert y_{i}^{+}-y_{i}^{0,+}\vert+\vert y_{i}^{-}-y_{i}^{0,-}\vert+\ell_{\alpha_{n}^{w}}\\
&\leq&\ell_{\alpha_{n}^{w}}+\Vert\mathbf{y}_n-\mathbf{y}^{0}_n\Vert\\
&\leq&\ell_{\alpha_{n}^{w}}+\Vert\mathbf{y}-\mathbf{y}^{0}\Vert.
\end{eqnarray*}

Let $l$ be a leaf segment of $\beta_n$ which is on the leaf of $\check{\mu}_n$ separating the $i$-th and the $i+1$-th wedge.
Thus
 $$
\ell_l=\vert y_{i}^{+}-y_{i+1}^{-}\vert.
$$
For each $n$, the sum $\ell_{\beta^{l}_{n}}$ of the lengths of leaf segments of $\beta_{n}$ satisfies
\begin{eqnarray*}
\ell_{\beta^{l}_{n}}=\sum_{i\in I_n}\vert y_{i}^{+}-y_{i+1}^{-}\vert&\leq&\sum_{i\in I_n}\vert y_{i}^{+}-y_{i}^{0,+}\vert+\vert y_{i}^{0,+}-y_{i+1}^{0,-}\vert+\vert y_{i+1}^{-}-y_{i+1}^{0,-}\vert\\
&\leq&\sum_{i\in I_n}\vert y_{i}^{+}-y_{i}^{0,+}\vert+\vert y_{i}^{-}-y_{i}^{0,-}\vert+\ell_{\alpha_{n}^{l}}\\
&\leq&\Vert\mathbf{y}_n-\mathbf{y}^{0}_n\Vert+\ell_{\alpha_{n}^{l}}\\
&\leq&\Vert\mathbf{y}-\mathbf{y}^{0}\Vert+\ell_{\alpha_{n}^{l}}.
\end{eqnarray*}

Summing both estimates, we get for each $n$
$$
\ell_{\beta_{n}}=\ell_{\beta^{l}_{n}}+\ell_{\beta^{w}_{n}}\leq\ell_{\alpha_{n}}+2\Vert\mathbf{y}-\mathbf{y}^{0}\Vert.
$$
We just have to prove that the sequence $(\ell_{\alpha_{n}})_n$ is bounded.
Now the geometric realization $(S^0_n,\check{\mu}_{n})$ is obtained, up to isometry, from $(S^{0},\check{\mu})$ by collapsing almost all wedges of $\check{\mu}$ along the leaves of the horocyclic foliation of $S^{0}$.
If one wedge is collapsed onto one geodesic, any wedge segment $w$ inside this wedge is then collapsed onto a geodesic segment $\underline{w}$ contained in that geodesic.
It is easy to check that the length of $\underline{w}$ is smaller to the length of the wedge segment $w$.
We thus conclude that, for all $n$, $\ell_{\alpha_n}<\ell_{\alpha_{n+1}}<\ell_{\alpha}$.
This shows that the sequence of lengths $(\ell_{\alpha_n})$ converges and that the limit is smaller or equal to $\ell_{\alpha}$.
In particular, the sequence $(\ell_{\alpha_{n}})_n$ is bounded and so is the sequence $(\ell_{\beta_{n}})_n$, which was our goal.

As a second case, let ${\bf y}$ be a arbitrary point in the convex set $\textrm{C}_{\check{\mu}}^0(S^0)$.
By Carath\'eodory theorem, $\textrm{C}_{\check{\mu}}^0(S^0)$ is the set of all barycenters of points in $\bigcup_{{\bf y}^0}B({\bf y}^0)$.
In particular, there exist two points ${\bf y}_1$ and ${\bf y}_2$ in $\bigcup_{{\bf y}^0}B({\bf y}^0)$ and a number $t\in[0,1]$ such that ${\bf y}=t{\bf y}_1+(1-t){\bf y}_2$.
For each $n$, consider the $\check{\mu}_{n}$-piecewise geodesic curves $\beta_{1,n},\,\beta_{2,n},\,\beta_{n}$ associated to ${\bf y}_1,\,{\bf y}_2,\,{\bf y}$ by keeping the indices $I_n$.
Because of the construction of finite approximations, for each $n$, the curve $\beta_n$ \emph{is} the linear combination $t\beta_{1,n}+(1-t)\beta_{2,n}$.
We therefore have the estimates
$$
\min\{\ell_{\beta_{1,n}},\ell_{\beta_{2,n}}\}\leq\ell_{\beta_n}\leq\max\{\ell_{\beta_{1,n}},\ell_{\beta_{2,n}}\}.
$$
The previous case showed that both $\ell_{\beta_{1,n}},\ell_{\beta_{2,n}}$ were bounded, so is $\ell_{\beta_n}$, as was to be shown.

Parameterizing the curves $\beta_n$ proportionnally to arclength so that they all have the same domain, say $[0;1]$, we now apply Proposition 1.4.11 of \cite{Papad13} which says that there exists a subsequence of $(\beta_{n})$ which converges uniformly to a curve $\beta$ with finite length (since $\ell_\beta\leq\ell_\alpha+2\Vert\mathbf{y}-\mathbf{y}^{0}\Vert$ in the first case treated above).

We now prove that the limit curve $\beta$ is $\check{\mu}$-piecewise geodesic.
Denote by $(p^{\pm}_{i,n})_{i\in I_n}$ the set of wedge segment endpoints of $\beta_n$.
For \emph{all} $n$, the points $p^{\pm}_{i,n}$ are located by the number $y^{\pm}_{i}$ with $i\in I_n\subset I$. 
For each $i\in I$ there exists a rank $N$ such that $i\in I_{n}$ for all $n\geq N$.
Because the geometric realizations $(S^{0}_{n},\check{\mu}_{n})$ converge to $(S^{0},\check{\mu})$, there exists for each $i\in I$ a point $p_{i}^{\pm}\in\check{\mu}$ in $\mathbb{H}^2$ which the sequence of points $(p_{i,n}^{\pm})$ converges to.
Consequently, all points $p_{i}^{\pm}$, $i\in I$, are points of $\beta$.
Since the restriction of every curve $\beta_n$ to the wedge indexed by $i$ consists in a single wedge segment $w_{i,n}$, namely the geodesic segment connecting the points $p_{i,n}^{\pm}$, we conclude that the intersection between $\beta$ and that wedge consists in the wedge segment $w_i$ connecting the points $p_{i}^{\pm}$.
Moreover, since the endpoints of all these wedge segments are located by the same numbers $y^{\pm}_{i}$, we have for all $n\geq N$ $\ell_{w_{i,n}}=\ell_{w_{i}}$.

Consider a frontier leaf $\gamma$ of $\check{\mu}$.
There exists a sequence of leaves $\gamma_n\subset\check{\mu}_n$ converging in the Hausdorff topology to that leaf.
Consider for each $n$ the leaf segment $l_n$ of $\beta_n$ contained in $\gamma_n$.
Let $p_{n}^{-}$ and $p_{n}^{+}$ be the endpoints of this leaf segment $l_n$.
The points $p_n^{\pm}$ are also wedge segment endpoints of $\beta_n$.
Because the sequence of curves $(\beta_n)$ converges uniformly, the points $(p_n^{\pm})$ are contained in a compact subset of $S^0$ and we can assume, up to taking a subsequence, that both sequences $(p_n^{\pm})$ converge; set $p^{-}$ and $p^{+}$ the limit points.
Both points $p^{\pm}$ belong to $\gamma$.
Because the sequence of curves $(\beta_n)$ converges uniformly, by compactness of these curves, we can conclude, up to taking a subsequence, that the points $p^{-}$ and $p^{+}$ belong to $\beta$.
Let $l$ denote the geodesic segment of $\gamma$ whose endpoints are $p^{-}$ and $p^{+}$.
The leaf segments $l_n$ being geodesic, we conclude that $(l_n)_n$ converges to $l$, that $\lim_{n\to\infty}\ell_{l_n}=\ell_l$ and that $l$ is contained in $\beta\cap\gamma$.
The segment $l$ is actually the whole intersection $\beta\cap\gamma$ since no other points of $\gamma$ is a limit point of the sequence $(\beta_n)$.

We thus have shown that the intersection between $\beta$ and any wedge of $\check{\mu}$ is a geodesic segment we call wedge segment of $\beta$. 
Likewise, the intersection between $\beta$ and any frontier leaf of $\check{\mu}$ is a geodesic segment we call leaf segment of $\beta$.
Furthermore, we have proven that if $\beta^{l}$ and $\beta^{w}$ denote the union of leaf segments and the union of wedge segments of $\beta$ respectively, we have 
$$
\lim_{n\to\infty}\ell_{\beta_n}=\ell_{\beta^l}+\ell_{\beta^w}\leq\ell_{\beta}.
$$
But by general principles (see \cite{Papad13} Corollary 1.4.5) we have $\ell_{\beta}\leq \lim_{n\to\infty}\ell_{\beta_n}$.
Therefore, 
$$
\ell_{\beta}=\lim_{n\to\infty}\ell_{\beta_n}.
$$
In particular, the intersection of $\beta$ with any leaf of $\check{\mu}$ which is not a frontier leaf is reduced to a point.
Consequently, the curve $\beta$ is $\check{\mu}$-piecewise geodesic, as was to be shown.
Thus $\beta\in\mathcal{C}_{\check{\mu}}(S^0)$ and we have $\tau_{S^0}(\beta)=\mathbf{y}$.
Lemma \ref{lemma:tau_injective} implies that this curve $\beta$ is uniquely defined and thus does not depend upon the chosen finite approximation used to define it.
\end{proof}

\subsubsection*{{\bf \S 3.}}
Now set $\textrm{T}_{\check{\mu}}^{0}(S^0):=\left\{\mathbf{x}\in\mathbb{R}^{\underline{\mathcal{W}}^2}\,:\,\Vert \mathbf{x}-\mathbf{x}^0\Vert<\infty \right\}$; this is a subset of $\mathbb{R}^{\underline{\mathcal{W}}^2}$ which is an increasing union of balls centered at $\mathbf{x}^0=\sigma_{\check{\mu}}(S^0)$.
We set $\mathcal{T}_{\check{\mu}}^{0}(S^0):=\sigma_{\check{\mu}}^{-1}\left(\textrm{T}_{\check{\mu}}^{0}(S^0)\right)$.
Let $\beta\in\mathcal{C}_{\check{\mu}}(S^0)$ be a $\check{\mu}$-piecewise geodesic curve in $S^0$ and let $S\in\mathcal{T}_{\check{\mu}}^{0}(S^0)$ be a hyperbolic structure close to $S^0$.
We want to define a way of ``recognizing" the curve $\beta$ in $S$ about the same way we did in the finite case.
We also proceed by approximation.

Let $(I_{n})_n$ be a finite approximation of $I$.
Let $(S^0_{n},\check{\mu}_{n})$ be a geometric realization of $I_{n}$ converging to $(S^0,\check{\mu})$.
Any curve $\beta\in\mathcal{C}_{\check{\mu}}(S^0)$ in $S^0$ induces a $\check{\mu}_n$-piecewise geodesic curve $\beta_{n}^{S^0_{n}}$ in $S^0_{n}$.
The process described in Section \ref{subsection:finiteconvexity} defines for each $n$ a $\check{\mu}_n$-piecewise geodesic curve $\beta_{n}^{S_{n}}$ in $S_{n}$, where $S_{n}$ is the strip $S^0_{n}$ sheared along $\check{\mu}_{n}$.
Since shears here are bounded, the very same proof as that of Lemma \ref{lemma:convergence_pwgeodesic} shows that $\ell_{\beta_{n}^{S_{n}}}\leq\ell_{\beta_{n}^{S^0_{n}}}+\Vert \mathbf{x}-\mathbf{x}^0\Vert$, hence $\ell_{\beta_{n}^{S_{n}}}\leq\ell_{\alpha}+\Vert \mathbf{x}-\mathbf{x}^0\Vert+2\Vert \mathbf{y}-\mathbf{y}^0\Vert$, where $\alpha$ is the geodesic segment with same endpoints as $\beta$.
Therefore, as this is explained in the proof of Lemma \ref{lemma:convergence_pwgeodesic}, there exists, up to taking a subsequence, a unique $\check{\mu}$-piecewise geodesic curve $\beta^{S}$ with finite length to which the sequence $(\beta_{n}^{S_{n}})$ converges uniformly.
Moreover, $\lim_{n\to\infty}\ell_{\beta_{n}^{S_{n}}}=\ell_{\beta^{S}}$ and, more precisely, the length of the wedge part, respectively the leaf part, of $\beta_n$ converges to the length of the wedge part, respectively the leaf part, of $\beta$.
We naturally identify the limit curve $\beta^{S}$ as the realization of $\beta$ in the strip $S$ (and we shall keep up with the notation $\beta$).

We can define the continuous bijective map
\begin{eqnarray*}
\tau\,:\,\mathcal{T}_{\check{\mu}}^{0}(S^0)\times\mathcal{C}_{\check{\mu}}^{0}(S^0)&\longrightarrow& \textrm{T}_{\check{\mu}}^{0}(S^0)\times\textrm{C}_{\check{\mu}}^{0}(S^0)\subset\mathbb{R}^{\underline{\mathcal{W}}^2}\times\left(\mathbb{R}^{\mathcal{W}}\right)^2\\
(S,\beta)&\mapsto& ({\bf x},{\bf y})
\end{eqnarray*}
such that $\tau(S^0,\cdot)=\tau_{S^0}(\cdot)$.

\subsubsection*{{\bf \S 4.}}
Length function
$$
\underline{\ell}\,:\,\textrm{T}_{\check{\mu}}^{0}(S^0)\times\textrm{C}_{\check{\mu}}^{0}(S^0)\longrightarrow\mathbb{R}_{+},
$$
is given by
$$
\underline{\ell}(\mathbf{x},\mathbf{y})=\ell_{\beta}(S),
$$
where $(S,\beta)=\tau^{-1}(\mathbf{x},\mathbf{y})$.
Note that it is finite over the neighbourhood $\textrm{T}_{\check{\mu}}^{0}(S^0)\times\textrm{C}_{\check{\mu}}^{0}(S^0)$ of $(\mathbf{x}^0,\mathbf{y}^0)=\tau(S^0,\alpha)$.

We now show the strict convexity of length function.

\begin{proposition}
\label{proposition:infiniteconvex}
Length function $\underline{\ell}$ is strictly convex over $\textrm{T}_{\check{\mu}}^{0}(S^0)\times\textrm{C}_{\check{\mu}}^{0}(S^0)$.
\end{proposition}

\begin{proof}
Let $(\mathbf{x},\mathbf{y})\in \textrm{T}_{\check{\mu}}^{0}(S^0)\times\textrm{T}_{\check{\mu}}^{0}(S^0)$ and $(S,\beta)=\tau^{-1}(\mathbf{x},\mathbf{y})$.
Let $(I_{n})$ be a finite approximation of $I$ and $(S_{n},\check{\mu}_{n})$ be a geometric realization of $I_{n}$ with respect to $(S,\check{\mu})$.
The curve $\beta$ in $S$ induces the sequence of $\check{\mu}_n$-piecewise geodesic curves $(\beta_{n})$ in $S_{n}$ which converges in $\mathbb{H}^2$ uniformly to $\beta$.
For every $n$ we decompose the length of $\beta_{n}$ into its wedge and leaf part:
$$
\ell_{\beta_{n}}=\ell_{\beta_{n}^{w}}+\ell_{\beta_{n}^{l}}.
$$
The length of the wedge part $\ell_{\beta_{n}^{w}}$ is the sum of the length of wedge segments $w_i$ ($i\in I_n$):
$$
\ell_{\beta_{n}^{w}}=\sum_{i\in I_{n}}\ell_{w_i}.
$$
Lemma \ref{lemma:convergence_pwgeodesic} implies that the wedge part of the limit curve $\beta$ satisfies
$$
\ell_{\beta^{w}}=\sum_{i\in I}\ell_{w_i}.
$$
Each length function $\ell_{w_i}$ seen as a function of $(\mathbf{x},\mathbf{y})$ is strictly convex along affine paths that do not keep the endpoints of $w_{i}$ fixed.
Since there is at least one $w_i$ whose endpoints do not remain fixed along such an affine path in $\textrm{T}_{\check{\mu}}^{0}(S^0)\times\textrm{C}_{\check{\mu}}^{0}(S^0)$, this shows that $\ell_{\beta^{w}}$ seen as a function of $(\mathbf{x},\mathbf{y})$ is \emph{strictly} convex. 

Similarly, for each $n$, $\ell_{\beta_{n}^{l}}$ is the sum of the leaf segment lengths: $\ell_{\beta_{n}^{l}}=\sum_{i\in I_n}\ell_{l_i^n}$.
As explained in the previous section, $\ell_{l_i^n}$ seen as a function of $(\mathbf{x},\mathbf{y})$ is a convex function for all $j$ and $n$.
Therefore, for each $n$, $\ell_{\beta_{n}^{l}}$ is a convex function.
Passing to the limit, we conclude that $\ell_{\beta^{l}}$ seen as a function of $(\mathbf{x},\mathbf{y})$ is convex.

The length of $\beta$ is the sum $\ell_{\beta_{n}^{w}}+\ell_{\beta_{n}^{l}}$.
We conclude that $\underline{\ell}$ is strictly convex.
\end{proof}

For each pair of numbers $(y^-,y^+)$, let $\alpha$ be the geodesic segment in $S^0$ connecting the points encoded by this pair.
Let $\textrm{C}_{\check{\mu}}^{0,\alpha}(S^0)$ be the subset of those ${\bf y}\in\textrm{C}_{\check{\mu}}^{0}(S^0)$ such that the curves $\beta=(\tau_{S^0})^{-1}({\bf y})$ and $\alpha$ have the same endpoints on $\partial S^0$.
These boundary conditions are expressed as linear combinations in the same way as in the finite case above.

Now note that the infima
$$
\underline{\ell}({\bf x},y^{-},y^{+}):=\inf_{{\bf y}\in\textrm{C}_{\check{\mu}}^{0,\alpha}(S^0)}\underline{\ell}({\bf x},{\bf y})
$$ 
and
$$
\underline{\ell}({\bf x}):=\inf_{{\bf y}\in\textrm{C}_{\check{\mu}}^{0}(S^0)}\underline{\ell}({\bf x},{\bf y})
$$ 
are actually attained.

At this point, the same proof as Lemma \ref{corollary:finiteconvexalpha} establishes that length function is strictly convex. 
Therefore we get the following results.

\begin{proposition}
\label{proposition:convex}
The length of the geodesic segment $\alpha$ is a strictly convex function in shear coordinates $\mathbf{x}$ and in signed distances of its endpoints.
\end{proposition}

and 

\begin{proposition}
\label{proposition:core_convex}
The length of the core of a strip is a strictly convex function in shear coordinates $\mathbf{x}$.
\end{proposition}

\subsection{Convexity of length function of a measured geodesic lamination over Teichm\"uller space}

Now we establish the strict convexity for the length function $\ell_{\lambda}$ of a measured geodesic lamination $\lambda$ over Teichm\"uller space 
in terms of shear coordinates given by a complete geodesic lamination $\mu$.

The idea is to replace $\mu$ by a train-track approximation $\Theta$ of $\mu$.
Then $\lambda$ decomposes as a finite union of (in general uncountably many) segments contained in $\Theta$ and of (in general uncountably many) segments contained outside $\Theta$.
Each segment of the former type is called a branch segment and each segment of the latter type is called a triangle segment.
The length of $\lambda$ is then the sum over the branches of $\Theta$ and over the triangle regions of $\Sigma\setminus\Theta$ of the lengths of the branch and triangle segments.
We are therefore led to consider the set of piecewise geodesic measured laminations which are concatenation of branch and triangle segments and which are in the class of $\lambda$.
The length of $\lambda$ is the infimum of the lengths of those piecewise geodesic laminations (Theorem \ref{theorem:minimal_length}).

\subsection{Piecewise geodesic laminations measure equivalent to and close to $\lambda$.}
\label{subsection:piecewisegeodesic}

Let $\mu$ be a complete geodesic lamination of the surface $\Sigma$ and $\lambda$ be a measured geodesic lamination transverse to $\mu$.
We fix a base hyperbolic structure $h_{0}$ in $\mathcal{T}(\Sigma)$.

Consider an $\epsilon$-train-track approximation $\Theta=\Theta(\mu, \lambda, h_{0},\epsilon)$ of $\mu$ in $h_{0}$ such that the leaves of $\lambda$ cross the branches of $\Theta$ while avoiding the singular traverses.
This is possible for $\epsilon$ small enough because there is a uniform lower bound to the angles in which the leaves of $\lambda$ intersect those of $\mu$.
Consider also the train-track $\hat{\Theta}$ associated to $\Theta$ in Section \ref{subsection:thicktraintrack}.

Let $\eta>0$ be a small number compared to the widths of the branches of $\Theta$.
Let $(\lambda_{pw})=(\lambda_{pw})(\eta,\lambda,\Theta)$ be the set of \emph{piecewise geodesic} laminations $\alpha$ that are in the measure class of $\lambda$ and whose leaves satisfy the following conditions:
\begin{itemize}
\item For any branch $b$ of $\hat{\Theta}$, the components of $\alpha\cap b$ are geodesic segments lying at a distance at least $\eta$ away from the singular traverses; we name them \textbf{branch segments}; 
\item The components of $\alpha\setminus\alpha\cap\hat{\Theta}$ are geodesic segments we name \textbf{triangle segments}.
\end{itemize}
In other words, the singular points of the leaves of $\alpha$ (if any) only belong to the interior of geodesic components of the boundary of $\hat{\Theta}$.
Each leaf of $\alpha$ is therefore a concatenation of branch and triangle segments.
Note that the leaves of such a ``lamination" are allowed to intersect one another.

Note that, from the choice of $\Theta$, the measured geodesic lamination $\lambda$ belongs to $(\lambda_{pw})$.
We equip this space with the Hausdorff topology.\\

\noindent{\it Nota bene}: As aforesaid, leaves of ``laminations" in $(\lambda_{pw})$ might not be simple nor disjoint.\\

\begin{figure}
\centering
\psfrag{a}{\tiny{-2.1}}
\psfrag{b}{\tiny{3.2}}
\includegraphics[width=.5\linewidth]{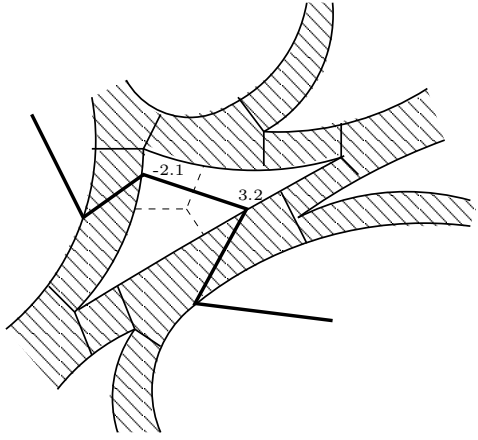}
\caption{This picture depicts a shadded piece of $\hat{\Theta}$ in which a triangle region $\Delta$ and several branches are seen. A piece $a$ of a leaf of an element $\alpha\in(\lambda_{pw})$ is also shown. The three dashed segments join the centre of the ideal triangle containing $\Delta$ to the corresponding distinguished points. The position of the two points of $a$ on the geodesic boundary of $\hat{\Theta}$ (or of $\Delta$) are given by two real numbers.}
\label{figshear6}
\end{figure}

Let us parameterize the set $(\lambda_{pw})$.
A lamination $\alpha$ in $(\lambda_{pw})$ is completely understood through the positions of its intersection points with the geodesic boundary of $\hat{\Theta}$.
The position of every such point is given by a real number that equals the signed distance of this point to the corresponding distinguished point, with the usual sign convention.
The number of leaves of a generic lamination is uncountable.
Because there is a canonical correspondence between leaves of any $\alpha\in(\lambda_{pw})$ and those of $\lambda$, if $b_{1},\ldots, b_{m}$ denote the branches of $\hat{\Theta}$ that are crossed through by $\lambda$ and if $K_{1},\ldots,K_{m}$ denote the leaf sets of $\lambda\cap b_{1},\ldots,\lambda\cap b_{m}$, then there is a canonical correspondence between the leaf sets of $\alpha\cap b_{j}$ and the leaf set $K_{j}$ ($1\leq j \leq m$).
(These sets are either finite or Cantor sets.)
Hence the endpoints of a branch segments of $\alpha$ contained in $b_{j}$ define a point of $\mathbb{R}^{K_{j}}\times\mathbb{R}^{K_{j}}$ (one factor per geodesic side of $b_{j}$).
We thus have a map
$$
\tau_{h_{0}}\,:\,(\lambda_{pw})\longrightarrow\left(\mathbb{R}^{K_{1}}\right)^{2}\times\cdots\times\left(\mathbb{R}^{K_{m}}\right)^{2}.
$$
We equip the Cartesian products with the weak$^\star$ topology.
Note that this space is a vector space.
With this topology, it is rather clear that $\tau_{h_{0}}$ is continuous.
The geodesic lamination $\lambda$ induces canonical bijections between the various sets $K_{j}$, encoding how the points on the boundary of $\hat{\Theta}$ are connected.
Therefore, given a point in $\left(\mathbb{R}^{K_{1}}\right)^{2}\times\cdots\times\left(\mathbb{R}^{K_{m}}\right)^{2}$ such that the corresponding points on $\partial\hat{\Theta}$ still belong to the geodesic sides of $\hat{\Theta}$, the connection pattern given by the bijections above yields a lamination in $(\lambda_{pw})$ in a unique way (recall that we allow the leaves to intersect each other, even themselves).
This shows that the image of the map $\tau_{h_{0}}$ is an open subset and that the map is a homeomorphism onto its image; we denote this image by $(\lambda)_{pw}(\Theta)=(\lambda)_{pw}(\Theta,\lambda,\eta,h_{0})$.

\subsection{Varying the hyperbolic structure about $h_{0}$.}
What happens when we slightly change the hyperbolic structure on $\Sigma$?
Before tackling the general case, let us focus on the case where the complete lamination $\mu$ has finitely many leaves.
We shall loosely describe this case, the omitted details will be provided while the general case will be discussed;
it is mainly a repetition of Section \ref{subsection:finiteconvexity} in the context of hyperbolic surfaces.

In this special case, each branch segment of $\alpha$ is entirely contained in a leaf of $\mu$.
Let us consider one leaf of $\alpha$ and one branch segment of this leaf.
This branch segment is contained in a geodesic side common to two adjacent ideal triangles $T_{+}, T_{-}$.
Let $p,q$ be the endpoints of this branch segment, 
labelled so that $p$ is connected to a triangle segment in $T_{+}$ and $q$ is connected to a triangle segment in $T_{-}$.
Let $y_{+}$ denote the coordinate of $p$ with respect to the distinguished point of $T_{+}$ and let $y_{-}$ be the coordinate of $q$ with respect to the distinguished point of $T_{-}$.
Let $\sigma(h_{0})$ be the shear between $T_{+}$ and $T_{-}$ for the base structure $h_{0}$ and let $\sigma(h)$ be the shear between the same two ideal triangles for a nearby structure $h$.
As we pass from the structure $h_{0}$ to the structure $h$, the triangles $T_{+}$ and $T_{-}$ are shifted from one the other by the amount $\sigma(h)-\sigma(h_{0})$.
If one represents both ideal triangles $T_{+}$ and $T_{-}$ in the hyperbolic plane, their shift can be described by fixing the common geodesic and by shifting each ideal triangle by the amount of $\frac{1}{2}(\sigma(h)-\sigma(h_{0}))$ as explained in \ref{subsection:finiteconvexity}.
Endpoints of triangle segments keep their respective positions within each ideal triangle whereas endpoints of branch segments are shifted by the amount of $\frac{1}{2}(\sigma(h)-\sigma(h_{0}))$.
Hence
$$
y_{\pm}(h)=y_{\pm}+\frac{1}{2}(\sigma(h)-\sigma(h_{0})).
$$
In order to get a connected curve for the structure $h$, we move the endpoints of the \emph{triangle segments} over the endpoints of the branch segment (cf. \ref{subsection:finiteconvexity}).
This reconnection process defines the image $\alpha(h)$ for the hyperbolic structure $h$ of the piecewise geodesic lamination $\alpha$.
This enables one to define the length of the lamination $\alpha$ with respect to the hyperbolic structure $h$ (which is the length of $\alpha_h$).
Note that in this case, only the lengths of the triangle segments change while the length of the branch segments remain the same. 
This state of fact won't persist as we move to the general case where lengths of branch segments will undergo changes as well, due to the varying thickness of the branches of the train-track $\Theta$ as the structure $h_{0}$ gets sheared.\\

Let us now deal with the general case with an arbitrary complete geodesic lamination $\mu$.
Consider hyperbolic structures $h$ on $\Sigma$ such that the $\epsilon$-train-track approximation $\Theta_{h}$ of $\mu$ for these structures $h$ is of the same topological type as $\Theta$ and such that the amplitude of the shears of $h$ compared to those of $h_{0}$ is smaller that $\eta/2$ for all branches of $\Theta$.
The set of such structures $h$ defines a neighbourhood $\mathcal{T}_{\eta}(\Theta)$ of $h_{0}$ in $\mathcal{T}(\Theta)$, namely, $\mathcal{T}_{\eta}(\Theta)=\{h\in\mathcal{T}(\Theta)\ :\ \vert\vert\sigma(h)-\sigma(h_{0})\vert\vert<\eta/2\}$.
Let us now consider an element $\alpha\in(\lambda_{pw})$ and a branch segment $a$ of it contained in the branch $b_{j}$ of $\hat{\Theta}$ ($j\in\{1,\cdots,m\}$).
The endpoints of $a$ give, via the map $\tau_{h_{0}}$, two real numbers $(y_{j}^{-},y_{j}^{+})$ encoding their positions with respect to the corresponding distinguished points.
A shear coordinate $\sigma_{j}$ is also associated to the branch $b_{j}$.
Now let $h$ be an element of $\mathcal{T}_{\eta}(\Theta)$.
We set
$$
y_{j}^{\pm}(h):=y_{j}^{\pm}+\frac{1}{2}(\sigma^{\pm}_{j}(h)-\sigma^{\pm}_{j}(h_{0})),
$$
where $\sigma^{\pm}_{j}(\cdot)$ is the shear between the ideal triangle whose side is $w^{\pm}$ and a \emph{fixed} ideal triangle crossing the branch $b_{j}$.
Note that $\sigma^{+}_{j}(\cdot)+\sigma^{-}_{j}(\cdot)=\sigma_{j}(\cdot)$.

Since we took care of bounding the intensity of shears by $\eta/2$, the newly defined points for the structure $h$ via their coordinates are still inside the geodesic boundary of $\hat{\Theta}_{h}$.
Using the reconnection pattern between the various sets $K_{j}$ given by $\lambda$, we get a measured lamination $\alpha_{h}$ of the surface $\Sigma$ with the structure $h$ which belongs to the measure class of $\lambda$.
If we denote by $(\lambda_{pw})_{h}$ the space of piecewise geodesic laminations with respect to $\Theta_{h}$ that belong to the measure class of $\lambda$, we get, through the map $\alpha\mapsto\alpha_{h}$, an identification between $(\lambda_{pw})$ and $(\lambda_{pw})_{h}$, for $h$ close enough to $h_{0}$.
We also get a continuous map
$$
\tau\,:\,\mathcal{T}_{\eta}(\Theta)\times(\lambda_{pw})\longrightarrow\left(\mathbb{R}^{K_{1}}\right)^{2}\times\cdots\times\left(\mathbb{R}^{K_{m}}\right)^{2},
$$
with $\tau(h_{0},\cdot)=\tau_{h_{0}}$, which associates to $(h,\alpha)$ the positions of the singular points of the leaves of $\alpha_{h}$ on the sides of ideal triangles, with respect to the corresponding distinguished points.

\noindent{\it Nota bene}: The image of $\lambda$ in $(\lambda_{pw})_{h}$ is \emph{not} $\lambda$, that is, $\lambda_{h}\neq\lambda$ in $(\lambda_{pw})_{h}$.

\subsection{Length function over $(\lambda_{pw})$}

Let $\alpha$ be an element of $(\lambda_{pw})$ and $h$ be a hyperbolic structure of $\mathcal{T}_{\eta}(\Theta)$.
Let us define the length $\ell_{\alpha}(h)$ of $\alpha$ with respect to the hyperbolic structure $h$.

Equip the surface $\Sigma$ with the structure $h$ and consider the $\epsilon$-train-track approximation $\Theta_{h}$ as defined above.
Let $y_{j}$ ($1\leq j\leq m$) be the coordinates of the point ${\bf y}=\tau(h,\alpha)\in\left(\mathbb{R}^{K_{1}}\right)^{2}\times\cdots\times\left(\mathbb{R}^{K_{m}}\right)^{2}$. 
For each branch $b_{j}$ of $\hat{\Theta}_{h}$ crossed through by leaves of $\alpha$ (or $\lambda$), the numbers $y_{j}\in\left(\mathbb{R}^{K_{j}}\right)^{2}$ define points on both geodesic sides of $b_{j}$: 
if $x$ is a point of $K_{j}$ then the pair of real numbers $y_{j}(x)$ represents the coordinates of two points whose distance apart we denote by $\ell_{x}(h)$.
The transverse measure $d\lambda$ of $\lambda$ defines a measure on each leaf set $K_{j}$ (see Section \ref{subsection:moreinfo}).
The geodesic sides of $\hat{\Theta}_{h}$ provide an arc system for $\alpha$ (Section \ref{section:length}) and the length of $\alpha\cap b_{j}$ is given by
$$
\ell_{\alpha\cap b_{j}}(h)=\int_{K_{j}}\ell_{x}(h)d\lambda(x).
$$
We proceed in the same manner with the triangle segments of $\alpha$ in each triangle component $\Delta$ of $\Sigma\setminus\hat{\Theta}_{h}$.
This case is actually simpler since each triangle component of $\Sigma\setminus\hat{\Theta}_{h}$ is isometric to a triangle component of $\Sigma\setminus\hat{\Theta}$.
We choose three arcs, $\delta_{1},\delta_{2},\delta_{3}$, transverse to the triangle segments of $\alpha\cap\Delta$ such that all triangle segments going from one given side of $\Delta$ to another are intersected by the same arc $\delta_{i}$, $i\in\{1,2,3\}$.
The $h$-length of $\alpha\cap\Delta$ of all triangle segments of $\alpha$ crossing $\Delta$ is given by
$$
\ell_{\alpha\cap\Delta}(h)=\sum_{i=1}^{3}\int_{\delta_{i}\cap\alpha}\ell_{x}(h)d\lambda(x).
$$

Summing all pieces gives the length of $\alpha$:
$$
\ell_{\alpha}(h)=\sum_{\Delta}\ell_{\beta\cap\Delta}(h)+\sum_{b\in B}\ell_{\beta\cap b}(h).
$$

The space $(\lambda_{pw})$ is parameterized by $(\lambda_{pw})(\Theta)$ while $\mathcal{T}_{\eta}(\Theta)$ by an open subset of ${\mathrm T}(\Theta)$. 
In order to lighten notation, we still denote this open subset by ${\mathrm T}(\Theta)$, forgetting the restriction represented by $\eta$.
The length functional $\ell$ can thus be viewed as a continuous map
$$
\underline{\ell}\,:\,{\mathrm T}(\Theta)\times(\lambda_{pw})(\Theta)\longrightarrow\mathbb{R}_{+},
$$
with $\underline{\ell}({\bf x},{\bf y})=\ell_{\alpha_{h}}(h)$ where $h=\sigma^{-1}({\bf x})$ and $\alpha_{h}$ obtained by connecting the points given by $\tau(h,\tau_{h_{0}}^{-1}({\bf y}))$.

\subsection{strict convexity of length function over $(\lambda_{pw})$}

We briefly recapitulate the setting.
Let $\mu$ be a \emph{complete} geodesic lamination of the surface $\Sigma$, $\lambda$ be a \emph{measured} geodesic lamination transverse to $\mu$ and $h_{0}$ a fixed base hyperbolic structure on $\Sigma$.
Let $\Theta$ be a train-track approximation of $\mu$ which is ``good" with respect to $\lambda$ in the sense of Section \ref{subsection:piecewisegeodesic}. 
Let $(\lambda_{pw})$ be the set of laminations in the measure class of $\lambda$ which are close enough to $\lambda$ in the Hausdorff topology and piecewise-geodesic with respect to $\Theta$.
Let $\sigma_{\Theta}\,:\,\mathcal{T}(\Theta)\to{\mathrm T}(\Theta)$ be the \emph{local} shear coordinates associated to the train-track approximation $\Theta$.
Let $\tau_{\Theta}\,:\,(\lambda_{pw})\to(\lambda_{pw})(\Theta)$ be the map that gives the positions on $\hat{\Theta}$ of the endpoints of the branch and triangle segments of elements of $(\lambda_{pw})$.
The set of branches of $\Theta$ is denoted by $B$.
We index the branches $b_{1},\cdots, b_{n}$ with $n=\vert B\vert$, so that the $m$ first branches ($m\leq n$) are crossed through by $\lambda$.
Recall that ${\mathrm T}(\Theta)$ is an open subset of $\mathbb{R}^{n}$ and that $(\lambda_{pw})(\Theta)$ is an open subset of $\left(\mathbb{R}^{K_{1}}\right)^{2}\times\cdots\times\left(\mathbb{R}^{K_{m}}\right)^{2}$, where $K_{j}$ is the leaf set of $\lambda\cap b_{j}$.
The length functional
$$
\underline{\ell}\,:\,{\mathrm T}(\Theta)\times(\lambda_{pw})(\Theta)\longrightarrow\mathbb{R}_{+}
$$
associates to the coordinates $({\bf x},{\bf y})$ the length of the lamination $\alpha:=\tau_{\Theta}^{-1}({\bf y})$ with respect to the hyperbolic structure $h:=\sigma_{\Theta}^{-1}({\bf x})$ (denoted by $\alpha_{h}$).\\

Here is the central result of our paper.
\begin{proposition}
\label{proposition:convexity1}
The length functional $\underline{\ell}$ is a strictly convex function of the coordinates $({\bf x},{\bf y})\in{\mathrm T}(\Theta)\times(\lambda_{pw})(\Theta)$.
\end{proposition}

The length $\underline{\ell}({\bf x},{\bf y})=\ell_{\alpha}(h)$ of $\alpha$ for the structure $h$ is the sum of the lengths of the branch segments contained in each branch of $\hat{\Theta}_{h}$ and of the triangle segments contained in each component of $\Sigma\setminus\hat{\Theta}_{h}$.
Let us focus on the length of each type of such segments.

\begin{lemma}
\label{lemma:triangle}
The length of a triangle segment is a strictly convex function of the coordinates $({\bf x},{\bf y})\in{\mathrm T}(\Theta)\times(\lambda_{pw})(\Theta)$.
\end{lemma}

\begin{lemma}
\label{lemma:branch}
The length of a branch segment is a convex function of the coordinates $({\bf x},{\bf y})\in{\mathrm T}(\Theta)\times(\lambda_{pw})(\Theta)$ which is strictly convex if the branch segment is not entirely contained in a leaf of $\mu$.
\end{lemma}

The case of a triangle segment is a straightforward consequence of the strict convexity of the hyperbolic distance function (see Proposition \ref{proposition:d_convex} of Section \ref{subsection:wedgeconvexity}).
The case of a branch segment is harder because its endpoints belong to geodesics that vary with the hyperbolic structure so it is not possible to deduce it directly from the strict convexity of the distance function, as opposed to the triangle segment case.
However this situation has been handled in Section \ref{subsection:finiteconvexity} and \ref{subsection:infiniteconvexity}.
Lemma \ref{lemma:branch} is a straightforward consequence of Proposition \ref{proposition:convex}: it suffices to lift the branch $e$ considered to the universal covering and to consider the strip extending this lift.
The restrictions imposed upon the hyperbolic structures $h$ match those imposed upon hyperbolic structures on strips.

\begin{proof}[Proof of Proposition \ref{proposition:convexity1}]
By Lemmas \ref{lemma:triangle} \& \ref{lemma:branch}, the length of a branch or triangle segment is a strictly convex function over ${\mathrm T}(\Theta)\times(\lambda_{pm})(\Theta)$.
Integrating these lengths against $d\lambda$ still yields strictly convex functions. 
The sum of all these functions is the length functional $\underline{\ell}$ and is strictly convex.
\end{proof}

We finally get our main result.

\begin{theorem}
\label{theorem:main}
Let $\lambda$ be a measured geodesic lamination transverse to the complete geodesic lamination $\mu$ of $\Sigma$.
Length function 
\begin{eqnarray*}
\ell_{\lambda}\,:\,\mathcal{T}(\Sigma)&\to&\mathbb{R}_{+},\\
h&\to&\ell_{\lambda}(h)
\end{eqnarray*}
is a convex function of shear coordinates defined over ${\mathrm T}(\Sigma)$ by $\mu$ which is strictly convex whenever $\lambda$ intersects all leaves of $\mu$.  
\end{theorem}

\begin{proof}
Firstly, for any given fine enough train-track approximation $\Theta$ of $\mu$, length function $\ell_{\lambda}$ is strictly convex in the shear coordinates over the open subset ${\mathrm T}(\Theta)$ of ${\mathrm T}(\Sigma)$:
we know from Proposition \ref{proposition:convexity1} that the length functional $\underline{\ell}$ is strictly convex over 
${\mathrm T}(\Theta)\times(\lambda_{pm})(\Theta)$. 
By Theorem \ref{theorem:minimal_length} and because $\lambda\in(\lambda_{pm})_{h}$ for all $h\in\mathcal{T}(\Theta)$, we have 
$$
\ell_{\lambda}(h)=\underline{\ell}({\bf x}):=\inf_{{\bf y}\in(\lambda_{pm})(\Theta)}\underline{\ell}({\bf x},{\bf y}),
$$
where ${\bf x}=\sigma_{\Theta}(h)\in{\mathrm T}(\Theta)$.
It is then easily checked that $\ell_{\lambda}$ is strictly convex indeed (see the proof of Proposition \ref{proposition:convex}).
\begin{comment}
we extend the function $\underline{\ell}$ over the vector space $V:=\mathbb{R}^{n}\times\big{(}\mathbb{R}^{K_{1}}\big{)}^{2}\times\cdots\times\big{(}\mathbb{R}^{K_{m}}\big{)}^{2}$ containing ${\mathrm T}(\Theta)\times(\lambda_{pm})(\Theta)$ as an open subset.
Let $({\bf x},{\bf y})$ and $({\bf x'},{\bf y'})$ be two points in $V$.
Let ${\bf y_{m}}$ and ${\bf y'_{m}}$ be the coordinates realizing the minimum of $\ell$ for the structures ${\bf x}$ and ${\bf x'}$ respectively, that is, 
$$
\min_{{\bf y}}\underline{\ell}({\bf x},{\bf y})=\underline{\ell}({\bf x},{\bf y_{m}})=\underline{\ell}({\bf x})\ \textrm{and}\ \min_{{\bf y'}}\underline{\ell}({\bf x'},{\bf y'})=\underline{\ell}({\bf x'},{\bf y'_{m}})=\underline{\ell}({\bf x'}).
$$
Note that ${\bf y_{m}}$ and ${\bf y'_{m}}$ both belong to $(\lambda_{pm})(\Theta)$.
For all $t\in[0;1]$, one has 
\begin{eqnarray*}
\underline{\ell}\left((1-t){\bf x}+t{\bf x'}\right)&\leq&\underline{\ell}((1-t){\bf x}+t{\bf x'},(1-t){\bf y_{m}}+t{\bf y'_{m}})\\
&\leq&\underline{\ell}((1-t)({\bf x},{\bf y_{m}})+t({\bf x'},{\bf y'_{m}}))\\
&<&(1-t)\underline{\ell}({\bf x},{\bf y_{m}})+t\underline{\ell}({\bf x'},{\bf y'_{m}})\\
&<&(1-t)\underline{\ell}({\bf x})+t\underline{\ell}({\bf x'}).
\end{eqnarray*}
\end{comment}

We then know that the function $\underline{\ell}$ is strictly convex over the open subset ${\mathrm T}(\Theta)$ of ${\mathrm T}(\Sigma)$ (or $\ell_{\lambda}$ is strictly convex over the open subset $\mathcal{T}(\Theta)$ of $\mathcal{T}(\Sigma)$).
We choose a sequence $(\Theta_{n})_{n}$, with $\Theta_{0}=\Theta$, of finer and finer train-track approximations of $\mu$ converging to $\mu$.
This sequence defines an exhaustion of $\mathcal{T}(\Sigma)$ by strictly increasing open subsets $\mathcal{T}(\Theta_{n})$ and, using the local coordinates $\sigma_{\Theta_{n}}\,:\,\mathcal{T}(\Theta_{n})\to{\mathrm T}(\Theta_{n})$ and the associated transport maps $p_{n}\,:\,{\mathrm T}(\Theta_{n})\to{\mathrm T}(\Theta)$ as explained in Section \ref{subsection:global_shear}, we get shear coordinates $\widetilde{\sigma}_{\Theta}$ over the whole Teichm\"uller space
$$
\widetilde{\sigma}_{\Theta}\,:\,\mathcal{T}(\Sigma)\to{\mathrm T}(\Theta)
$$
with $\widetilde{\sigma}_{\Theta}=p_{n}\circ\sigma_{\Theta_{n}}$ over ${\mathrm T}(\Theta_{n})$.

Let us show that $\ell_{\lambda}$ is strictly convex on $\mathcal{T}(\Sigma)$, that is, that the function $\underline{\ell}$ is strictly convex over ${\mathrm T}(\Theta)$.
Let $h,h'$ be two points of Teichm\"uller space $\mathcal{T}(\Sigma)$ and let $\widetilde{\bf x}=\widetilde{\sigma}_{\Theta}(h)$ and $\widetilde{\bf x}'=\widetilde{\sigma}_{\Theta}(h')$ be the corresponding coordinates in ${\mathrm T}(\Theta)$.
The image in ${\mathrm T}(\Theta)$ of the affine segment $\{t\widetilde{\bf x}+(1-t)\widetilde{\bf x}'\,:\,t\in[0\,;1]\}$ by $\widetilde{\sigma}_{\Theta}^{-1}$ is compact so there exists an integer $N$ such that this image is contained in all patches $\mathcal{T}(\Theta_{n})$, $n\geq N$.
Fix such an $n$ and set ${\bf x}=\sigma_{\Theta}(h)$ and ${\bf x}'=\sigma_{\Theta}(h')$.
By definition,
$$
\widetilde{\bf x}=p_{n}({\bf x})\ \textrm{and}\ \widetilde{\bf x}'=p_{n}({\bf x}').
$$ 
Thus, for all $t\in[0\,;1]$, one has 
\begin{align*}
\underline{\ell}\left(t\widetilde{\bf x}+(1-t)\widetilde{\bf x}'\right)&=\underline{\ell}\left(tp_{n}({\bf x})+(1-t)p_{n}({\bf x'})\right)\\
&=\underline{\ell}\left(p_{n}(t{\bf x}+(1-t){\bf x'})\right)\\
&=\underline{\ell}_{n}\left(t{\bf x}+(1-t){\bf x'}\right)\\
&<t\underline{\ell}_{n}({\bf x})+(1-t)\underline{\ell}_{n}({\bf x'})\\
&<t\underline{\ell}(p_{n}({\bf x}))+(1-t)\underline{\ell}(p_{n}({\bf x'}))\\
&<t\underline{\ell}(\widetilde{\bf x})+(1-t)\underline{\ell}(\widetilde{\bf x}').
\end{align*}
This concludes the proof
\end{proof}

\section{An application to stretch lines}

A {\bf stretch line} is an oriented curve $\gamma$ in Teichm\"uller space:
\begin{eqnarray*}
\gamma\,:\,\mathbb{R}&\to&\mathcal{T}(\Sigma)\\
t&\mapsto& \gamma(t).
\end{eqnarray*}

Such a line is obtained by fixing a complete geodesic lamination $\mu$ on $\Sigma$, called the \textbf{support} of the stretch line, and then deforming a given hyperbolic structure $h_{0}$ on $\Sigma$ by multiplying the shears between the ideal triangles given by $\mu$ by the factor $e^{t}$.
Stretch lines are geodesic lines for Thurston's asymmetric metric $d_{T}$ on $\mathcal{T}(\Sigma)$ (\cite{Thurston76}) one of whose definition is
$$
d_{T}(g,h)=\log\sup_{\lambda}\frac{\ell_{\lambda}(h)}{\ell_{\lambda}(g)}.
$$

Consider Thurston's shear coordinates on $\mathcal{T}(\Theta,h_{0})$ given by a train-track approximation $\Theta$ of the complete geodesic lamination $\mu$.
In these coordinates, the stretch line with support $\mu$ and passing through the point $h_{0}$ is described by the linear expansion of the shear coordinates of $h_{0}$ by the factor $e^{t}$, where $t$ varies in $\mathbb{R}$.
In particular, the whole stretch line is contained in the patch $\mathcal{T}(\Theta,h_{0})$ of $\mathcal{T}(\Sigma)$.
In order to make a link with our paper's results, we also consider stretch lines parameterized by \textbf{logarithm of arc length}, namely, curves $\underline{\gamma}$
\begin{eqnarray*}
\mathbb{R}_{+}&\to&\mathcal{T}(\Sigma)\\
u&\mapsto& \underline{\gamma}(u)=\gamma(\log u).
\end{eqnarray*}
In other words, the hyperbolic structure $\underline{\gamma}(u)$ is obtained by expanding the shear coordinates of $h_{0}$ \emph{linearly} by the factor $u>0$. 
Because of its parameterization, this curve is no longer a geodesic for the metric $d_{T}$.

This paragraph gives an explanation for the behaviour of the curves drawn in \cite{PapThe07} p.190 representing the variations of the length functions of some geodesic laminations along some stretch lines. 
%These curves were plotted using a computer, but it seems that they were rather deformed.
  
Let us recall the results obtained in \cite{Theret03} \cite{Theret05} concerning the asymptotic behaviour of the length of $\lambda$ along a stretch line.
The \textbf{stump} of a geodesic lamination $\mu$ of $\Sigma$ is the greatest sublamination of $\mu$ that can be equipped with a transverse measure (of full support). 
Denote it by $\mu_{0}$. 
Consider a stretch line $\gamma$ supported by a complete geodesic lamination $\mu$ and let $\mu_{0}$ be the stump of $\mu$. Let $\lambda$ denote the horocyclic lamination associated to $\gamma$, that is, the geodesic lamination obtained from the horocyclic foliation by straightening its leaves into geodesics. 
Our main result implies the following theorem.

\begin{theorem}
The length of a measured geodesic lamination which intersects transversely the stump of a complete geodesic lamination $\mu$ is strictly convex along the stretch line supported by $\mu$ once the line is parameterized by logarithm of arc length.
\end{theorem}

\begin{proof}
Let $\alpha$ be a measured geodesic lamination which intersects transversely the stump $\mu_{0}$ of a stretch line supported by the complete geodesic lamination $\mu$.
Let $\Theta$ be a train-track approximation of $\mu$.
The assumption that $\alpha$ intersects $\mu_{0}$ implies that leaves of $\alpha$ cross branches of $\Theta$ themselves crossed by leaves of $\mu$ with \emph{non-zero} shears, because any segment $a$ transverse to $\mu_{0}$ is necessarily crossed by infinitely many leaves of $\mu$, hence the spikes of the ideal triangles of $\Sigma\setminus\mu$ cross  $a$ infinitely many times, preventing entirely vanishing shears.
So the shear coordinates involved in the computation of the length of $\alpha$ are not constant along the stretch line and vary linearly when the parameterization is by logarithm of arc length.
Our main theorem implies that the length of $\alpha$ is strictly convex.
\end{proof}

This result sheds light upon the variations of the length function 
$$
u\mapsto\ell_{\alpha}(\underline{\gamma}(u))
$$
hence upon the variations of the length function
$$
t\mapsto\ell_{\alpha}(\gamma(t))
$$
since these variations do not depend upon the (positive) parameterization of the stretch line.
Before being more explicit, we recall that we have the following asymptotic behaviour for the length of a measured geodesic lamination $\alpha$ in terms of its intersection pattern with $\mu_{0}$ and $\lambda$ (\cite{Theret03}, \cite{Theret05}):\\

 \begin{tabular}{|c|c|c|c|c|}
  \hline
  $\ $ & $\alpha\cap\mu_{0}=\emptyset$ & $\alpha\cap\mu_{0}=\emptyset$ &
  $\alpha\cap\mu_{0}\neq\emptyset$ & $\alpha\cap\mu_{0}\neq\emptyset$ \\
  $\ $ & $\alpha\cap\lambda=\emptyset$ & $\alpha\cap\lambda\neq\emptyset$ &
  $\alpha\cap\lambda=\emptyset$ & $\alpha\cap\lambda\neq\emptyset$ \\
  \hline
  $t\to+\infty$ & $<\infty$ & $0$ if
  $\alpha\subseteq\lambda$ & $<\infty$ & $0$ if
  $\alpha\subseteq\lambda$ \\
  $\ $ & & $\infty$  if
  $\alpha\nsubseteq\lambda$ & & $\infty$ if
  $\alpha\nsubseteq\lambda$\\
  \hline
  $t\to-\infty$ & $<\infty$ & $<\infty$ & $0$ if
  $\alpha\subseteq\mu_{0}$ & $0$ if
  $\alpha\subseteq\mu_{0}$\\
  $\ $ & & & $\infty$ if
  $\alpha\nsubseteq\mu_{0}$ & $\infty$ if
  $\alpha\nsubseteq\mu_{0}$ \\
  \hline
  \end{tabular}
 \\
 
Note that this asymptotic behaviour is independent of the parameterization of the stretch line.

From this asymptotic behaviour and the strict convexity, we get for instance the following result:

\begin{corollary}
The length of the horocyclic lamination is strictly decreasing and converges to zero as $t$ converges to infinity.
\end{corollary}

\begin{proof}
The length of the horocyclic lamination $\lambda$ converges to $+\infty$ as $t$ converges to $-\infty$ and converges to $0$ as $t$ converges to $+\infty$.
These limits are independent from the way the stretch line is parameterized, 
so we can assume the line parameterized by the logarithm of arc length.
The horocyclic lamination intersects the stump of $\mu$ transversely, so its length function is strictly convex.
This implies that the length of $\lambda$ is strictly decreasing when the stretch line is parameterized by logarithm of arc length, hence when it is parameterized by arc length as well.
The proof is over.
\end{proof}

Some other results can be drawn from the convexity of the length function. For instance, recalling that a compound function $g\circ f$ of two strictly convex functions with $g$ being strictly increasing produces a strictly convex function, we conclude that wherever the length function $t\mapsto\ell_{\alpha}(\gamma(t))$ is stritcly increasing, it is strictly convex.
Inspection of the graphs of these functions suggests that they are also strictly convex when decreasing, but I have no proof of this.

The table above says that whenever $\alpha$ is disjoint from the stump $\mu_{0}$ of $\mu$ and from the horocyclic lamination $\lambda$, then its length function is bounded. 
One might at first sight deduce that this length function is constant, because of strict convexity. This is not true in general however. But, in this case, we have 

\begin{corollary}
Let $\alpha$ be a measured geodesic lamination disjoint from the stump $\mu_{0}$ of $\mu$ and from the horocyclic lamination $\lambda$.
Then its length function $t\mapsto\ell_{\alpha}(\gamma(t))$ along the stretch line $\gamma$ supported by $\mu$ and directed by $\lambda$ is either constant or is strictly decreasing with 
$$
0<\lim_{t\to+\infty}\ell_{\alpha}(\gamma(t))<\lim_{t\to-\infty}\ell_{\alpha}(\gamma(t))<\infty.
$$
\end{corollary}
 
Note that in the situation of this corollary, the length function is no longer strictly convex. 
This is due to the fact that the factor $e^{t}$ destroys here the strict convexity which is true for $u\mapsto\ell_{\alpha}(\underline{\gamma}(u))$. 

Another type of transformations can be considered, namely earthquake deformations along a \emph{measured} geodesic lamination $\gamma_{0}$.
If one completes $\gamma_{0}$ into $\mu$ such that the stump of $\mu$ is $\gamma_{0}$, the earthquake deformation in shear coordinates associated to $\mu$ are given by
$$
t\mapsto ({\bf x}_{0}\pm t, {\bf x}'),
$$
where ${\bf x}'$ are the shear coodinates associated to the isolated leaves of $\mu\setminus\gamma_{0}$ and where the plus sign corresponds to a left earthquake whereas the minus sign is for a right earthquake.
Our results shows that the length of any measured geodesic lamination transverse to $\gamma_{0}$ is strictly convex along an earthquake supported by $\gamma_{0}$.
So we recover this celebrated result of Kerckhoff's.


\begin{thebibliography}{3}

\bibitem{BBFS09} Bestvina, M., Bromberg, K., Fujiwara, K., Souto, J.:
\textsl{Shearing coordinates and convexity of length functions on Teichm\"uller space.}
arXiv:0902.0829v1.


\bibitem{Bonahon88} Bonahon, F.:
\textsl{The geometry of Teichm\"uller space via geodesic currents.}
Invent. math. 92, 139-162 (1988).


\bibitem{Bonahon96} Bonahon, F.: 
\textsl{Shearing hyperbolic surfaces, bending pleated surfaces and Thurston's symplectic form.}
Ann. Sci. Toulouse Math. (6) 5 (1996).
 

\bibitem{FLP} Fathi, A., Laudenbach, F., Po\'enaru, V.:
\textsl{Travaux de Thurston sur les surfaces.}
 Ast\'erisque 66-67 (1979)


\bibitem{Fock07} Fock, V. V., Goncharov A. B.:
\textsl{Dual Teichm\"uller and lamination spaces.}
in Handbook of Teichm\"uller theory, Vol. I, 647-684, IRMA Lect. Math. Theor. Phys., 11, EMS, Z\"urich, 2007.


\bibitem{Papad13} Papadopoulos, A.:
\textsl{Metric Spaces, Convexity and Nonpositive Curvature.}
Second Edition, European Math. Soc. (2014).

\bibitem{Pap91} Papadopoulos, A.:
\textsl{On Thurston's boundary of Teichm\"uller space and the extension of earthquakes.}
Topology and its Applications \textbf{41}, 147-177 (1991)


\bibitem{PapThe07} Papadopoulos, A., Th\'eret, G.:
\textsl{On Teichm\"uller's metric and Thurston's asymmetric metric on Teichm\"uller space.}
in Handbook of Teichm\"uller theory, Vol. I, 111-204, IRMA Lect. Math. Theor. Phys., 11, EMS, Z\"urich, 2007.


\bibitem{PapThe08} Papadopoulos, A., Th\'eret, G.:
\textsl{Shift coordinates, stretch lines and polyhedral structures on Teichm\"uller space.}
Monatsh. Math. {\bf 153}  (2008),  no. 4, 309-346.


\bibitem{PapThe09} Papadopoulos, A., Th\'eret, G.:
\textsl{Shortening all the simple closed geodesics on hyperbolic surfaces with boundary},
Proc. Amer. Math. Soc. {\bf 138} (2010), 1775-1784. 


\bibitem{Penner92} Penner, R.C., Harer, J.:
\textsl{Combinatorics of Train Tracks.}
 Annals of Math. Studies 125, Princeton University Press (1992)


\bibitem{Theret03} Th\'eret, G.:  
\textsl{On Thurston's Stretch Lines in Teichm\"uller Space},
preprint.


\bibitem{Theret05} Th\'eret, G.:  
\textsl{On the negative convergence of Thurston's Stretch Lines towards the boundary of Teichm\"uller Space},
Annales Academi\ae\ Scientiarum Fennic\ae\ Mathematica
{\bf 32} (2007), 381-408.


\bibitem{Thurston76} Thurston, W.P.:
\textsl{The Geometry and Topology of Three-Manifolds.}
Lecture notes, Princeton University (1976--77)


\bibitem{Thurston86} Thurston, W.P.:
\textsl{Minimal Stretch Maps Between Hyperbolic Surfaces.}
 preprint (1986) arXiv:math.GT/9801039 


\bibitem{Thurston97} Thurston, W.P.: Three-dimensional geometry and
  topology. Vol. 1. Edited by Silvio Levy. Princeton Mathematical
  Series, 35. Princeton University Press, Princeton, NJ, (1997)

\end{thebibliography}
\end{document}